\documentclass[12pt]{amsart}

\usepackage{enumerate, amsmath, amsthm, amsfonts, amssymb, xy,  mathrsfs, graphicx, paralist}
\usepackage[usenames, dvipsnames]{color}
\usepackage[margin=1in]{geometry} 
\usepackage[bookmarks, bookmarksdepth=2, colorlinks=true, linkcolor=blue, citecolor=blue, urlcolor=blue]{hyperref}

\input xy
\xyoption{all}

\numberwithin{equation}{subsection}
\newtheorem{theorem}[equation]{Theorem}

\newtheorem{proposition}[equation]{Proposition}
\newtheorem{lemma}[equation]{Lemma}
\newtheorem{corollary}[equation]{Corollary}

\newtheorem{question}[equation]{Question}

\theoremstyle{definition}
\newtheorem{rmk}[equation]{Remark}
\newenvironment{remark}[1][]{\begin{rmk}[#1] \pushQED{\qed}}{\popQED \end{rmk}}
\newtheorem{eg}[equation]{Example}
\newenvironment{example}[1][]{\begin{eg}[#1] \pushQED{\qed}}{\popQED \end{eg}}
\newtheorem{defn}[equation]{Definition}
\newenvironment{definition}[1][]{\begin{defn}[#1]\pushQED{\qed}}{\popQED \end{defn}}

\newcommand{\cA}{\mathcal{A}}

\newcommand{\cB}{\mathcal{B}}

\newcommand{\bC}{\mathbf{C}}
\newcommand{\cC}{\mathcal{C}}

\newcommand{\bD}{\mathbf{D}}
\newcommand{\cD}{\mathcal{D}}

\newcommand{\rD}{\mathrm{D}}

\newcommand{\bF}{\mathbf{F}}
\newcommand{\cF}{\mathcal{F}}

\newcommand{\bG}{\mathbf{G}}

\newcommand{\rH}{\mathrm{H}}

\newcommand{\bI}{\mathbf{I}}
\newcommand{\cI}{\mathcal{I}}

\newcommand{\bK}{\mathbf{K}}
\newcommand{\cK}{\mathcal{K}}

\newcommand{\rK}{\mathrm{K}}

\newcommand{\bM}{\mathbf{M}}

\newcommand{\cO}{\mathcal{O}}

\newcommand{\bP}{\mathbf{P}}
\newcommand{\cP}{\mathcal{P}}

\newcommand{\bQ}{\mathbf{Q}}
\newcommand{\cQ}{\mathcal{Q}}

\newcommand{\rR}{\mathrm{R}}

\newcommand{\bS}{\mathbf{S}}

\newcommand{\cT}{\mathcal{T}}

\newcommand{\cV}{\mathcal{V}}

\newcommand{\bZ}{\mathbf{Z}}

\newcommand{\rf}{\mathrm{f}}

\newcommand{\fg}{\mathfrak{g}}

\newcommand{\bk}{\mathbf{k}}

\newcommand{\fm}{\mathfrak{m}}

\renewcommand{\phi}{\varphi}
\renewcommand{\emptyset}{\varnothing}

\renewcommand{\tilde}[1]{\widetilde{#1}}
\newcommand{\ol}[1]{\overline{#1}}
\newcommand{\ul}[1]{\underline{#1}}

\newcommand{\arxiv}[1]{\href{http://arxiv.org/abs/#1}{{\tt arXiv:#1}}}

\makeatletter
\def\Ddots{\mathinner{\mkern1mu\raise\p@
\vbox{\kern7\p@\hbox{.}}\mkern2mu
\raise4\p@\hbox{.}\mkern2mu\raise7\p@\hbox{.}\mkern1mu}}
\makeatother

\DeclareMathOperator{\im}{image} 

\DeclareMathOperator{\coker}{coker}
\renewcommand{\hom}{\operatorname{Hom}}

\DeclareMathOperator{\trace}{Tr}

\DeclareMathOperator{\ext}{Ext}

\DeclareMathOperator{\End}{End}

\DeclareMathOperator{\Sym}{Sym}

\DeclareMathOperator{\Aut}{Aut}
\DeclareMathOperator{\depth}{depth}
\DeclareMathOperator{\Tor}{Tor}
\DeclareMathOperator{\pdim}{pdim}

\DeclareMathOperator{\sgn}{sgn}

\DeclareMathOperator{\Mod}{Mod}
\DeclareMathOperator{\Ch}{\mathbf{Ch}}

\newcommand{\GL}{\mathbf{GL}}

\usepackage[boxsize=1em]{ytableau}

\usepackage{stmaryrd}

\let\lbb\llbracket
\let\rbb\rrbracket
\let\wt\widetilde
\newcommand{\BS}{\mathrm{BS}}
\newcommand{\HS}{\mathrm{HS}}
\newcommand{\VS}{\mathrm{VS}}
\newcommand{\Part}{\mathrm{Part}}
\newcommand{\tors}{\mathrm{tors}}
\newcommand{\gfin}{\mathrm{gf}}
\newcommand{\id}{\mathrm{id}}
\newcommand{\op}{\mathrm{op}}
\newcommand{\lw}[1]{{\textstyle \bigwedge^{#1}}}
\renewcommand{\Vec}{\mathrm{Vec}}

\newcommand{\Perf}{\mathrm{Perf}}
\newcommand{\Tors}{\mathrm{Tors}}
\newcommand{\gf}{\mathrm{gf}}

\DeclareMathOperator{\soc}{soc}
\DeclareMathOperator{\Proj}{Proj}
\DeclareMathOperator{\Ind}{Ind}
\DeclareMathOperator{\Rep}{Rep}

\DeclareMathOperator{\Hom}{Hom}
\DeclareMathOperator{\Ext}{Ext}
\DeclareMathOperator{\Equiv}{Equiv}

\newcommand{\ttors}{\tilde{\Mod}\vphantom{\Mod}_A^{\tors}}

\title[GL-equivariant modules over infinite polynomial rings]{GL-equivariant modules over polynomial rings\\in infinitely many variables}
\date{June 13, 2015}

\author{Steven V Sam}
\address{Department of Mathematics, MIT, Cambridge, MA}
\curraddr{Department of Mathematics, University of California, Berkeley, CA}
\email{svs@math.berkeley.edu}
\author{Andrew Snowden}
\address{Department of Mathematics, MIT, Cambridge, MA}
\curraddr{Department of Mathematics, University of Michigan, Ann Arbor, MI}
\email{asnowden@umich.edu}


\begin{document}

\begin{abstract}
Consider the polynomial ring in countably infinitely many variables over a field of characteristic zero, together with its natural action of the infinite general linear group $G$. We study the algebraic and homological properties of finitely generated modules over this ring that are equipped with a compatible  $G$-action.  We define and prove finiteness properties for analogues of Hilbert series, systems of parameters, depth, local cohomology, Koszul duality, and regularity. We also show that this category is built out of a simpler, more combinatorial, quiver category which we describe explicitly.

Our work is motivated by recent papers in the literature which study finiteness properties of infinite polynomial rings equipped with group actions. (For example, the paper by Church, Ellenberg and Farb on the category of FI-modules, which is equivalent to our category.) Along the way, we see several connections with the character polynomials from the representation theory of the symmetric groups. Several examples are given to illustrate that the invariants we introduce are explicit and computable.
\end{abstract}

\maketitle

\setcounter{tocdepth}{1}
\tableofcontents

\section*{Introduction}

This paper concerns the algebraic and homological properties of modules over the twisted commutative algebra (tca) $A = \Sym(\bC\langle 1 \rangle)$.  There are several ways to view the algebra $A$, but the main point of view we take in this paper is that $A$ is the symmetric algebra on the infinite dimensional vector space $\bC^\infty = \bigcup_{n \ge 0} \bC^n$ equipped with the action of the general linear group $\GL_\infty = \bigcup_{n \ge 0} \GL_n$. Modules over this algebra are required to have a compatible polynomial action of $\GL_\infty$, and concepts such as ``finite generation'' are defined relative to this structure.

Our study of $A$-modules is motivated by recent results in the literature.  In \cite{efw} and \cite{sw}, free resolutions over $A$ are studied.  Although the terminology ``twisted commutative algebra'' is not mentioned there, the idea of using Schur functors without evaluating on a vector space is implicit.  In \cite{snowden}, $A$-modules and modules over more general twisted commutative algebras are used to help establish properties of $\Delta$-modules, which are in turn used to study syzygies of Segre embeddings.  In \cite{fimodules}, $A$-modules are studied under the name ``FI-modules'' (see Proposition~\ref{prop:fi-mod} for the equivalence) and many examples are given.  Some other papers of a similar flavor are \cite{draisma}, \cite{draismakuttler}, \cite{hillarmartin}, \cite{hillarsullivant}.  With this paper, we hope to initiate a systematic study of twisted commutative algebras from the point of view of commutative algebra. 

\subsection{Statement of results}

The first difficulty one encounters when studying the category $\Mod_A$ (the category of {\it finitely generated} $A$-modules) is that it has infinite global dimension:  indeed, the Koszul resolution of the residue field $\bC$ is unbounded, since no wedge power of $\bC^{\infty}$ vanishes. (In fact, every non-projective object of $\Mod_A$ has infinite projective dimension.) As a consequence, the Grothendieck group of $\Mod_A$ is not spanned by projective objects.  In the first part of this paper, we study the structure of $A$-modules and establish results that allow one to deal with these difficulties.  We mention a few specific results here:
\begin{enumerate}[1.]
\item \label{item:proj=inj} Projective $A$-modules are also injective (Corollary~\ref{cor:projinj}).
\item \label{item:finiteinjdim} Every finitely generated $A$-module has finite injective dimension, and, in fact, admits a finite length resolution by finitely generated injective $A$-modules (Theorem~\ref{thm:injA}).
\item \label{item:triangledecomp} Every object $M$ of the bounded derived category of finitely generated $A$-modules fits into an exact triangle of the form
\begin{displaymath}
M_t \to M \to M_f \to
\end{displaymath}
where $M_t$ is a finite length complex of finitely generated torsion modules and $M_f$ is a finite length complex of finitely generated projective modules (Theorem~\ref{thm:pt-decomp}).
\item The Grothendieck group of $\Mod_A$ is spanned by the classes of projective and simple modules (Proposition~\ref{prop:KA}).
\end{enumerate}

Our approach to studying the category $\Mod_A$ is to break it up into two pieces:  the category $\Mod_A^{\tors}$ of torsion $A$-modules and the Serre quotient $\Mod_K = \Mod_A / \Mod_A^{\tors}$.  The category $\Mod_K$ can be thought of as modules over the ``generic point'' of $\Proj(A)$.  In \S \ref{sec:modK}, we give basic structure results on $\Mod_K$:  we compute the simple and injective objects and explicitly describe the injective resolutions of simple objects (Theorem~\ref{thm:BGGresolution}).  We also show, somewhat surprisingly, that $\Mod_K$ is equivalent to $\Mod_A^{\tors}$.  In \S \ref{sec:quiver}, we refine these results and describe $\Mod_K$ and $\Mod_A^{\tors}$ as the category of representations of a certain quiver.  This quiver has wild representation type (Remark~\ref{rmk:wildtype}), which means one cannot expect very fine results for the structure of $A$-modules.

In \S\ref{sec:sectionfunctor}, we apply the results from \S \ref{sec:modK} and \S \ref{sec:quiver} to study $\Mod_A$.  We prove the results mentioned above, as well as a few others: for instance, we show that the auto-equivalence group of $\Mod_A$ is trivial and give a complete description of $\rD^b(A)$ involving only the simpler category $\rD^b(K)$.  We also introduce local cohomology and the section functor.  These are adjoints to the natural functors $\Mod_A^{\tors} \to \Mod_A$ and $\Mod_A \to \Mod_K$, and are important tools in establishing the results of this section.

The second part of the paper studies invariants of $A$-modules.  Using the results from the first part, we obtain easy proofs of the following results:
\begin{enumerate}[1.]
\setcounter{enumi}{4}
\item \label{item:hilbertratl} An analogue of the Hilbert series is ``rational'' (Theorem~\ref{thm:enhanced}).
\item The existence of systems of parameters is substituted by the statement that modules are annihilated by ``differential operators'' (Theorem~\ref{thm:diffeq}).
\item \label{item:fglinearstrand} A generalization of the Hilbert syzygy theorem (as rephrased in Theorem~\ref{thm:hilb-syz-orig}): the minimal projective resolution of a finitely generated $A$-module is finitely generated as a comodule over the exterior algebra $\bigwedge{\bC^{\infty}}$ (Theorem~\ref{thm:hilbsyz}). In particular, regularity is finite, i.e., only finitely many linear strands are non-zero (Corollary~\ref{cor:regularity}). This gives a precise sense in which projective resolutions are determined by a finite amount of data.
\item \label{item:depth} There is a well-defined notion of depth which specializes to the usual notion (Theorem~\ref{thm:depth}).
\item We establish analogues of well-known relationships between local cohomology and depth and Hilbert series.
\end{enumerate}

We note that some of the results in this paper, such as results~\ref{item:proj=inj}, \ref{item:finiteinjdim}, \ref{item:triangledecomp}, \ref{item:hilbertratl} mentioned above, remain interesting if we replace $\bC^\infty$ by a finite-dimensional $\bC^n$, while others only contain content in the infinite setting (generally due to the fact that the polynomial ring in finitely many variables has finite global dimension).

\subsection{Duality}

Koszul duality gives an equivalence between the derived category of $A$-modules and the derived category of $\bigwedge{\bC^{\infty}}$-comodules (equipped with a compatible group action).  Result~\ref{item:fglinearstrand} above shows that Koszul duality induces an equivalence of the bounded derived category of finitely generated $A$-modules with the bounded derived category of finitely cogenerated $\bigwedge{\bC^{\infty}}$-comodules.  However, the abelian category of $\bigwedge{\bC^{\infty}}$-comodules is equivalent to that of $A$-modules, via duality and the transpose operation on partitions --- this is a consequence of the $\GL_\infty$ actions, and is not seen in the analogous finite dimensional situation.  Combining the two equivalences, we obtain an autoduality
\begin{displaymath}
\cF \colon \rD^b(A)^{\op} \to \rD^b(A),
\end{displaymath}
which we call the ``Fourier transform.'' The Fourier transform interchanges the perfect and torsion pieces of result~\ref{item:triangledecomp}, i.e., we have $(\cF M)_t=\cF(M_f)$ (Theorem~\ref{thm:perf-tors}).

\subsection{Applications}

The explicit resolutions we construct in $\Mod_K$ allow us to give a conceptual derivation of a formula for character polynomials (see \S\ref{ss:charpoly}), while our theory of local cohomology provides invariants which detect the discrepancy between the character polynomial and the actual character in low degrees. In particular, see Remark~\ref{rmk:stabdeg} which applies explicit local cohomology calculations to improve some bounds given in \cite{fimodules}.  

Our result on enhanced Hilbert series, and its generalization to multivariate tca's, suggests how to define and prove rationality of an enhanced Hilbert series for $\Delta$-modules.  Similarly, our result on Poincar\'e series, and its generalization to multivariate tca's, suggests how to prove a rationality result for Poincar\'e series of $\Delta$-modules.  This affirmatively answers Questions~4 and~7 from \cite{snowden}.

\subsection{Analogy with $\bC[t]$}

As the notation $\Sym(\bC\langle 1 \rangle)$ is meant to suggest, we think of $A$ as being analogous to the graded polynomial ring $\bC[t]$.  We think of $A$-modules as analogous to \emph{nonnegatively} graded $\bC[t]$-modules.  This analogy is not perfect, but is surprisingly good, and serves as something of a guiding principle: many of the results and definitions in this paper have simpler analogues in the setting of $\bC[t]$-modules.  For example, result~\ref{item:proj=inj} above may seem unexpected at first, but is analogous to the fact that in the category of nonnegatively graded $\bC[t]$-modules, $\bC[t]$ is injective.  We point out many other instances of this analogy along the way, and encourage the reader to find more still.

Just as $\bC[t]$ can be generalized to multivariate polynomial rings, so too can $\Sym(\bC\langle 1 \rangle)$ be generalized to multivariate tca's:  these are rings of the form $\Sym(U \otimes \bC^{\infty})$, where $U$ is a finite dimensional vector space, equipped with the obvious $\GL_{\infty}$-action.  (Actually, these are just the polynomial tca's generated in degree 1.)  We have not yet succeeded in generalizing the results from the first part of this paper to the multivariate setting.  Nonetheless, we have proved analogues of results \ref{item:hilbertratl}--\ref{item:depth} listed above in this setting.  The proofs of these results in the general case are significantly different (and longer), and will be treated in \cite{koszul} and \cite{hilbert}.

\subsection{Roadmap}

We hope the results in this paper will appeal to those interested in abelian categories, commutative algebra, and/or combinatorial representation theory. Here we provide a brief roadmap to try to indicate what might be interesting to whom.

We begin each section with a brief overview of the results that it contains. We advise that any reader of this paper look through these overviews to get a first approximation of the results contained in this paper.

The first part of the paper is largely abstract and categorical in nature.  For those interested in these aspects, we highlight Theorem~\ref{thm:facile}, which gives an elegant description of a natural class of abelian categories, its application to $\Mod_K$ in \S \ref{ss:quivermodK} and the description of the category $\rD^b(A)$ given in Theorem~\ref{thm:Dequiv}.  For the reader mainly interested in the second part of the paper, the most important results from the first part are contained in \S \ref{ss:simp-inj} and \S \ref{ss:injres}; see also Proposition~\ref{prop:ind}.  

The second part of the paper contains the content which is more likely to be of interest to the commutative algebraist or combinatorial representation theorist. In particular, the connection with character polynomials is contained in \S\ref{ss:charpoly} and \S\ref{ss:localchar}. We also wish to highlight \S\ref{ss:quivermodK} which shows that a certain simplicial complex related to Pieri's formula is contractible. As for the extension of the basic invariants of commutative algebra, we refer the reader to \S\ref{ss:enhancedhilbert} for Hilbert series and \S\ref{sec:fourier} for basic properties of Koszul duality and finiteness properties for Tor. The results in \S\ref{sec:depth} give analogues of the notions of depth and local cohomology, and should be of interest to both commutative algebraists and combinatorialists.

For explicit calculations, see Remark~\ref{rmk:computer} for some information, as well as \S\ref{ss:exampleEFW} and \S\ref{ss:localchar}.

\subsection{Notation}

Throughout, $\bC$ denotes the complex numbers, though all results work equally well over an arbitrary field of characteristic 0.  We use the symbol $A$ for the twisted commutative algebra (tca) $\Sym(\bC\langle 1 \rangle)$. An introductory treatment of tca's (along with other background material) can be found in \cite{expos}. We will not use the full theory of tca's, and it will often be enough for the reader to treat $A$ as $\Sym(\bC^\infty)$ with a $\GL_\infty$-action. Other notation is defined in the body of the paper.

Unless otherwise stated (notably in \S\ref{ss:sectiondefn}), ``$A$-module'' will always mean ``finitely generated $A$-module.'' The category of finitely generated $A$-modules is denoted $\Mod_A$.

\subsection*{Acknowledgements}

We thank Thomas Church, Jordan Ellenberg, Ian Shipman, David Treumann, and Yan Zhang for helpful correspondence. We also thank an anonymous referee for a very thorough reading of a previous draft and for numerous suggestions which significantly improved the quality of the paper.

Steven Sam was supported by an NDSEG fellowship while this work was done. 

\section{Background} \label{sec:background}

We refer to \cite{expos} for a more thorough treatment of the material in this section.

\subsection{Basic notions}

Given a partition $\lambda$ of size $n$ (denoted $|\lambda| = n$), let $\bS_\lambda$ be the Schur functor indexed by $\lambda$ and let $\bM_\lambda$ be the corresponding irreducible representation of the symmetric group $S_n$. We index them so that if $\lambda = (n)$ has one part, then $\bS_{(n)}$ is the $n$th symmetric power functor, and $\bM_{(n)}$ is the trivial representation of $S_n$. The notation $(a^b)$ means the sequence $(a, \dots, a)$ with $a$ repeated $b$ times. In particular, if $\lambda = (1^n)$, then $\bS_{(1^n)}$ is the $n$th exterior power functor, and $\bM_{(1^n)}$ is the sign representation of $S_n$.

Given an inclusion of partitions $\lambda \subseteq \mu$ (i.e., $\lambda_i \le \mu_i$ for all $i$), we say that $\mu / \lambda$ is a {\bf horizontal strip} of size $|\mu| - |\lambda|$ if $\mu_i \ge \lambda_i \ge \mu_{i+1}$ for all $i$. For notation, we write $\mu / \lambda \in \HS_d$ (implicit in this notation is that $\lambda \subseteq \mu$), where $d=|\mu|-|\lambda|$, and we write $\mu / \lambda \in \HS$ if there is some $d$ for which $\mu / \lambda \in \HS_d$. We recall Pieri's formula, which states that 
\[
\bS_\lambda \otimes \Sym^d = \bS_\lambda \otimes \bS_{(d)} = \bigoplus_{\mu / \lambda \in \HS_d} \bS_\mu.
\]
We define the {\bf transpose partition} $\lambda^\dagger$ by $\lambda^\dagger_i = \#\{j \ge i \mid \lambda_j \ge i\}$. If $\lambda \subseteq \mu$, we say that $\mu / \lambda \in \VS_d$ if and only if $\mu^\dagger / \lambda^\dagger \in \HS_d$, and say that $\mu / \lambda$ is a {\bf vertical strip}. The notation $\VS$ is defined similarly. The dual version of Pieri's formula states that
\[
\bS_\lambda \otimes \bigwedge^d = \bS_\lambda \otimes \bS_{(1^d)} = \bigoplus_{\mu / \lambda \in \VS_d} \bS_\mu.
\]

\subsection{The category $\cV$} \label{sec:catV}

Consider the following three abelian categories:
\begin{itemize}
\item Let $\cV_1$ be the category of polynomial representations of $\GL_\infty$, where a representation of $\GL_\infty$ is polynomial if it appears as a subquotient of an arbitrary direct sum of tensor powers of the standard representation $\bC^{\infty}$. Morphisms are maps of representations. The simple objects in this category are the representations $\bS_{\lambda}(\bC^{\infty})$.
\item Let $\cV_2$ be the category of polynomial endofunctors of $\Vec$, where a functor is polynomial if it appears as a subquotient (in the category of functors $\Vec \to \Vec$) of an arbitrary direct sum of functors of the form $V \mapsto V^{\otimes n}$. Morphisms are natural transformations of functors. The simple objects in this category are the Schur functors $\bS_{\lambda}$.
\item Let $\cV_3$ be the category of sequences $(V_n)_{n \ge 0}$, where $V_n$ is a representation of the symmetric group $S_n$. A morphism $f \colon (V_n) \to (W_n)$ is a sequence $f=(f_n)$ where $f_n$ is a map of representations $V_n \to W_n$. The simple objects in this category are the representations $\bM_{\lambda}$ (placed in degree $\vert \lambda \vert$, with 0 in all other degrees).
\end{itemize}
The three categories are equivalent. One can see this directly from the structure of the three categories: in each, every object is a direct sum of simple objects, and the simple objects are indexed by partitions. Better, one can give natural equivalences between them, as follows:
\begin{itemize}
\item The equivalence $\cV_2 \to \cV_1$ takes a functor $F$ to the representation $F(\bC^{\infty})$.
\item The equivalence $\cV_3 \to \cV_2$ takes a sequence $(V_n)$ to the functor $F$ given by
\begin{displaymath}
F(U)=\bigoplus_{n \ge 0} (U^{\otimes n} \otimes V_n)_{S_n}
\end{displaymath}
where the subscript denotes coinvariants.
\item The equivalence $\cV_1 \to \cV_2$ takes a polynomial representation $U$ to the sequence $(V_n)$ where $V_n$ is the $1^n$ weight space of $U$ (i.e., the subspace of $U$ where a diagonal matrix $[z_1, z_2, \ldots]$ acts by multiplication by $z_1 \cdots z_n$).
\end{itemize}
See \cite[\S 5]{expos} for more details on these equivalences.

Each category $\cV_i$ has a symmetric tensor product. For $\cV_1$, the tensor product is the usual tensor product of representations. For $\cV_2$, it is the pointwise tensor product: $(F \otimes G)(V)=F(V) \otimes G(V)$. In $\cV_3$, it is given as follows. Let $(V_n)$ and $(W_n)$ be two objects of $\cV_3$. Then
\begin{displaymath}
(V \otimes W)_n=\bigoplus_{i+j=n} \Ind_{S_i \times S_j}^{S_n} (V_i \otimes W_j).
\end{displaymath}
The equivalences between the $\cV_i$ respect the tensor product. The structure coefficients of the tensor product of simple objects are computed using the Littlewood--Richardson rule.

Let $\cV_i^{\gf}$ be the subcategory of $\cV_i$ on graded-finite objects, i.e., objects in which each simple appears with finite multiplicity. Each category $\cV_i^{\gf}$ has a duality denoted $(-)^{\vee}$. For $\cV_3^{\gf}$, duality is given by $(V_n)^{\vee}=(V_n^*)$, where $(-)^*$ denotes the usual dual vector space. For $\cV_2^{\gf}$, duality is given by $F^{\vee}(V)=F(V^*)^*$ if $V$ is finite dimensional; in general, $F^{\vee}(V)$ is the direct limit of $F(W^*)^*$ over the finite dimensional subspaces $W$ of $V$. And for $\cV_1^{\gf}$, duality is given by
\begin{displaymath}
V^{\vee}=\Hom_{\GL_\infty}(V, \Sym(\bC^{\infty} \otimes \bC^{\infty})),
\end{displaymath}
where we let $\GL_\infty$ act on the first factor of $\bC^{\infty} \otimes \bC^{\infty}$ to form the $\Hom$ space, and the action of $\GL_\infty$ on the $\Hom$ space comes from its action on the second factor. Note that the linear dual of a polynomial representation is not a polynomial representation, so $V^{\vee}$ is very different from $V^*$. The dualities on $\cV_i^{\gf}$ are compatible with the equivalences between them. The dualities are compatible with tensor products, i.e., the duality functor $\cV_i^{\gf} \to (\cV_i^{\gf})^{\op}$ is a tensor functor. Note that every object of $\cV_i^{\gf}$ is non-canonically isomorphic to its dual, and canonically isomorphic to its double dual. The functor $(-)^\vee$ can be extended to $\cV$, but is no longer a duality.

The category $\cV_3$ admits an operation, which we call {\bf transpose}, and denote by $(-)^{\dag}$. It is given by $(V_n)^{\dag}=(V_n \otimes \sgn)$, where $\sgn$ is the sign representation. Transpose takes the simple object $\bM_{\lambda}$ to $\bM_{\lambda^{\dag}}$. The transpose functor is a tensor functor, but not a symmetric tensor functor; see \cite[\S 7.4]{expos} for details. We transfer this operation to $\cV_1$ and $\cV_2$ via the equivalences; there does not seem to be a nice formula for this operation on these categories, however.

Since the categories $\cV_i$ are canonically identified and come with the same structure, it will at times be convenient to treat them all at once. We therefore let $\cV$ be an abstract category equivalent to any of the $\cV_i$. We treat $\cV$ as an abelian category equipped with a tensor product, transpose functor, and duality (on the graded-finite objects).

For a partition $\lambda$, we write $\ell(\lambda)$ for the number of rows in $\lambda$.  For an object $M$ of $\cV$, we write $\ell(M)$ for the supremum of the $\ell(\lambda)$ over those $\lambda$ for which $\bS_{\lambda}$ is a constituent of $M$.

\subsection{The algebra $A$}

The ring $A=\bigoplus_{d \ge 0} \Sym^d(\bC^{\infty})$ is a commutative algebra object in the category $\cV=\cV_1$. An $A$-module is an object $M$ of $\cV$ with an appropriate multiplication map $A \otimes M \to M$. An $A$-module $M$ is finitely generated if there exists a surjection $A \otimes V \to M \to 0$ where $V$ is a finite length object of $\cV$. Furthermore, $M$ has finite length as an $A$-module if and only if it has finite length as an object of $\cV$.

\begin{definition}
We denote by $\Mod_A$ the category of finitely generated $A$-modules.  
\end{definition}

We will repeatedly use the following fundamental fact \cite[Theorem 2.3]{snowden}:

\begin{theorem}
Every submodule of a finitely generated $A$-module is also finitely generated; in other words, $A$ is noetherian as an algebra object in $\cV$. In particular, $\Mod_A$ is an abelian category.
\end{theorem}

If $V$ is a finite length object of $\cV$ then $A \otimes V$ is a projective object of $\Mod_A$.  An easy argument with Nakayama's lemma shows that all projective objects are of this form.  In particular, the indecomposable projectives of $\Mod_A$ are exactly the modules of the form $A \otimes \bS_{\lambda}$. We emphasize that Pieri's formula implies that we have a multiplicity-free decomposition
\[
A \otimes \bS_\lambda = \bigoplus_{\mu / \lambda \in \HS} \bS_\mu.
\]
This is a fundamental formula, and will be used throughout the rest of the paper without explicit mention. We have the following important result on these modules:

\begin{proposition} \label{prop:pierisubmod}
If $\bS_\mu \subset A \otimes \bS_\lambda$, then the $A$-submodule generated by $\bS_\mu$ contains all $\bS_\nu \subset A \otimes \bS_\lambda$ such that $\mu \subseteq \nu$.
\end{proposition}

\begin{proof}
This can be found in \cite[\S 8]{olver} or \cite[Lemma 2.1]{sw}.
\end{proof}

\begin{remark} \label{rmk:computer}
We wish to emphasize the fact that the maps $\bS_\mu \to A \otimes \bS_\lambda$ given by Pieri's formula can be made concrete. A computer implementation of these maps, based on \cite{olver}, has been written by the first author as a {\tt Macaulay 2} package \cite{pierimaps}. Since all relevant calculations in this paper can be done by replacing $\bC^\infty$ with $\bC^n$ for $n \gg 0$, this means that they can be done on a computer.
\end{remark}

Finally, we prove that the category of $A$-modules is equivalent to the category of FI-modules over $\bC$ \cite[Definition 1.1]{fimodules}. (This remains true over any field of characteristic 0.) Let FI denote the category whose objects are finite sets and whose morphisms are injective maps. Then an FI-module is a functor from FI to the category of $\bC$-vector spaces. This result will not be used in the rest of the paper except in Remark~\ref{rmk:stabdeg}.

\begin{proposition} \label{prop:fi-mod}
The category of ${\rm FI}$-modules is equivalent to the category of $A$-modules.
\end{proposition}

\begin{proof}
Let $B=(B_n)_{n \ge 0}$ be the object of $\cV_3$ given by $B_n=\bC$, with the trivial action of $S_n$, for all $n \ge 0$. We give $B$ the structure of an algebra object of $\cV_3$ in the obvious manner. A $B$-module is an object $M=(M_n)_{n \ge 0}$ of $\cV_3$ equipped with a unital and associative multiplication map $B \otimes M \to M$. Since $B$ is generated in degree 1, such a map is determined by the maps $B_1 \otimes M_n=M_n \to M_{n+1}$ it induces. The associativity condition exactly means that the composite map $M_n \to M_{n+d}$ lands in the $S_d$-invariant space of the target. Thus a $B$-module is the same thing as an FI-module, by \cite[Lemma~2.1]{fimodules}.

Under the equivalence $\cV_1 \to \cV_3$, the algebra $A$ is sent to the algebra $B$. Since this is an equivalence of tensor categories, it induces an equivalence between the categories of $A$-modules and $B$-modules.
\end{proof}

\part{Structure of $A$-modules}


\section{\texorpdfstring{The structure of $\Mod_K$ and $\Mod_A^{\tors}$}{The structure of ModK and ModKtors}} \label{sec:modK}

The main purpose of this section is to define and study a localization functor 
\[
T \colon \Mod_A \to \Mod_K.
\]
The definition of $\Mod_K$ is given in \S\ref{ss:modK}. The category $\Mod_K$ is analyzed and described explicitly. In particular, we classify the simple objects and the injective objects in \S\ref{ss:simp-inj} and construct the minimal injective resolutions of every simple object in \S\ref{ss:injres}. This gives the transition matrices in K-theory between simple and injective objects, and it is shown that the multiplicative structure constants in K-theory are the same in both bases. Finally, in \S\ref{ss:modKmodAequiv}, we show that $\Mod_K$ is equivalent to the category $\Mod_A^{\tors}$ of torsion objects in $\Mod_A$.

\subsection{\texorpdfstring{The category $\Mod_K$}{The category ModK}} \label{ss:modK}

Let $\Mod_K$ be the Serre quotient of $\Mod_A$ by the Serre subcategory $\Mod_A^{\tors}$ of finite length objects. Recall that the objects of $\Mod_K$ are the objects of $\Mod_A$, and that 
\[
\hom_{\Mod_K}(M,N) = \varinjlim \hom_{\Mod_A}(M', N/N')
\]
where the colimit is over all submodules $M' \subseteq M$ and $N' \subseteq N$ such that $M/M'$ and $N'$ have finite length. Hence, if $M$ and $N$ are two $A$-modules then a map $M \to N$ in $\Mod_K$ comes from a map $M' \to N/N'$ in $\Mod_A$, where $M'$ and $N'$ are submodules of $M$ and $N$ such that $M/M'$ and $N'$ have finite length.  Two objects of $\Mod_A$ become isomorphic in $\Mod_K$ if and only if there are maps $M \stackrel{f}{\leftarrow} L \stackrel{g}{\to} N$ in $\Mod_A$ such that the kernel and cokernel of both $f$ and $g$ are finite length. There is a natural functor
\begin{displaymath}
T \colon \Mod_A \to \Mod_K
\end{displaymath}
called the {\bf localization functor}. It is the identity on objects, and takes a morphism in $\Mod_A$ to the morphism it represents in $\Mod_K$. We sometimes use the phrase ``the image of $M$ under $T$'' when we want to regard the object $M$ of $\Mod_A$ as an object of $\Mod_K$.

Note that if $M$ and $N$ are given, then a map $T(M) \to T(N)$ only lifts to a map $M' \to N/N'$, as above. However, if we are given $N$ together with a map $f \colon \ol{M} \to T(N)$, then we can find a lift $M$ of $\ol{M}$ (e.g., a module $M$ with an isomorphism $T(M) \to \ol{M}$) and a map $M \to N$ lifting $f$.

\begin{proposition} \label{prop:extsurj}
Let $M$ and $N$ be objects of $\Mod_A$. Then the natural map 
\[
\varinjlim \Ext^i_A(M', N) \to \Ext^i_K(T(M), T(N))
\]
is surjective for $i=0,1$, where the limit is over the finite colength $A$-submodules $M'$ of $M$.
\end{proposition}

\begin{proof}
The key to this lemma is the existence of truncation functors $(-)_{\ge n}$, as follows. Let $P$ be an object of $\cV$. Define $P_{\ge n}$ to be the sum of the $\bS_{\lambda}$-isotypic pieces of $P$ over partitions $\lambda$ with $\vert \lambda \vert \ge n$. Then $(-)_{\ge n}$ is clearly a functor from $\cV$ to itself. Furthermore, if $P$ is an $A$-module then $P_{\ge n}$ is a submodule of $P$. Note that if $P \in \Mod_A^{\tors}$ then $P_{\ge n}=0$ for $n \gg 0$.

We now treat the $i=0$ case of the lemma. A map $T(M) \to T(N)$ lifts, by definition, to a map $f \colon M' \to N/N'$, where $M' \subset M$ has finite colength and $N' \subset N$ has finite length. For $n \gg 0$, the maps $M'_{\ge n} \to M_{\ge n}$ and $N_{\ge n} \to (N/N')_{\ge n}$ are isomorphisms, so $f_{\ge n}$ can be regarded as a map $M_{\ge n} \to N_{\ge n}$. Of course, we can then compose with the inclusion $N_{\ge n} \to N$, and thus regard $f$ as an element of $\varinjlim \Hom_A(M_{\ge n}, N)$. This achieves the required lifting.

We now treat the $i=1$ case. Let
\begin{displaymath}
0 \to T(N) \to L \stackrel{g}{\to} T(M) \to 0
\end{displaymath}
be an extension in $\Mod_K$. Lift $g$ to a map $\wt{g} \colon \wt{L} \to M$ in $\Mod_A$, and let $K=\ker(\wt{g})$, so that we have an exact sequence
\begin{equation}
\label{eq:extsurj}
0 \to K \to \wt{L} \to M
\end{equation}
Since $T$ is exact, we have an isomorphism $T(K) \to T(N)$. Lift this isomorphism to a map $K_0 \to N/N_0$, where $K_0 \subset K$ has finite colength and $N_0 \subset N$ has finite length. Now, for $n$ sufficiently large, each of the maps
\begin{displaymath}
N_{\ge n} \to (N/N_0)_{\ge n} \leftarrow (K_0)_{\ge n} \to K_{\ge n}
\end{displaymath}
is an isomorphism, and the map $\wt{L}_{\ge n} \to M_{\ge n}$ is surjective. Applying $(-)_{\ge n}$ to \eqref{eq:extsurj} and using these facts, we obtain an extension
\begin{displaymath}
0 \to N_{\ge n} \to \wt{L}_{\ge n} \to M_{\ge n} \to 0
\end{displaymath}
lifting the given extension. Pushing out along the map $N_{\ge n} \to N$, we obtain a class in $\Ext^1_A(M_{\ge n}, N)$ lifting the given class.
\end{proof}

\subsection{Simple and injective objects} \label{ss:simp-inj}

Let $\lambda$ be a partition and let $D \ge \lambda_1$ be an integer.  The module $A \otimes \bS_\lambda$ then has a unique subspace 
\[
L^{\ge D}_{\lambda} = \bigoplus_{d \ge D} \bS_{(d,\lambda)},
\]
which is easily seen to be an $A$-submodule. Proposition~\ref{prop:pierisubmod} implies that $L^{\ge D}_\lambda$ is generated as an $A$-module by $\bS_{(D, \lambda)}$. We will also use $L_\lambda^0$ to denote $L_\lambda^{\ge \lambda_1}$. We define $L_{\lambda}=T(L_{\lambda}^0)$ and $Q_{\lambda}=T(A \otimes \bS_{\lambda})$.

\begin{proposition}
\label{T-simple}
Let $M$ be an $A$-module. Then $T(M)$ is a simple object of $\Mod_K$ if and only if for every $A$-submodule $N$ of $M$, either $N$ or $M/N$ is finite length.
\end{proposition}

\begin{proof}
Suppose that $T(M)$ is simple, and let $N$ be an $A$-submodule of $M$. Then $T(N)$ is either 0 or $M$. In the first case, $N$ has finite length. In the second, $T(M/N)=0$, and so $M/N$ has finite length. We now prove the converse. Let $T(N)$ be a non-zero subobject of $T(M)$. The inclusion $T(N) \to T(M)$ is represented by a map $N' \to M/M'$ in $\Mod_A$ where $N/N'$ and $M'$ have finite length. If the image of $N'$ in $M/M'$ had finite length, then the image of $T(N)$ in $T(M)$ would be zero. Thus this is not the case, and so, by assumption, the cokernel of the map $N' \to M/M'$ has finite length. But this means that the inclusion $T(N) \to T(M)$ is surjective, which proves that $T(M)$ is simple.
\end{proof}

\begin{proposition}
\label{Llambda-sub}
The non-zero $A$-submodules of $L_{\lambda}^{\ge D}$ are the modules $L_{\lambda}^{\ge d}$ with $D \le d$.
\end{proposition}

\begin{proof}
Let $M$ be a non-zero $A$-submodule of $L_{\lambda}^{\ge D}$. Then $M$ contains one of the spaces $\bS_{(d,\lambda)}$ with $d \ge D$; choose $d$ to be minimal. Then obviously $M \subset L_{\lambda}^{\ge d}$. Proposition~\ref{prop:pierisubmod} shows that this containment is an equality.
\end{proof}

\begin{corollary}
The object $L_{\lambda}$ is a simple object of $\Mod_K$. The natural map $T(L_{\lambda}^{\ge D}) \to T(L_{\lambda})$ is an isomorphism, for any $D \ge \lambda_1$.
\end{corollary}

\begin{proof}
For $D'\ge D$, the quotient $L_{\lambda}^{\ge D}/L_{\lambda}^{\ge D'}$ is $\bigoplus_{D \le d <D'} \bS_{(d,\lambda)}$, and therefore of finite length. It follows that the inclusion $L_{\lambda}^{\ge D'} \subset L_{\lambda}^{\ge D}$ is an isomorphism in $\Mod_K$. Simplicity of this object follows from Proposition~\ref{T-simple}.
\end{proof}

\begin{proposition} \label{prop:filtration}
There is a filtration
\[
0 = F_{-1} \subset F_0 \subset \cdots \subset F_{\lambda_1} = A \otimes \bS_\lambda
\] 
where each $F_i$ is an $A$-submodule of $A \otimes \bS_\lambda$, such that for all $i \ge 0$ we have an isomorphism of $A$-modules
\begin{align*}
F_i / F_{i-1} \cong \bigoplus_{\mu,\ \lambda / \mu \in \HS_i} L^{\ge \lambda_1}_\mu.
\end{align*}
\end{proposition}

\begin{proof}
First, the space $A \otimes \bS_\lambda$ is multiplicity-free as a representation of $\GL_\infty$ by Pieri's formula. Let $F_i$ be the subobject (in $\cV$) of $A \otimes \bS_\lambda$ which is the sum of all $\bS_\nu$ such that $|\lambda| - |\nu| + \nu_1 \le i$. From Pieri's formula, $F_i$ is an $A$-submodule for all $i$, and $F_{\lambda_1} = A \otimes \bS_\lambda$, and $F_0 = L_\lambda^0$.

We have an isomorphism $F_i / F_{i-1} \cong \bigoplus_\mu F_\mu$ in $\cV$, where the sum is over all $\mu$ with $\lambda / \mu \in \HS_i$, and where $F_\mu$ is the sum of all $\bS_\nu$ with $|\lambda| - |\nu| + \nu_1 = i$ and $(\nu_2, \nu_3, \dots) = \mu$. It follows from Pieri's formula that each $F_\mu$ is an $A$-submodule of $F_i / F_{i-1}$, and that $F_\mu \cong L_\mu^{\ge \lambda_1}$ as objects of $\cV$. Finally, Proposition~\ref{prop:pierisubmod} implies that $F_\mu \cong L_\mu^{\ge \lambda_1}$ as $A$-modules.
\end{proof}

\begin{corollary}
\label{Qlambda-gr}
The object $Q_\lambda$ has finite length, and its simple constituents are those $L_{\mu}$ such that $\lambda/\mu \in \HS$.
\end{corollary} 

\begin{corollary} \label{cor:modKfinitelength}
Every object of $\Mod_K$ has finite length.
\end{corollary}

\begin{proof}
Every object of $\Mod_A$ is a quotient of an object of the form $A \otimes V$ where $V$ is a finite length object of $\cV$. It follows that every object of $\Mod_K$ is a quotient of $T(A \otimes V)$ for some such $V$, and these have finite length by the previous corollary.
\end{proof}

\begin{corollary} \label{cor:allsimples}
Every simple object of $\Mod_K$ is isomorphic to $L_{\lambda}$ for some $\lambda$.
\end{corollary}

\begin{proof}
As in the previous proof, any simple object of $\Mod_K$ is a quotient (and thus constituent) of an object of the form $T(A \otimes V)$ with $V \in \cV$ finite length. By Corollary~\ref{Qlambda-gr}, all constituents of $T(A \otimes V)$ are of the form $L_{\lambda}$.
\end{proof}

The {\bf socle} of an object $M \in \Mod_K$, denoted by $\soc(M)$, is its largest semisimple submodule. In particular, every non-zero submodule of $M$ has non-zero intersection with $\soc(M)$.

\begin{proposition} \label{prop:Qsoc}
The socle of $Q_{\lambda}$ is $L_{\lambda}$.
\end{proposition}

\begin{proof}
It suffices to show that every non-zero submodule of $Q_{\lambda}$ contains $L_{\lambda}$. Every submodule of $Q_{\lambda}$ is the image under $T$ of a submodule of $A \otimes \bS_{\lambda}$. It follows immediately from Proposition~\ref{prop:pierisubmod} that any non-zero submodule of $A \otimes \bS_{\lambda}$ contains $L^{\ge D}_{\lambda}$, for some $D$, which completes the proof.
\end{proof}

\begin{proposition} \label{prop:ext1Q}
$\Ext^1_K(L_\mu, Q_\lambda) = 0$ for all partitions $\lambda, \mu$.
\end{proposition}

\begin{proof}
Suppose that we have a short exact sequence $0 \to Q_\lambda \to M \to L_\mu \to 0$ in $\Mod_K$. By Proposition~\ref{prop:extsurj}, this lifts to an extension $0 \to A \otimes \bS_{\lambda} \to \wt{M} \to L_{\mu}^{\ge D} \to 0$ in $\Mod_A$, for some $D$; we can assume $D \ge \max(\lambda_1, \mu_1)$. Let $N \subset \tilde{M}$ be a subobject in the category $\cV$ such that $N$ maps isomorphically to $\bS_{(D,\mu)}$ under the map $\tilde{M} \to L_\mu^{\ge D}$. By Proposition~\ref{prop:Qsoc}, either the smallest $A$-submodule of $\tilde{M}$ containing $N$ has non-zero intersection with $L_{\lambda}^0$, or our original sequence splits. So assume the first case happens. In particular, $A \otimes \bS_{(D,\mu)}$ contains $\bS_{(D',\lambda)}$ for some $D'$, so by Pieri's formula, we conclude that $\lambda / \mu$ is a horizontal strip.

Consider the multiplication map $A_{|\lambda|-|\mu|} \otimes \tilde{M}_{(D,\mu)} \to \tilde{M}$ where $\tilde{M}_\nu$ is the $\bS_\nu$-isotypic component of $\tilde{M}$ in the category $\cV$. Since $\tilde{M}_{(D,\mu)}$ contains $\bS_{(D,\mu)}$ with multiplicity $2$, and since $L_{\lambda}^0$ contains $\bS_{(D,\lambda)}$ with multiplicity $1$, there is a copy of $\bS_{(D,\mu)} \subset \tilde{M}$ which does not generate $\bS_{(D,\lambda)}$ under $A$. In particular, it does not generate $\bS_{(D',\lambda)}$ under $A$ for any $D' \ge D$. By Proposition~\ref{prop:Qsoc}, this implies that the $A$-submodule generated by this copy of $\bS_{(D,\mu)}$ has zero intersection with $A \otimes \bS_\lambda$ (to be precise, the image of these two submodules under $T$ have zero intersection, which implies that they intersect in a torsion module; then use that $\tilde{M}$ is torsion-free), so it maps injectively to $L_\mu^{\ge D}$ under the surjection $\tilde{M} \to L_\mu^{\ge D}$ and implies that the original sequence $0 \to Q_\lambda \to M \to L_\mu \to 0$ is split.
\end{proof}

\begin{proposition} \label{prop:modKinj}
$Q_\lambda$ is an injective object in $\Mod_K$.
\end{proposition}

\begin{proof}
Every object of $\Mod_K$ has a finite filtration by objects of the form $L_\mu$ by Corollary~\ref{cor:modKfinitelength} and Corollary~\ref{cor:allsimples}. So by Proposition~\ref{prop:ext1Q}, $\Ext^1_K(M, Q_\lambda) = 0$ for all $M \in \Mod_K$, which implies that $Q_\lambda$ is injective.
\end{proof}

\begin{proposition} \label{HomLQ}
We have $\dim\Hom(L_{\lambda}, Q_{\mu})=\delta_{\lambda,\mu}$.
\end{proposition}

\begin{proof}
First suppose $\lambda=\mu$. The inclusion $L^0_{\lambda} \to A \otimes \bS_{\lambda}$ induces a non-zero map in $\Mod_K$, and so $\Hom(L_{\lambda}, Q_{\lambda}) \ne 0$. That the dimension is at most $1$ follows formally from the fact that $L_{\lambda}$ has multiplicity $1$ in $Q_{\lambda}$.

Now suppose $\lambda \ne \mu$, and we have a non-zero map $L_{\lambda} \to Q_{\mu}$. Since every $A$-submodule of $L^0_{\lambda}$ is of the form $L^{\ge D}_{\lambda}$ and $A \otimes \bS_\mu$ has no finite length submodules, the given map lifts to a non-zero map of $A$-modules $L^{\ge D}_{\lambda} \to A \otimes \bS_{\mu}$, for some $D$. Since any proper quotient of $L^{\ge D}_\lambda$ has finite length, this map must be injective. However, Pieri's rule and Proposition~\ref{prop:pierisubmod} imply that $L_{\lambda}^{\ge D}$ can be a submodule of $A \otimes \bS_{\mu}$ only if $\lambda = \mu$.
\end{proof}

\begin{lemma}
\label{zerohom}
Suppose $M$ is an object of $\Mod_K$ which does not contain $L_{\mu}$ as a constituent. Then $\Hom_K(M, Q_{\mu})=0$.
\end{lemma}

\begin{proof}
Suppose $f \colon M \to Q_{\mu}$ is a non-zero map. Then $f(M)$ is a non-zero subobject of $Q_{\mu}$, and therefore contains $\soc(Q_{\mu})=L_{\mu}$. This implies $L_{\mu}$ is a constituent of $M$, which is a contradiction.
\end{proof}

\begin{proposition}
\label{prop:injhomsets1}
Given partitions $\lambda, \mu$, we have
\[
\hom_K(Q_\lambda, Q_\mu) \cong \begin{cases} \bC & \text{if $\lambda / \mu \in \HS$}\\
0 & \text{otherwise} \end{cases}.
\]
\end{proposition}

\begin{proof}
Suppose $\lambda/\mu \in \HS$. Since $L_{\mu}$ is a constituent of $L_{\lambda}$, we can find a subobject $N$ of $Q_{\lambda}$ such that $L_{\mu}$ is a subobject of $M=Q_{\lambda}/N$. Since $Q_{\lambda}$ is multiplicity-free and $L_{\mu}$ is a constituent of $M$, it follows that $L_{\mu}$ is not a constituent of $N$, and so $\Hom(N, Q_{\mu})=0$ by Lemma~\ref{zerohom}. It follows that the natural map $\Hom(M, Q_{\mu}) \to \Hom(Q_{\lambda}, Q_{\mu})$ is an isomorphism. Any map $M \to Q_{\mu}$ killing $L_{\mu}$ is zero by Lemma~\ref{zerohom}, and so the restriction map $\Hom(M, Q_{\mu}) \to \Hom(L_{\mu}, Q_{\mu})$ is injective, and thus bijective since $Q_{\mu}$ is injective. Any map $L_{\mu} \to Q_{\mu}$ has image contained in $\soc(Q_{\mu})$, and so the natural map $\bC=\Hom(L_{\mu}, L_{\mu}) \to \Hom(L_{\mu}, Q_{\mu})$ is an isomorphism. We have thus shown that $\Hom(Q_{\lambda}, Q_{\mu})$ is one-dimensional.

Now suppose that $\lambda/\mu \not\in \HS$. Since $L_{\mu}$ is not a constituent of $Q_{\lambda}$, we have $\Hom_K(Q_{\lambda}, Q_{\mu})=0$ by Lemma~\ref{zerohom}.
\end{proof}

\begin{corollary} \label{cor:injenv}
The object $Q_{\lambda}$ is indecomposable and is the injective envelope of $L_\lambda$.
\end{corollary}

\begin{proof}
This follows from $\End(Q_{\lambda})=\bC$.
\end{proof}

\begin{corollary} \label{cor:indecinj}
Every indecomposable injective object of $\Mod_K$ is isomorphic to $Q_\lambda$ for some $\lambda$. In particular, every injective object of $\Mod_K$ is a direct sum of $Q_\lambda$'s.
\end{corollary}

\begin{proof}
Follows from Corollary~\ref{cor:allsimples} and Corollary~\ref{cor:injenv}.
\end{proof}

\begin{corollary} \label{cor:A-K-isom}
The natural map $\Hom_A(A \otimes \bS_{\lambda}, A \otimes \bS_{\mu}) \to \Hom_K(Q_{\lambda}, Q_{\mu})$ is an isomorphism.
\end{corollary}

\begin{proposition}
\label{prop:injhomsets2}
Let $f \colon Q_\lambda \to Q_\mu$ be a non-zero map. Then $L_{\nu}$ is a constituent of $\ker(f)$ if and only if it is a constituent of $Q_{\lambda}$ and not $Q_{\mu}$. Similarly, $L_{\eta}$ is a constituent of $\coker(f)$ if and only if it is a constituent of $Q_{\mu}$ and not $Q_{\lambda}$.
\end{proposition}

\begin{proof}
By the above corollary, a non-zero map $Q_\lambda \to Q_\mu$ can be represented by a non-zero $A$-linear map $A \otimes \bS_\lambda \to A \otimes \bS_\mu$. By Proposition~\ref{prop:pierisubmod}, its kernel is the sum of all Schur functors which appear in $A \otimes \bS_\lambda$, but which do not appear in $A \otimes \bS_\mu$.
\end{proof}

\subsection{Injective resolutions and consequences} \label{ss:injres}

We now determine the injective resolutions of simple objects.

\begin{theorem} \label{thm:BGGresolution}
Let $\lambda$ be a partition. Define $\bI^j = \bigoplus_{\mu,\, \lambda/\mu \in \VS_j} Q_\mu$. There are morphisms $\bI^j \to \bI^{j+1}$ so that $L_\lambda \to \bI^\bullet$ is an injective resolution. In particular, the injective dimension of $L_\lambda$ is $\ell(\lambda)$.
\end{theorem}

We start with a definition and two lemmas. We call $(\mu, \mu', \mu'', \nu)$ a square if $\mu / \mu' \in \HS_1$, $\mu / \mu'' \in \HS_1$, $\mu' / \nu \in \HS_1$, $\mu'' / \nu \in \HS_1$, and $\nu / \mu \in \HS_2$.

\begin{lemma} \label{lem:inj-equal}
There exist a choice of maps $f_{\mu, \mu'} \colon Q_\mu \to Q_{\mu'}$ for all $\mu/\mu' \in \HS$ such that $f_{\mu', \nu} f_{\mu, \mu'} = f_{\mu'', \nu} f_{\mu, \mu''}$ for all squares $(\mu, \mu', \mu'', \nu)$.
\end{lemma}

\begin{proof}
By Proposition~\ref{prop:injhomsets1}, $\hom_K(Q_\mu, Q_{\mu'}) \cong \bC$ whenever $\mu/\mu' \in \HS$. Pick arbitrary non-zero maps $\tilde{f}_{\mu, \mu'} \colon Q_\mu \to Q_{\mu'}$ for all $\mu / \mu' \in \HS$. If $\mu' / \nu \in \HS$ and $\mu / \nu \in \HS$, then the composition $\tilde{f}_{\mu', \nu} \tilde{f}_{\mu, \mu'} \colon Q_\mu \to Q_\nu$ is non-zero (this follows from Proposition~\ref{prop:pierisubmod}), so there is a non-zero scalar $\alpha_{\mu, \mu', \nu}$ so that $\tilde{f}_{\mu, \nu} = \alpha_{\mu, \mu', \nu} \tilde{f}_{\mu', \nu} \tilde{f}_{\mu, \mu'}$. This gives a $2$-cocycle in $\rH^2(N(S); \bC^\times)$, where $S$ is the category whose objects are all partitions and where there is a morphism $\mu \to \mu'$ if and only if $\mu / \mu' \in \HS$, and $N(S)$ is the nerve of $S$. We show in Lemma~\ref{lem:partcontract} that $N(S)$ is contractible, so this $2$-cocycle is cohomologous to the identity, which means we can rescale each $\tilde{f}_{\mu,\mu'}$ to some $f_{\mu, \mu'}$ such that $f_{\mu, \nu} = f_{\mu', \nu} f_{\mu, \mu'}$. 
\end{proof}

\begin{lemma} \label{lem:inj-signs}
Let $\cP$ be the set of partitions $\mu$ such that $\lambda / \mu \in \VS$. For every pair of partitions $\mu, \mu' \in \cP$ with $\mu/\mu' \in \HS_1$, there exists a choice of signs $s_{\mu, \mu'}$ so that $s_{\mu', \nu} s_{\mu, \mu'} = - s_{\mu'', \nu} s_{\mu, \mu''}$ for all squares $(\mu, \mu', \mu'', \nu)$ whose elements belong to $\cP$.
\end{lemma}

\begin{proof}
Consider $\cP$ as a poset under the inclusion relation. For a nonnegative integer $n$, let $[n]$ be the totally ordered poset on the set $\{0,1,\dots,n\}$. Let $m_i(\lambda)$ be the number of parts of $\lambda$ that are equal to $i$. Then the poset above is isomorphic to the product $[m_1(\lambda)] \times [m_2(\lambda)] \times \cdots$ (starting with a vertical strip, the element in $[m_i(\lambda)]$ is how many parts of size $i$ had a box removed from them). A covering relation is of the form $(a_1, \dots, a_i, \dots, a_N) \to (a_1, \dots, a_i + 1, \dots, a_N)$. We assign to this covering relation the sign $(-1)^{a_1 + \cdots + a_{i-1}}$. These can be used for the signs $s_{\mu,\mu'}$: in each square of covering relations, there is an odd number of signs which are $-1$. 
\end{proof}

\begin{proof}[Proof of Theorem~\ref{thm:BGGresolution}]
Using the notation from Lemmas~\ref{lem:inj-equal} and \ref{lem:inj-signs}, define the differential $\bI^j \to \bI^{j+1}$ to be $\sum_{\mu, \mu'} s_{\mu, \mu'} f_{\mu, \mu'}$ where the sum is over all pairs $\mu' \subset \mu$ such that $\lambda/\mu \in \VS_j$ and $\lambda/\mu' \in \VS_{j+1}$. Given $Q_\mu \subset \bI^j$ and $Q_\nu \subset \bI^{j+2}$, there are at most two injective summands of $\bI^{j+1}$ which can receive non-zero maps from $Q_\mu$ and which can map to $Q_\nu$ nontrivially. If there is just one, then $\nu / \mu \notin \HS_2$, so the composition $\bI^j \to \bI^{j+2}$ restricted to $Q_\mu \to Q_\nu$ is $0$. Otherwise, we get a square, and the composition $\bI^j \to \bI^{j+2}$ restricted to $Q_\mu \to Q_\nu$ is $0$ by  Lemmas~\ref{lem:inj-equal} and \ref{lem:inj-signs}.

Now we prove acyclicity of the complex. By Proposition~\ref{prop:injhomsets2}, the kernel of $\bI^0 \to \bI^1$ has a composition series consisting of $L_\nu$ such that $L_\nu$ appears in a composition series of $Q_\lambda$ but not in a composition series of any $Q_\mu$ such that $\mu \subset \lambda$ and $|\lambda| = |\mu| + 1$. If $L_\nu$ appears in a composition series of $Q_\lambda$, then $\lambda / \nu \in \HS$ by Corollary~\ref{Qlambda-gr}. If $|\lambda| > |\nu|$, then pick $i$ such that $\lambda_i > \nu_i$ and let $\mu$ be the partition defined by $\mu_j = \lambda_j$ for $j \ne i$ and $\mu_i = \lambda_i - 1$. Then $L_\nu$ appears in a composition series of $Q_\mu \subset \bI^1$. In conclusion, the kernel of $\bI^0 \to \bI^1$ is $L_\lambda$, which proves acyclicity at $\bI^0$.

Now we prove acyclicity at $\bI^j$ for $j>0$. Recall that $T(A \otimes \bS_\alpha) = Q_\alpha$ and that the natural map $\hom_A(A \otimes \bS_\alpha, A \otimes \bS_\beta) \to \hom_K(Q_\alpha, Q_\beta)$ is an isomorphism (Corollary~\ref{cor:A-K-isom}). Hence we can lift $\bI^\bullet$ to a complex of $A$-modules uniquely, so that it makes sense to pick elements. Pick an element $x$ in the kernel of $\bI^j \to \bI^{j+1}$. We wish to show that it is in the image of $\bI^{j-1} \to \bI^j$. Without loss of generality, we may assume that $x$ is the highest weight vector for the generator of a simple $L_\eta$. Say that the simple $L_\eta$ appears in $Q_{\mu^1}, \dots, Q_{\mu^r} \subset \bI^j$. We have a decomposition of $x = x_1 + \cdots + x_r$ into a sum of highest weight vectors coming from these $r$ copies of $L_\eta$. We throw out any $\mu^i$ from the list $\{\mu^1, \dots, \mu^r\}$ for which $x_i = 0$. If $r=1$, then we have $\lambda / \mu^1 \in \VS_j$ and $\mu^1/\nu \in \HS$. Pick $i$ so that $\lambda^\dagger_i > (\mu^1)^\dagger_i$ and define $\eta$ by $\eta^\dagger_j = (\mu^1)^\dagger_j$ for $j \ne i$ and $\eta^\dagger_i = (\mu^1)^\dagger_i + 1$. Then $\lambda / \eta \in \VS_{j-1}$ and $\eta / \nu \in \HS$, so in the differential $\bI^{j-1} \to \bI^j$, the map $f_{\eta,\mu^1} \colon Q_\eta \to Q_{\mu^1}$ is non-zero and $x$ is in its image by Proposition~\ref{prop:injhomsets2}.

So suppose now that $r>1$. If for some $\mu^a$ and $\mu^b$, it is the case that $Q_{\mu^a}$ and $Q_{\mu^b}$ do not map to any common $Q_\nu \subset \bI^{j+1}$, then there is some column of $\lambda$ such that $\lambda / \mu^a$ has at least $2$ more boxes than $\lambda / \mu^b$ (and vice versa). In particular, one cannot remove a horizontal strip from both to get a common partition, so they contain no common simple. In particular, this situation never occurs in our set $\{\mu^1, \dots, \mu^r\}$.

So for each $\mu^a, \mu^b$, we can get from one to the other by removing a box in some column and then adding a box in a different column. So $\rho = \mu^a \cup \mu^b$ (we define $(\alpha \cup \beta)_i = \max(\alpha_i, \beta_i)$) is independent of $a,b$ (for any other $c$, we must have $\mu^c \subset \mu^a \cup \mu^b$ or else either the pair $\mu^a, \mu^c$ or $\mu^b,\mu^c$ would not have the property that we can get one from the other by moving a single box from one column to another column). Furthermore, $\lambda / \rho \in \VS_{j-1}$ and $\rho / \eta \in \HS$, so $Q_\rho \subset \bI^{j-1}$ and $L_\eta$ appears in a composition series of $Q_\rho$. We claim that there is a highest weight vector $y$ in this $L_\eta$ which maps to $x$. For any $a,b$, the quadruple $(\rho, \mu^a, \mu^b, \mu^a \cap \mu^b)$ is a square, so that $f_{\rho, \mu^a} f_{\mu^a, \mu^a \cap \mu^b} = -f_{\rho, \mu^b} f_{\mu^b, \mu^a \cap \mu^b}$. By Proposition~\ref{prop:injhomsets2}, these maps are non-zero on each $L_\nu$, so in particular, there exists a unique $y \in Q_\rho$ such that $f_{\rho, \mu^a}(y) = x_a$ and $f_{\rho, \mu^b}(y) = x_b$. By uniqueness of $y$, and the fact that this is true for all $a,b$, we conclude that $y \mapsto x$ under the differential $\bI^{j-1} \to \bI^j$, so we are done with acyclicity.

The length of the resolution $\bI^\bullet$ is $\ell(\lambda)$, so we have just shown that the injective dimension of $L_\lambda$ is at most $\ell(\lambda)$. Let $\alpha = (\lambda_1 - 1, \dots, \lambda_{\ell(\lambda)} - 1)$. By Proposition~\ref{HomLQ}, we get $\Hom_K(L_\alpha, \bI^j) = 0$ for $j<\ell(\lambda)$ and $\Hom_K(L_\alpha, \bI^{\ell(\lambda)}) \cong \bC$. In particular, $\Ext^{\ell(\lambda)}_K(L_\alpha, L_\lambda) \cong \bC$, so the injective dimension of $L_\lambda$ is at least $\ell(\lambda)$, and hence it is exactly $\ell(\lambda)$.
\end{proof}

The idea for the above proof was loosely modeled on \cite[\S 10]{bgg}.

\begin{corollary} \label{cor:extsimples}
$\Ext^d_K(L_\lambda, L_\mu) \cong \begin{cases} \bC & \textrm{if $\mu / \lambda \in \VS_d$}\\ 0 & \text{otherwise} \end{cases}$
\end{corollary}

\begin{proof}
This follows from Proposition~\ref{HomLQ}.
\end{proof}

\begin{corollary} \label{cor:finiteinj}
Every object in $\Mod_K$ has finite injective dimension.
\end{corollary}

\subsection{\texorpdfstring{$\rK$-theory of $\Mod_K$}{K-theory of ModK}}
\label{ss:modKKtheory}

We now describe the Grothendieck group of $\Mod_K$, and several structures on it. For an abelian category $\cA$, we use $\rK(\cA)$ to denote the free abelian group generated by symbols $[M]$ for each object $M \in \cA$ modulo the relations $[M_2] = [M_1] + [M_3]$ for every short exact sequence $0 \to M_1 \to M_2 \to M_3 \to 0$. We will also use the notation $\rK(\cA)_\bQ = \rK(\cA) \otimes \bQ$. 

There is an action of $\rK(\cV_{\rf})$ on $\rK(\Mod_A)$ via tensor product, and this descends to an action of $\rK(\cV_{\rf})$ on $\rK(\Mod_K)$. 

\begin{proposition} \label{prop:cartanmatrix}
The group $\rK(\Mod_K)$ has for a basis the elements $[L_{\lambda}]$. It has another basis consisting of the elements $[Q_\lambda]$. As a module over $\rK(\cV_{\rf})$, it is free of rank $1$ and spanned by $[Q_0] = [L_0]$.
We have the following change of basis matrices in $\rK(\Mod_K)$:
\begin{align*}
[Q_\lambda] &= \sum_{\mu,\, \lambda / \mu \in \HS} [L_\mu]\\
[L_\mu] &= \sum_{\nu,\, \mu / \nu \in \VS} (-1)^{|\mu|-|\nu|} [Q_\nu].
\end{align*}
\end{proposition}

\begin{proof}
The first formula comes from Proposition~\ref{prop:filtration}. The second formula is the Euler characteristic of the injective resolution of $L_\lambda$ given in Theorem~\ref{thm:BGGresolution}. The fact that the $[L_\lambda]$ form a basis is immediate from Corollary~\ref{cor:modKfinitelength} and Corollary~\ref{cor:allsimples}. The fact that the $[Q_\lambda]$ form a basis follows since we can relate it to the $[L_\lambda]$ basis via an upper unitriangular change of basis matrix given by the formulas above. The freeness over $\rK(\cV_{\rf})$ follows from the identity $[\bS_{\lambda}][Q_0]=[Q_{\lambda}]$.
\end{proof}

\begin{remark} \label{rmk:changeofbasis}
We can prove directly that the change of basis matrices in Proposition~\ref{prop:cartanmatrix} are inverse to one another. Given $\nu \subseteq \lambda$, this is equivalent to showing that
\begin{align} \label{eqn:stripid}
\sum_{\mu,\, \lambda / \mu \in \HS,\, \mu / \nu \in \VS} (-1)^{|\mu|-|\nu|} = \delta_{\lambda, \nu}
\end{align}
where $\delta$ means Kronecker delta. We can prove this identity as follows. Let $h_d(x) = {\rm char}(\Sym^d)$ and $e_d(x) = {\rm char}(\bigwedge^d)$ and define
\begin{align*}
H(t) &= \sum_{d \ge 0} h_d(x) t^d = \prod_{i \ge 1} \frac{1}{1-tx_i}\\
E(t) &= \sum_{d \ge 0} e_d(x) t^d = \prod_{i \ge 1} (1+tx_i).
\end{align*}
Then clearly $H(t)E(-t) = 1$. But the left hand side of \eqref{eqn:stripid} is the coefficient of the Schur function $s_\lambda$ in $s_\nu H(t) E(-t)$ (here we are using both forms of Pieri's formula from \S\ref{sec:background}), so we are done.
\end{remark}

\begin{remark}
\label{rmk:modAmodKsym}
In $\Mod_A$, the simple objects are given by $\bS_\lambda$, while the projective objects are $[A \otimes \bS_\lambda]$, and we have the following change of basis matrices:
\begin{align*}
[A \otimes \bS_\lambda] &= \sum_{\mu,\, \mu / \lambda \in \HS} [\bS_\mu]\\
[\bS_\mu] &= \sum_{\nu,\, \nu / \mu \in \VS} (-1)^{|\nu|-|\mu|} [A \otimes \bS_\nu].
\end{align*}
We point out that since these are infinite sums, they do not literally hold in $\rK(\Mod_A)$. It is curious that the rule is exactly the same as the one in Remark~\ref{rmk:changeofbasis}, except that we {\it add} horizontal/vertical strips rather than {\it remove} them. In fact, the existence of such a symmetry can be deduced from properties of the Fourier transform defined in \S \ref{sec:fourier}.
\end{remark}

\begin{remark}[The pairing on $\rK$-theory]
\label{rmk:modKpairing}
There is a natural pairing
\begin{displaymath}
\langle\, , \rangle \colon \rK(\Mod_K) \otimes \rK(\Mod_K) \to \bZ, \qquad
\langle [M], [N] \rangle=\chi(\Ext^*(M, N)),
\end{displaymath}
where $\chi$ denotes Euler characteristic.  We have
\begin{align*}
\langle Q_{\lambda}, Q_{\mu} \rangle &= \begin{cases}
1 & \textrm{if $\lambda/\mu \in \HS$} \\
0 & \textrm{if not}
\end{cases}
&
\langle L_{\lambda}, Q_{\mu} \rangle &= \delta_{\lambda,\mu}
\\
\langle L_\lambda, L_\mu \rangle &= \begin{cases} (-1)^{|\mu| - |\lambda|} & \textrm{if $\mu / \lambda \in \VS$} \\
0 & \textrm{if not} 
\end{cases}
&
\langle Q_{\lambda}, L_{\mu} \rangle &= \sum_{\nu,\ \lambda/\nu \in \HS,\ \mu/\nu \in \VS} (-1)^{\vert \mu \vert-\vert \nu \vert}
\end{align*}
The top left formula follows from Proposition~\ref{prop:injhomsets1}; the bottom left from Corollary~\ref{cor:extsimples}; the top right is immediate; and the bottom right follows from Theorem~\ref{thm:BGGresolution} and Proposition~\ref{prop:injhomsets1}. Furthermore, $\langle Q_\lambda, L_\mu \rangle \in \{-1,0,1\}$, but we will wait until Remark~\ref{rmk:modKpairing2} to explain this as more combinatorial notions need to be introduced first. Consider the involution $F \colon [Q_{\lambda}] \leftrightarrow (-1)^{\vert \lambda \vert} [L_{\lambda^{\dag}}]$. Then these formulas have the symmetry $\langle M, N \rangle = \langle F(N), F(M) \rangle$. In \S \ref{ss:modKfourier}, we will see that this map on $\rK$-theory is induced by a version of the Fourier transform on $\rD^b(\Mod_K)$.
\end{remark}

\begin{remark}[The ring structure on $\rK$-theory] \label{rem:tensorK}
The tensor product on $\Mod_A$ descends to define a tensor product on $\Mod_K$, which is exact. (Exactness can be proven as follows. It is clear that $\Tor_K^i(Q_{\lambda}, -)=0$. Now use the exact sequence $0 \to L_{\lambda} \to Q_{\lambda} \to Q_{\lambda}/L_{\lambda} \to 0$ and the fact that the constituents of $Q_{\lambda}/L_{\lambda}$ are $L_{\mu}$ with $\mu$ smaller than $\lambda$ to inductively prove $\Tor_K^i(L_{\lambda}, -)=0$ for all $\lambda$.) It follows that $\rK(\Mod_K)$ has the structure of a ring. The map $\phi \colon \rK(\cV_{\rf}) \to \rK(\Mod_K)$ given by $\bS_{\lambda} \mapsto Q_{\lambda}$ is a ring isomorphism.

Since $\phi$ is a ring isomorphism, the tensor product of classes of injective objects in $\rK(\Mod_K)$ is computed using the Littlewood--Richardson rule, that is
\begin{displaymath}
[Q_{\lambda}] [Q_{\mu}] = \sum_{\nu} c_{\lambda,\mu}^{\nu} [Q_{\nu}],
\end{displaymath}
where $c$ denotes the Littlewood--Richardson coefficient.  (In fact, this even holds before passing to $\rK$-theory.)  Somewhat surprisingly, the product of classes of simple objects in $\rK$-theory is also computed with the Littlewood--Richardson rule.  We now show this.  Since $\phi$ is a ring isomorphism, we can calculate with the $[Q_\lambda]$ using Schur functions $s_\lambda$. From Proposition~\ref{prop:cartanmatrix}, we see that $[L_\lambda]$ becomes $\sum_{d \ge 0} (-1)^d s_{1^d}^\perp s_\lambda$ where $s_{1^d}^\perp$ is the skewing operator ($s_\lambda^\perp$ is the adjoint to multiplication by $s_\lambda$ with respect to the inner product for which the Schur functions are orthonormal). In general, we have the identity
\[
s_\eta^\perp(fg) = \sum_{\xi, \theta} c^\eta_{\xi, \theta} s_\xi^\perp(f) s_\theta^\perp(g)
\]
\cite[Example I.5.25(d)]{macdonald}. Hence we get the identity
\[
[L_\lambda][L_\mu] = \Big( \sum_{d \ge 0} (-1)^d s_{1^d}^\perp s_\lambda \Big)\Big( \sum_{d \ge 0} (-1)^d s_{1^d}^\perp s_\mu \Big) = \sum_{d \ge 0} (-1)^d s_{1^d}^\perp(s_\lambda s_\mu) = \sum_\nu c_{\lambda, \mu}^\nu [L_\nu]. \qedhere
\]
\end{remark}

\subsection{\texorpdfstring{The equivalence of $\Mod_K$ with $\Mod_A^{\tors}$}{The equivalence of ModK with ModAtors}}
\label{ss:modKmodAequiv}

We now give an equivalence between $\Mod_K$ and $\Mod_A^{\tors}$.  For clarity, let $V$ and $W$ be two copies of $\bC^{\infty}$ (in particular, we have a canonical identification $V=W$), let $A=\Sym(V)$ and $A'=\Sym(W)$, and let $\Mod_K$ be the usual quotient of $\Mod_A$.  Put $\cK=\Sym(V) \otimes \Sym(V \otimes W)$.  Then $\cK$ is an $A$-module, via the first factor, and an $A'$-comodule, via the maps
\begin{align*}
\Sym(V) \otimes \Sym(V \otimes W)
&\to \Sym(V) \otimes \Sym(V \otimes W) \otimes \Sym(V \otimes W) \\
&\to \Sym(V) \otimes \Sym(V \otimes W) \otimes \Sym(V) \otimes \Sym(W)\\
&\to \Sym(V) \otimes \Sym(V \otimes W) \otimes \Sym(W),
\end{align*}
where the first map uses the comultiplication on $\Sym(V \otimes W)$, the second is the natural map coming from the Segre embedding $\bP(V) \times \bP(W) \subset \bP(V \otimes W)$, and the third uses multiplication on $\Sym(V)$. To check that this is a comodule structure, we use the geometric interpretation of this map. Note that $\Sym(V)$ is the ring of functions on $V^*$ (to avoid technicalities, we note that to check the comodule structure on any given graded piece, we may replace $V$ and $W$ with finite dimensional vector spaces of large enough dimension, so we may use this to take duals). Then the above map corresponds to the map
\begin{align*}
V^* \times (V \otimes W)^* \times W^* &\to V^* \times (V \otimes W)^*\\
(x, y, z) &\mapsto (x, y + (x \otimes z)).
\end{align*}
The fact that we get an $A'$-comodule follows from the identity 
\[
(x, y + (x \otimes (z_1 + z_2))) = (x, y + (x \otimes z_1) + (x \otimes z_2)).
\]

We obtain functors
\begin{displaymath}
\Phi' \colon\Mod_A \to \Mod_{A'}, \qquad M \mapsto \Hom_A(M, \cK)^{\vee}
\end{displaymath}
and
\begin{displaymath}
\Psi' \colon\Mod_{A'} \to \Mod_A, \qquad M \mapsto \Hom_{A'}(M^{\vee}, \cK).
\end{displaymath}
Note that the $\Hom$ in the definition of $\Psi'$ is as $A'$-comodules. In \S\ref{ss:saturated}, we will show that $\cK$ is ``saturated'' as an $A$-module. This implies that we can find a functor $\Phi \colon \Mod_K \to \Mod_{A'}$ such that $\Phi'=\Phi T$. We show in the proof below that $\Phi$ takes values in $\Mod_{A'}^{\tors}$. We let $\Psi \colon \Mod_{A'}^{\tors} \to \Mod_K$ be the restriction of $T \Psi'$ to $\Mod_{A'}^{\tors}$.

\begin{theorem} \label{thm:equivmodKmodA}
The functors $\Phi \colon \Mod_K \to \Mod_{A'}^\tors$ and $\Psi \colon \Mod_{A'}^\tors \to \Mod_K$ are mutually quasi-inverse equivalences of categories.
\end{theorem}

\begin{proof}
We first show that $\Phi$ takes values in $\Mod_{A'}^{\tors}$. The Cauchy formula \cite[(3.13)]{expos} gives
\begin{displaymath}
T(\cK)=\bigoplus_{\lambda} Q_{\lambda} \otimes \bS_{\lambda}(W)
\end{displaymath}
as a $\GL(W)$-equivariant object of $\Mod_K$. Let $M$ be an object of $\Mod_A$. Since $T(M)$ has finite length in $\Mod_K$, it has only finitely many constituents, and so $\Hom(T(M), Q_{\lambda})$ is non-zero for only finitely many $\lambda$, by Lemma~\ref{hom-constituent} below. As $\Phi(M)=\Hom_{\Mod_K}(T(M), T(\cK))$, we see that $\Phi(M)$ has finite length as a representation of $\GL(W)$, and is thus a torsion $A'$-module.

We now compute what $\Phi$ does to simple objects. We have $\Phi(L_{\lambda})=\Hom_{\Mod_K}(L_{\lambda}, T(\cK))$. By the above formula and another application of Lemma~\ref{hom-constituent}, we find $\Phi(L_{\lambda})=\bS_{\lambda}(W)$ as a representation of $\GL(W)$. Since this is simple as a representation of $\GL(W)$, there is only one possible $A'$-module structure on it. We therefore see that $\Phi$ sends the simple object $L_{\lambda}$ to the simple $A'$-module $\bS_{\lambda}(W)$.

We now compute what $\Psi$ does to simple objects. Let $\bS_\lambda(W)$ be a simple module over $\Sym(W)$. Write $(-)^{\vee}$ for the duality on $\cV^{\otimes 2}$ (=category of polynomial representations of $\GL(V) \times \GL(W)$) induced by $(-)^{\vee}$ on each factor. Via this duality map $(-)^\vee$, an $A'$-comodule map $M^\vee \to \cK$ is the same as an $A'$-module map $\cK^\vee \to M$, that is, we have a natural isomorphism of $A$-modules
\begin{displaymath}
\hom_{A'}(M^\vee,\cK) = \hom_{A'}(\cK^\vee,M)
\end{displaymath}
Then $\hom_{A'}(\cK^\vee, \bS_\lambda(W))$ is the $\bS_\lambda(W)$-isotypic component of $(\cK^\vee / \fm_W\cK^\vee)^\vee$ where $\fm_W$ is the maximal ideal of $\Sym(W)^\vee$. Note that the $\bS_\lambda(W)^\vee$-isotypic component of $\cK^\vee$ is $\Sym(V)^\vee \otimes \bS_\lambda(V)^\vee$, so we need to calculate the quotient of $\Sym(V)^\vee \otimes \bS_\lambda(V)^\vee$ by its intersection with the $\Sym(W)$-submodule generated by $\Sym(V)^\vee \otimes \bS_\mu(V)^\vee$ for $|\mu| < |\lambda|$. Pick $\mu \subset \lambda$ with $|\mu| = |\lambda|-1$. Let $W \subset \Sym(W)$ be the degree $1$ part. Then the multiplication of $A'$ on $\cK^{\vee}$ is described by
\[
W^\vee \otimes \Sym^d(V)^\vee \otimes \bS_\mu(V)^\vee \otimes \bS_\mu(W)^\vee \to \Sym^{d-1}(V)^\vee \otimes (V \otimes W)^\vee \otimes \bS_\mu(V)^\vee \otimes \bS_\mu(W)^\vee,
\]
where we use comultiplication on $\Sym^d(V)^\vee$. The image of the multiplication $(V \otimes W)^\vee \otimes \bS_\mu(V)^\vee \otimes \bS_\mu(W)^\vee$ contains $\bS_\lambda(V)^\vee \otimes \bS_\lambda(W)^\vee$ (see for example \cite[\S 4]{dep}). In particular, the $\bS_\mu(W)^\vee$-isotypic component has a non-zero image in the $\bS_\lambda(W)^\vee$-isotypic component under multiplication. Now $T(\Sym(V) \otimes \bS_\mu(V))^\vee$ is generated by $T(L^0_\mu)^\vee$, so the image of $(L^0_\mu)^\vee$ under multiplication by $W^\vee$ is non-zero in the $\bS_\lambda(W)^\vee$-isotypic component. From this, we conclude that the quotient of $\Sym(V)^\vee \otimes \bS_\lambda(V)^\vee$ by the image of multiplication by $W^\vee$ is $(L^0_\lambda)^\vee$. So $\hom_A(M^\vee, \cK) = L^0_\lambda$.

We now show that $\Phi$ and $\Psi$ are mutually quasi-inverse. There are injective natural transformations $\id \to \Phi \Psi$ and $\id \to \Psi \Phi$.  Since $\Psi$ and $\Phi$ are both right exact, to check that these maps are isomorphisms it is enough to do so on simple objects, and this follows from the explanation above.
\end{proof}

\begin{remark}
Both $\Mod_A^{\tors}$ and $\Mod_K$ have tensor structures (see Remark~\ref{rem:tensorK}), but they are not equivalent as tensor categories.  To see this, simply note that no object of $\Mod_A^{\tors}$ is flat, but the injective objects of $\Mod_K$ are flat.
\end{remark}


\section{\texorpdfstring{Quiver descriptions of $\Mod_K$ and $\Mod_A^{\tors}$}{Quiver descriptions of ModK and ModAtors}}
\label{sec:quiver}

As we have seen in \S\ref{ss:modKmodAequiv}, the categories $\Mod_K$ and $\Mod_A^{\tors}$ are equivalent. It turns out that there is a third way to describe these categories which is also useful: as the category of representations of an explicitly described locally finite quiver with relations.

We first abstract the important properties of $\Mod_K$ in the concept of a facile abelian category. The main result of \S\ref{ss:facile} is a complete classification of such categories. The data that controls them is an ordered set together with a cohomology class in its nerve. We also give a general criterion for the nerve to be contractible (so that all cohomological data is trivial). This theory is applied in \S\ref{ss:quivermodK} to describe both $\Mod_K$ and $\Mod_A^{\tors}$ as representations of a combinatorially defined quiver.

\subsection{Facile abelian categories} \label{ss:facile}

In this section, we completely classify a certain class of abelian categories.  This result is purely abstract, and will be applied in the next section to the categories of interest.

We begin by recalling some terminology. Let $\cA$ be an abelian category in which every object has finite length. Let $M$ be an object of $\cA$. A {\bf composition series} of $M$ is a chain of subobjects $0=F^0M \subset F^1 M \subset \cdots \subset F^nM=M$ such that $F^{i+1}M/F^iM$ is a simple object for each $0 \le i < n$. Every object has a composition series, but they are typically far from unique. For a simple object $L$, we define the {\bf multiplicity} of $L$ in $M$ to be the number of indices $i$ such that $F^{i+1}M/F^iM$ is isomorphic to $L$. This is independent of the choice of composition series by the Jordan--H\"older theorem. We say that $M$ is {\bf multiplicity-free} if every simple object has multiplicity 0 or 1 in $M$. We say that a simple object is a {\bf constituent} of $M$ if it has non-zero multiplicity.

Fix an algebraically closed field $\bk$.  

\begin{definition}
\label{defn:facile}
A $\bk$-linear abelian category is {\bf facile} if it satisfies the following conditions:
\begin{enumerate}[(a)]
\item It is essentially small, it has enough injectives, every object has finite length, and all $\Hom$ spaces are finite dimensional over $\bk$.
\item Indecomposable injectives are multiplicity-free.
\item For any non-zero map $f \colon I \to I'$ of indecomposable injective objects, the simple constituents of $\ker(f)$ are precisely the constituents of $I$ which are not constituents of $I'$. \qedhere
\end{enumerate}
\end{definition}

The finiteness conditions in (a) can be relaxed, but the categories of interest to us all satisfy these conditions.  Suppose $\cA$ is facile.  Let $S=S(\cA)$ denote the set of isomorphism classes of simple objects (which could be infinite).  For $i \in S$ write $L_i$ for the corresponding simple and $Q_i$ for the injective envelope of $L_i$.  Recall that the {\bf Cartan matrix} of $\cA$ is the $S \times S$ matrix whose entry at $(i, j)$ is the multiplicity of $L_i$ in $Q_j$.  Since $\cA$ is facile, the entries of the Cartan matrix are all 0 or 1.  We prefer to record this data in a slightly different form.  Define $i \le j$ if $L_i$ is a constituent of $Q_j$, i.e., if the $(i,j)$ entry of the Cartan matrix is 1.  We call $\le$ the {\bf Cartan relation}.  It is reflexive and, as we show below, antisymmetrical (i.e., $x \le y$ and $y \le x$ implies $x=y$), but not necessarily transitive.  It is also downwards finite, i.e., for any $j \in S$ there are only finitely many $i \in S$ for which $i \le j$. We write $i_1 \le \cdots \le i_n$ to mean $i_a \le i_b$ for all $a \le b$.

\begin{lemma}
\label{hom-constituent}
Suppose $M$ is an object of $\cA$ which does not admit $L_i$ as a constituent. Then $\Hom(M, Q_i)=0$.
\end{lemma}

\begin{proof}
The proof is exactly the same as that of Lemma~\ref{zerohom}.
\end{proof}

\begin{lemma}
\label{lem:facileQhom}
$\Hom(Q_j, Q_i) \cong \begin{cases} \bk & \textrm{if $i \le j$} \\ 0 &  \textrm{otherwise} \end{cases}$.
\end{lemma}

\begin{proof}
The proof is exactly the same as that of Proposition~\ref{prop:injhomsets1}.
\end{proof}

\begin{lemma}
The Cartan relation is antisymmetrical.
\end{lemma}

\begin{proof}
Suppose $i \le j$ and $j \le i$.  We then have non-zero maps $f \colon Q_i \to Q_j$ and $g \colon Q_j \to Q_i$.  Since both $Q_i$ and $Q_j$ have $L_i$ and $L_j$ as constituents, neither $\ker(f)$ nor $\ker(g)$ has $L_i$ or $L_j$ as a constituent.  Thus $\im(f)$ has an $L_i$ in it, and so $g(\im(f))$ is non-zero.  We thus see the $gf$, and symmetrically $fg$, are non-zero. Condition (c) of Definition~\ref{defn:facile} implies that any non-zero endomorphism of an indecomposable injective is an isomorphism.  Thus $gf$ and $fg$ are invertible, so $f$ and $g$ are isomorphisms, and $i=j$.
\end{proof}

Let $S$ be a set with a reflexive antisymmetrical downwards finite relation $\le$.  An {\bf $S$-rigidified facile abelian category} is a ($\bk$-linear) facile abelian category $\cA$ equipped with a bijection $S \to S(\cA)$ under which $\le$ corresponds to the Cartan relation.  An equivalence of $S$-rigidified facile abelian categories is required to be compatible with the rigidification.  We define the {\bf nerve} $N(S)$ of $S$ to be the simplicial set whose $n$-simplices are chains $i_0 \le \cdots \le i_n$ in $S$. (See \cite[\S VII.5]{maclane} for some generalities on simplicial sets.)  The following theorem is our main result on facile abelian categories.

\begin{theorem}
\label{thm:facile}
Fix a set $S$ with a reflexive antisymmetric downwards finite relation $\le$.
\begin{enumerate}[\rm (a)]
\item The set of equivalence classes of $S$-rigidified facile abelian categories is canonically in bijection with $\rH^2(N(S), \bk^{\times})$.
\item The auto-equivalence group of an $S$-rigidified facile abelian category is canonically isomorphic to $\rH^1(N(S), \bk^{\times})$.
\item The automorphism group of the identity functor of an $S$-rigidified abelian category is canonically isomorphic to $\rH^0(N(S), \bk^{\times})$.
\end{enumerate}
\end{theorem}

We require a lemma first.  For an abelian category $\cA$, let $\ul{\cI}(\cA)$ (resp.\ $\ul{\cI}_0(\cA)$) denote the full subcategory of $\cA$ on the injective objects (resp.\ indecomposable injective objects).  Write $\Equiv(\cA, \cB)$ for the category of equivalences between two categories $\cA$ and $\cB$.

\begin{lemma}
\label{prop:equiv}
Let $\cA$ and $\cB$ be abelian categories with enough injectives.  Then the natural functor
\begin{displaymath}
\Equiv(\cA, \cB) \to \Equiv(\ul{\cI}(\cA), \ul{\cI}(\cB))
\end{displaymath}
is an equivalence.  If every injective object of $\cA$ and $\cB$ is a finite direct sum of indecomposable injectives then the same is true with $\ul{\cI}_0$ in place of $\ul{\cI}$.
\end{lemma}

\begin{proof}
Let $\cI$ be an additive category.  Define $\ul{\wt{\cA}}(\cI)$ to be the additive category whose objects are morphisms $I=[I^0 \stackrel{d}{\to} I^1]$ in $\cI$ and whose morphisms $I \to J$ are commutative squares. We say that a morphism $I \to J$ in $\ul{\wt{\cA}}(\cI)$ is null-homotopic if the map $I^0 \to J^0$ factors through $d \colon I^0 \to I^1$.  We define $\ul{\cA}(\cI)$ to be the category obtained from $\ul{\wt{\cA}}(\cI)$ by identifying morphisms whose difference is null-homotopic. Now suppose that $\cA$ is an abelian category with enough injectives.  Then an easy argument shows that the functor
\begin{displaymath}
\ul{\cA}(\ul{\cI}(\cA)) \to \cA, \qquad [I^0 \stackrel{d}{\to} I^1] \mapsto \ker(d)
\end{displaymath}
is an equivalence.

Now, if $\cI$ and $\cI'$ are additive categories then any equivalence $\cI \to \cI'$ induces an equivalence $\ul{\cA}(\cI) \to \ul{\cA}(\cI')$. It follows from the above discussion that if $\cA$ and $\cB$ are abelian categories with enough injectives then there is a natural functor
\begin{displaymath}
\Equiv(\ul{\cI}(\cA), \ul{\cI}(\cB)) \to \Equiv(\cA, \cB).
\end{displaymath}
We leave to the reader the easy verification that this functor is naturally a quasi-inverse to the obvious one in the opposite direction. This proves the first claim of the lemma.

Now let $\cC$ be a category enriched over the category of abelian groups. Define $\cC^+$ to be the following category, also enriched over abelian groups. An object is a finite sequence $(A_1, \ldots, A_n)$ of objects of $\cC$. The space of morphisms $(A_1, \ldots, A_n) \to (B_1, \ldots, B_m)$ is $\bigoplus_{i,j} \Hom_{\cC}(A_i, B_j)$. Composition of morphisms is given by matrix multiplication. Then $\cC^+$ is an additive category, and in fact is the universal additive category equipped with an enriched functor from $\cC$. 

Now suppose that $\cC$ is an additive category and every object is a finite sum of indecomposable objects. Let $\cC_0$ be the full subcategory on the indecomposable objects. Then the natural functor $\cC_0^+ \to \cC$ is an equivalence of categories: it is essentially surjective by the assumption on $\cC$, and it is fully faithful since $\Hom$'s in $\cC_0^+$ are computed just like $\Hom$'s of direct sums in $\cC$. If $\cC'$ is a second such category, then using this and functoriality of $(-)^+$ we obtain a functor
\begin{displaymath}
\Equiv(\cC_0, \cC'_0) \to \Equiv(\cC, \cC').
\end{displaymath}
This is easily seen to be quasi-inverse to the natural functor in the opposite direction. Taking $\cC=\ul{\cI}(\cA)$ and $\cC'=\ul{\cI}(\cB)$ yields the second claim of the lemma.
\end{proof}

\begin{proof}[Proof of Theorem~\ref{thm:facile}]
(a) Let $\cA$ be an $S$-rigidified facile abelian category.  For each $i \le j$ in $S$, choose a non-zero map $f_{ji} \colon Q_j \to Q_i$.  If $i \le j \le k$ then $f_{ji} f_{kj} \ne 0$ (since it does not kill the $L_i$ inside of $Q_k$, by condition (c) of Definition~\ref{defn:facile}), and so
\begin{displaymath}
f_{ki}=\alpha_{kji}  f_{ji} f_{kj}
\end{displaymath}
for some $\alpha_{kji} \in \bk^{\times}$, by Lemma~\ref{lem:facileQhom}.  The $\alpha_{kji}$ satisfy the cocycle condition, and thus define a class $[\cA]$ in $\rH^2(N(S), \bk^{\times})$, which we call the {\bf cohomological invariant} of $\cA$.

It is clear that invariants of equivalent categories are equal.  Suppose now that $\cA$ and $\cB$ are $S$-rigidified facile abelian categories with equal invariants.  We use our usual notation for $\cA$ and primed versions for $\cB$.  For $j \ge i$ in $S$, choose non-zero maps $f_{ji} \colon Q_j \to Q_i$ in $\cA$ and $f_{ji}' \colon Q'_j \to Q'_i$ in $\cB$.  Since the cocycles $[\cA]$ and $[\cB]$ are equal, we can scale the $f'_{ji}$ so that $\alpha_{kji}=\alpha'_{kji}$.  It follows that $Q_i \mapsto Q'_i$ and $f_{ji} \mapsto f'_{ji}$ defines an equivalence of categories $\ul{\cI}_0(\cA) \to \ul{\cI}_0(\cB)$. Thus $\cA$ and $\cB$ are equivalent by Lemma~\ref{prop:equiv}.

We have thus shown that there is a natural injection from the set of equivalence classes of $S$-rigidified facile abelian categories into $\rH^2(N(S), \bk^{\times})$.  Surjectivity of this map will follow from Proposition~\ref{prop:quiv}.

(b) Let $\cA$ be an $S$-rigidified facile abelian category and let $F$ be an auto-equivalence of $\cA$.  Since $F$ is required to be compatible with the $S$-rigidification, $F(L_i)$ is isomorphic to $L_i$ for all $i$.  It follows that $F(Q_i)$ is isomorphic to $Q_i$ for each $i$, as $Q_i$ is the injective envelope of $L_i$.  Choose an isomorphism $\phi_i \colon Q_i \to F(Q_i)$.  For $i \le j$, we have
\begin{displaymath}
F(f_{ji})=\alpha_{ji} \phi_i f_{ji} \phi_j^{-1}
\end{displaymath}
for some $\alpha_{ji} \in \bk^{\times}$, where $f_{ij} \colon Q_j \to Q_i$ is a chosen non-zero map.  The $\alpha_{ij}$ satisfy the cocycle condition, and define an element $[F] \in \rH^1(N(S), \bk^{\times})$.

It is clear that $F \mapsto [F]$ is a well-defined group homomorphism $\Aut(\cA) \to \rH^1(N(S), \bk^{\times})$.  It is also clear that the restriction of $F$ to $\cI_0=\ul{\cI}_0(\cA)$ is determined, up to isomorphism, by $[F]$.  It follows from Lemma~\ref{prop:equiv} that $F$ is determined, up to isomorphism, by $[F]$.  Now let $\alpha$ be a 1-cocycle on $N(S)$ with values in $\bk^{\times}$.  Define a functor $F \colon \cI_0 \to \cI_0$ by $F(Q_i)=Q_i$ and $F(f_{ji})=\alpha_{ji} f_{ji}$.  The cocycle condition implies that $F$ is actually a functor (i.e., compatible with composition), and it is clearly an equivalence.  Another application of Lemma~\ref{prop:equiv} shows that $F$ extends canonically to an equivalence of $\cA$, and it is clear that $[F]=\alpha$.  We have thus shown that $F \mapsto [F]$ is a bijection.

(c) Let $\cA$ be an $S$-rigidified facile abelian category and let $\phi$ be an automorphism of the identity functor.  Then $\phi(L_i)$ is given by multiplication by a scalar $\alpha_i$ on $L_i$.  We leave it to the reader to verify that $\phi \mapsto \alpha$ defines an isomorphism $\Aut(\id_{\cA}) \to \rH^0(N(S), \bk^{\times})$.
\end{proof}

We now show how to construct facile abelian categories.  Let $S$ be a set with a relation $\le$ which is reflexive, antisymmetrical and downwards finite and let $\alpha$ be a 2-cocycle representing a class in $\rH^2(N(S), \bk^{\times})$.  Define a quiver-with-relations $\cQ^\alpha$ as follows.  The vertices of $\cQ^\alpha$ are the elements of $S$.  If $i \le j$ are elements of $S$ then there is an arrow $\gamma_{ij} \colon i \to j$ in $\cQ^\alpha$.  Suppose that $i \le j$ and $j \le k$.  If $i \le k$ then we impose the relation $\gamma_{ik}=\alpha_{kji}^{-1} \gamma_{jk} \gamma_{ij}$, while if $i \not \le k$ then we impose the relation $\gamma_{jk} \gamma_{ij}=0$.

A representation $V$ of $\cQ^\alpha$ is a choice of vector space $V_i$ for each vertex $i$ and a linear map $f_{\gamma} \colon V_i \to V_j$ for each arrow $\gamma \colon i \to j$ such that composition of the linear maps are compatible with the relations imposed on $\cQ^\alpha$ whenever they exist. A morphism of representations $\phi \colon V \to V'$ is a choice of linear maps $\phi_i \colon V_i \to V'_i$ such that $f'_{\gamma} \phi_i = \phi_j f_{\gamma}$ for all arrows $\gamma \colon i \to j$. We denote by $\Rep(\cQ^\alpha)$ the category of representations where each $V_i$ is finite dimensional and all but finitely many of the $V_i$ are $0$. It is easy to see that this is an abelian category.

\begin{proposition}
\label{prop:quiv}
The category $\Rep(\cQ^\alpha)$ is naturally an $S$-rigidified facile abelian category with cohomological invariant $\alpha$.
\end{proposition}

\begin{proof}
For $k \in S$, we let $L_k$ be the representation of $\cQ^{\alpha}$ which assigns $\bk$ to the vertex $k$ and 0 to all other vertices.  Then $L_k$ is a simple and $k \mapsto L_k$ defines a bijection $S \to S(\Rep(\cQ^{\alpha}))$.  The injective envelope $Q_k$ of $L_k$ assigns to the vertex $j$ the space $\bk$ if $j \le k$ and 0 otherwise, and to the arrow $\gamma_{ij}$ the map scaling by $\alpha_{kji}$, if $i \le j \le k$, and 0 otherwise (the arrow $\gamma_{ik}$ is assigned the identity map).  From this description, it is clear that the Cartan relation corresponds to $\le$.

To calculate the cohomological invariants, we define $f_{ji} \colon Q_j \to Q_i$ when $j \ge i$ as follows: for a vertex $\ell$ with $\ell \not\le i$, we set $(f_{ji})_\ell = 0$; for a vertex $\ell$ with $\ell \le i$ and $\ell \le j$, we define $(f_{ji})_\ell$ to be multiplication by $\alpha_{ji\ell}$. Using these maps, if $i \le j \le k$, we get $f_{ki} = \alpha_{kji} f_{ji} f_{kj}$, which shows that $\Rep(\cQ^\alpha)$ has cohomological invariant $\alpha$.
\end{proof}

\begin{corollary}
Let $\cA$ be an $S$-rigidified facile abelian category with cohomological invariant $\alpha$.  Then $\cA$ is (non-canonically) equivalent to $\Rep(\cQ^{\alpha})$, with $\cQ^{\alpha}$ as above.
\end{corollary}

\begin{remark}
We can rephrase Theorem~\ref{thm:facile} neatly in the language of homotopy theory as follows. Fix $S$. Let $\cC$ be the 2-groupoid whose objects are $S$-rigidified facile abelian categories, whose 1-morphisms are equivalences of $S$-rigidified facile abelian categories, and whose 2-morphisms are isomorphisms of equivalences. Then we have
\begin{displaymath}
\pi_i(\cC, x)=\rH^{2-i}(N(S), \bk^{\times})
\end{displaymath}
for $0 \le i \le 2$ and for any base point $x$.
\end{remark}

\begin{remark}
The theory of facile abelian categories is reminiscent of the theory of gerbes.  Let $X$ be a topological space and let $\cC'$ be the 2-groupoid of $\bk^{\times}$-gerbes on $X$.  Then one has $\pi_i(\cC', x)=\rH^{2-i}(X, \bk^{\times})$, just as above.  Given a gerbe $\alpha$, one can form the abelian category of $\alpha$-twisted sheaves on $X$, and the subcategory of such sheaves which are constructible with respect to some stratification.  Not all such categories are facile, and not all facile abelian categories come from this construction, but there is overlap.
\end{remark}

\begin{question}
\label{ques1}
Let $\cA$ and $\cB$ be two facile abelian categories.  What are necessary and sufficient conditions on the invariants $(S(\cA), [\cA])$ and $(S(\cB), [\cB])$ so that $\cA$ and $\cB$ are derived equivalent?  How does one recover the group $\Aut(\rD^b(\cA))$ in terms of the invariants $(S(\cA), [\cA])$?
\end{question}

We close this section with the following simple result, which gives a criterion for the nerve to be contractible.

\begin{proposition}
\label{prop:contract}
Let $S$ be a set with a reflexive antisymmetric relation $\le$.  Suppose there is a map $\tau \colon S \to S$ and a point $x_0 \in S$ satisfying the following two conditions:
\begin{enumerate}[\rm (a)]
\item If $x \le y$ then $\tau y \le x$. (Applying this twice, this implies that $\tau x \le \tau y$.)
\item For any $x \in S$ we have $\tau^n(x)=x_0$ for $n \gg 0$.
\end{enumerate}
Then $N(S)$ is contractible.
\end{proposition}

\begin{proof}
Let $\Delta$ be the simplex category (objects are sets $[n]=\{0, \ldots, n\}$, morphisms are weakly order-preserving functions), so that a simplicial set is a contravariant functor from $\Delta$ to the category of sets.  The simplicial set $N(S)$ takes $[n]$ to the set of chains $\lambda_0 \le \cdots \le \lambda_n$.  The interval $I$ is the simplicial set represented by the object $[1]$ of $\Delta$; thus $I([n])$ is the set of partitions of $[n]$ into two intervals of the form $\{0,\dots,i\} \cup \{i+1, \dots, n\}$.

Since $x \le y$ implies $\tau x \le \tau y$, we have a map $\tau \colon N(S) \to N(S)$.  We claim that it is homotopic to the identity map.  To see this, define a map $h \colon N(S) \times I \to N(S)$ as follows.  An $n$-simplex of $N(S) \times I$ is a chain $\lambda_0 \le \cdots \le \lambda_n$ in $S$ together with a partition $[n]=\{0,\ldots,i\} \cup \{i+1,\ldots,n\}$.  We map this to the $n$-simplex of $N(S)$ given by $\tau(\lambda_{i+1}) \le \cdots \le \tau(\lambda_n) \le \lambda_0 \le \cdots \le \lambda_i$.  Restricting $h$ to the vertex 0 of $I$ (which corresponds to partitions $[n]=[n] \cup \emptyset$), we get the identity map, while restricting $h$ to the vertex 1 of $I$ (which corresponds to partitions $[n]=\emptyset \cup [n]$), we get $\tau$.  This proves the claim.

Let $S_n \subset S$ be the set of elements $x$ for which $\tau^n(x)=x_0$.  Then $\tau$ induces a map $N(S_n) \to N(S_n)$, and $h$ induces a homotopy of this map with the identity map.  Since $\tau^n$ collapses $N(S_n)$ to a point, but is also homotopic to the identity map, we see that $N(S_n)$ is contractible.  Finally, $N(S)$ is the filtered colimit of the spaces $N(S_n)$, and is therefore contractible as well. 
\end{proof}

\subsection{\texorpdfstring{Description of $\Mod_K$ and $\Mod_A^{\tors}$}{Description of ModK and ModAtors}} \label{ss:quivermodK}

Let $S$ be the set of partitions.  Define a relation $\le$ on $S$ by $\mu \le \lambda$ if $\lambda/\mu \in \HS$.  Let $\Part_{\HS}$ be the quiver-with-relations $\cQ$ constructed in \S\ref{ss:facile} from $S$, using the trivial cocycle. Representations of this quiver coincide with the notion of a ``hypercomplex'' in the sense of Olver \cite[Definition 7.1]{olver}. The main theorem of this section is the following:  

\begin{theorem}
\label{thm:cat}
The categories $\Mod_K$ and $\Mod_A^{\tors}$ are equivalent to $\Rep(\Part_{\HS})$.  Furthermore, if $\cC$ denotes any of these categories then $\Aut(\cC)=1$ and $\Aut(\id_{\cC})=\bC^{\times}$.
\end{theorem}

\begin{remark}
Heuristically, $\Mod_K$ is the representation category of the stabilizer subgroup $P \subset \GL_{\infty}$ of a generic point $\bC^\infty$. We can interpret this as the category of homogeneous bundles on $\GL_\infty/P$, which is the total space of the line bundle $\cO_{\bP^\infty}(-1)$ on $\bP^\infty$. Then our description of this category as representations of the quiver $\Part_{\HS}$ is analogous to the result of \cite{or}, which describes the category of homogeneous bundles on a Hermitian symmetric variety (for example, projective space) as representations of a quiver.
\end{remark}

\begin{remark}
\label{rmk:modKop}
It is also true that there are no equivalences $\cC^{\op} \to \cC$: indeed, there are no non-zero projective objects in $\cC$ but there are non-zero injective objects.
\end{remark}

\begin{remark} \label{rmk:wildtype}
We note that the quiver $\Part_\HS$ has wild representation type. In particular, it contains a vertex with valency $5$ (for example, the partition $(5,4,3,2,1)$), and describing the moduli space of representations of this subquiver contains as a subproblem the moduli space of 5 points in $\bP^1$, which is known to have dimension 2. General principles then imply that the representation theory of this quiver contains as a subproblem the moduli space of pairs of linear operators, which is known to be ``wild''.
\end{remark}

We begin with some lemmas.

\begin{lemma} \label{lem:partcontract}
The space $N(S)$ is contractible.
\end{lemma}

\begin{proof}
Let $\tau \colon S \to S$ be the map which deletes the first row of a partition (and takes the empty partition to itself).  If $x \le y$ then $\tau y \le x$.  Of course, any partition is mapped to the empty partition by some iterate of $\tau$.  The result now follows from Proposition~\ref{prop:contract}.
\end{proof}

\begin{lemma} \label{lem:partrigid}
The structure $(S, \le)$ has no nontrivial automorphisms.
\end{lemma}

\begin{proof}
The zero partition can be recovered as the unique element $x$ of $S$ such that $y \le x$ implies $x=y$.  Put $S_0=\{0\}$.  The set $S_n$ of partitions with $n$ rows can be characterized as the set of $x \in S$ such that $x \ge y$ for some $y \in S_{n-1}$, but $x \not \in S_{n-1}$.  Furthermore, the length of the first row of $x$ is the maximal length of a chain $x_0<x_1<\cdots<x_r = x$ with $x_0 \in S_{n-1}$ (note: this implies $x_i \le x$ for all $i$ by our conventions); the partition obtained from $x$ by deleting the first row is $x_0$. Also, this partition $x_0$ can be characterized as the minimal object $y \in S_{n-1}$ satisfying $y < x$.  This allows $x$ to be recovered inductively.  We have thus shown that a partition can be recovered by how it fits into the order $\le$, which shows that each partition is fixed by an automorphism of $(S, \le)$.  This proves the lemma.
\end{proof}

\begin{lemma}
\label{prop:proj}
One can choose for each $\mu \le \lambda$ a projector $p_{\mu,\lambda} \colon A \otimes \bS_{\mu} \to \bS_{\lambda}$ (in the category $\cV$) such that for $\nu \le \mu \le \lambda$ the square
\begin{displaymath}
\xymatrix{
A \otimes A \otimes \bS_{\nu} \ar[r]^-{\id \otimes p_{\nu,\mu}} &
A \otimes \bS_{\mu} \ar[d]^-{p_{\mu,\lambda}} \\
A \otimes \bS_{\nu} \ar[r]^-{p_{\nu,\lambda}} \ar[u]^-{c \otimes \id} & \bS_{\lambda} }
\end{displaymath}
commutes, where $c$ denotes the comultiplication map on $A$.
\end{lemma}

\begin{proof}
Begin by choosing arbitrary projectors $p_{\mu,\lambda}$.  Since $A \otimes \bS_{\nu}$ is multiplicity-free, the above diagram commutes up to scalars; that is, we have
\begin{displaymath}
p_{\mu,\lambda}(\id \otimes p_{\nu,\mu})(c \otimes \id)=\alpha_{\lambda,\mu,\nu}p_{\nu,\lambda}
\end{displaymath}
for some scalar $\alpha_{\lambda,\mu,\nu}$, which we claim is non-zero. To see this, apply the duality functor $(-)^{\vee}$.  This reverses the directions of the arrows and replaces comultiplication with multiplication and projectors with injectors. Then the fact that the scalar is non-zero is the content of Proposition~\ref{prop:pierisubmod}.  Hence $\alpha$ defines a 2-cocycle on $N(S)$ with values in $\bC^{\times}$.  Since $N(S)$ is contractible, $\alpha$ is a coboundary, and so we can rescale our projectors so that $\alpha=1$.
\end{proof}

\begin{proof}[Proof of Theorem~\ref{thm:cat}]
We first show that $\Mod_K$ and $\Rep(\Part_{\HS})$ are equivalent.  The results of \S\ref{ss:simp-inj} show that $\Mod_K$ is a facile abelian category.  Furthermore, we know that $L_{\mu}$ appears as a constituent in $Q_{\lambda}$ in $\Mod_K$ if and only if $\mu \le \lambda$, and so $\le$ is the Cartan relation for $\Mod_K$.  The cohomological invariant of $\Mod_K$ vanishes, since $N(S)$ is contractible, and so $\Mod_K$ is equivalent to $\Rep(\Part_{\HS})$.

We now show that $\Mod_A^{\tors}$ and $\Rep(\Part_{\HS})$ are equivalent.  For a partition $\lambda$, put $\cK_{\lambda}=A \otimes \bS_{\lambda}$.  The chosen projector $p_{\mu,\lambda}$ extends uniquely to a map of $A$-comodules $\cK_{\mu} \to \cK_{\lambda}$, and Lemma~\ref{prop:proj} exactly says that these maps respect the relations in $\Part_{\HS}$.  We can therefore think of $\cK$ as a representation of $\Part_{\HS}$ in the category of $A$-comodules.  We obtain a functor
\begin{displaymath}
\Phi \colon \Mod_A^{\tors} \to \Rep(\Part_{\HS}), \qquad M \mapsto \Hom_A(M^{\vee}, \cK)
\end{displaymath}
and a functor
\begin{displaymath}
\Psi \colon \Rep(\Part_{\HS}) \to \Mod_A, \qquad M \mapsto \Hom_A(M, \cK)^{\vee}.
\end{displaymath}
Here $(-)^{\vee}$ denotes the duality functor on $\cV$, which takes $A$ to itself and $A$-modules to $A$-comodules.  It is clear that $\Phi$ takes the simple $A$-module $\bS_{\lambda}$ to the simple representation $L_{\lambda}$ of $\Part_{\HS}$, and that $\Psi$ does the opposite.  There are injective natural transformations $\id \to \Phi \Psi$ and $\id \to \Psi \Phi$ which are easily seen to be isomorphisms on simple objects.  Since $\Phi$ and $\Psi$ are both right exact, these maps are necessarily isomorphisms.

The results on $\Aut(\cC)$ and $\Aut(\id_{\cC})$ follow from the contractibility of $N(\cC)$ (Lemma~\ref{lem:partcontract}) and the fact that $(S, \le)$ has no automorphisms (Lemma~\ref{lem:partrigid}), which implies that any auto-equivalence of $\cC$ induces the identity map on $S$.
\end{proof}

\begin{remark}
The functors $\Phi$ and $\Psi$ in the above proof are actually very concrete.  Let $M$ be a finite length object of $\cV$, and let $M_{\lambda}$ denote the multiplicity space of $\bS_{\lambda}$ in $M$.  By Pieri's formula, giving a map $A \otimes M \to M$ is the same as giving a map $f_{\mu,\lambda} \colon M_{\mu} \to M_{\lambda}$ for each $\mu \le \lambda$.  Thus if $M$ is an $A$-module then we get something that looks like it should be a representation of $\Part_{\HS}$.  In fact, the above proof shows that this thing is in fact a representation of $\Part_{\HS}$ (i.e., the relations are satisfied), and this representation is none other than $\Phi(M)$.  Similar comments hold in the reverse direction.
\end{remark}


\section{\texorpdfstring{The structure of $\Mod_A$}{The structure of ModA}} \label{sec:sectionfunctor}

In the previous two sections, we have studied the structure of $\Mod_A^{\tors}$ and $\Mod_K$.  In this section, we study $\Mod_A$ and especially the manner in which it is built out of these two pieces.
We begin in \S \ref{ss:saturated} by defining a class of modules called {\bf saturated} modules, and show that the modules $L_{\lambda}^0$ (defined in \S\ref{ss:simp-inj}) and $A \otimes \bS_{\lambda}$ are saturated.  In \S \ref{ss:sectiondefn}, we show that the localization functor $T \colon \Mod_A \to \Mod_K$ has a right adjoint, which we denote by $S$ and call the {\bf section functor}.  It is almost formal that such an adjoint exists if we allow non-finitely generated modules, but to show that the section functor takes finitely generated modules to finitely generated modules requires the special results on saturated modules.  As an immediate consequence of this work, we show in \S \ref{ss:modAinjectives} that every object of $\Mod_A$ has finite injective dimension.  In \S \ref{ss:loccoh}, we define a right adjoint $\rH^0_{\fm}$ to the inclusion $\Mod_A^{\tors} \to \Mod_A$ and define the {\bf local cohomology} functors $\rH^i_{\fm}$ as the derived functors of $\rH^0_{\fm}$. We will also use $\rR \Gamma_\fm$ to denote the total derived functor of $\rH^0_\fm$. 

At this point, we have a diagram
\begin{displaymath}
\xymatrix{
\rD^b_{\tors}(A) \ar@<.5ex>[r] & \rD^b(A) \ar@<.5ex>[l]^-{\rR \Gamma_{\fm}} \ar@<.5ex>[r]^-T &
\rD^b(K) \ar@<.5ex>[l]^-{\rR S} }.
\end{displaymath}
In \S \ref{ss:deriveddecomp} we show that whenever one is in a situation like this, one can describe $\rD^b(A)$ as the category of triples $(X, Y, f)$ with $X \in \rD^b(K)$, $Y \in \rD^b_{\tors}(A)$ and $f \in \Hom_{\rD^b(A)}(\rR S(X), Y)$. This is specialized to $\rD^b(A)$ in \S \ref{ss:pt-decomp}. This is not a completely satisfactory description of $\rD^b(A)$, as $f$ makes reference to $\rD^b(A)$.  In \S \ref{ss:DbA} we show that, in our specific situation, the data of $f$ can be defined without any reference to $\rD^b(A)$.  This allows us to describe $\rD^b(A)$ purely in terms of the simpler category $\rD^b(K)$.

Finally, in \S \ref{ss:modArigid} we show that the abelian category $\Mod_A$ is rigid, in a suitable sense, and in \S \ref{ss:modAKtheory} we study the $\rK$-theory of $\Mod_A$.

\subsection{Saturated modules} \label{ss:saturated}

We say that an $A$-module $M$ is {\bf saturated} if $\Ext^i_A(N, M)=0$ for $i=0,1$ whenever $N$ is a torsion $A$-module.  The relevance of this condition will be seen in the following sections.

\begin{proposition} \label{prop:Llamsat}
The module $L_{\lambda}^0$ is saturated.
\end{proposition}

\begin{proof}
The statement for $\Ext^0_A$ translates to $\hom_A(N,L_\lambda^0) = 0$ if $N$ is torsion. But $L_\lambda^0$ is defined as a submodule of $A \otimes \bS_\lambda$ which does not contain any torsion submodules, so the statement is clear.

For the statement for $\Ext^1_A$, it suffices to show that $\Ext^1_A(\bS_{\mu}, L_{\lambda}^0)=0$ for all $\mu$.  The Koszul complex $A \otimes \bS_{\mu} \otimes \lw{\bullet}$ gives a projective resolution of $\bS_{\mu}$.  Applying $\Hom_A(-, L_{\lambda}^0)$, we obtain a complex
\begin{equation}
\label{eq:cx}
\Hom_\cV(\bS_{\mu}, L_{\lambda}^0) \to \Hom_\cV(\bS_{\mu} \otimes \lw{1}, L_{\lambda}^0) \to \Hom_\cV(\bS_{\mu} \otimes \lw{2}, L_{\lambda}^0) \to \cdots
\end{equation}
which calculates $\Ext^{\bullet}_A(\bS_{\mu}, L_{\lambda}^0)$. Notice that $\bS_{\mu} \otimes \lw{i}$ is multiplicity-free and all partitions occurring in it have the same size, while $L_{\lambda}^0$ is also multiplicity-free but no two partitions occurring in it have the same size.  It follows that each $\Hom$ space in \eqref{eq:cx} is 0 or 1 dimensional.  In particular, $\Ext^i_A(\bS_{\mu}, L_{\lambda}^0)$ is 0 or 1 dimensional for all $i$.

Suppose that $\mu$ occurs in $L_{\lambda}^0$, so that $\Hom_\cV(\bS_{\mu}, L_{\lambda}^0) \ne 0$.  Since $\Hom_A(\bS_{\mu}, L_{\lambda}^0)=0$, the first differential in \eqref{eq:cx} is injective.  Since the second term of \eqref{eq:cx} has dimension at most 1, the first differential is an isomorphism, and so  $\Ext^1_A(\bS_{\mu}, L_{\lambda}^0)=0$.

Suppose now that $\mu$ does not occur in $L_{\lambda}^0$.  We claim that the second differential in \eqref{eq:cx} is injective. We may assume that $\Hom_\cV(\bS_\mu \otimes \bigwedge^1, L_\lambda^0) \ne 0$, otherwise there is nothing to prove. Since we assumed that $\mu$ does not occur in $L_\lambda^0$, it must be the case that $\mu = (D,\nu)$ where $D \ge \lambda_1$, $\nu \subset \lambda$, and $|\lambda| - |\nu| = 1$. Then $\bS_{(D,\lambda)} \subset \bS_\mu \otimes \bigwedge^1$ and $\bS_{(D+1,\lambda)} \subset \bS_\mu \otimes \bigwedge^2$. Therefore $\Hom_\cV(\bS_\mu \otimes \bigwedge^2, L_\lambda^0) \cong \bC$, and we just need to show that the second differential in \eqref{eq:cx} is non-zero. An application of Proposition~\ref{prop:pierisubmod} shows that the composition of the top row of
\[
\xymatrix{
A \otimes \bS_{(D+1,\lambda)} \ar[r] \ar[d] & A \otimes \bS_{(D,\lambda)} \ar[r]^-f \ar[d] & L_\lambda^0 \subset A \otimes \bS_\lambda \\
A \otimes \bS_\mu \otimes \bigwedge^2 \ar[r] & A \otimes \bS_\mu \otimes \bigwedge^1 \ar[ur]
}
\]
is non-zero when $f$ is non-zero. To see that the image of the left map is not annihilated by the Koszul differential, we can use an exterior algebra version of Proposition~\ref{prop:pierisubmod} (either apply the transpose functor $\dagger$, or see \cite[Corollary 1.8]{sw}). Here the bottom row is the differential from the Koszul complex and the vertical maps are the inclusions. Hence the second differential in \eqref{eq:cx} is injective.
\end{proof}

\begin{proposition} \label{prop:freesat}
The module $A \otimes \bS_{\lambda}$ is saturated.
\end{proposition}

\begin{proof}
The statement for $\Ext^0$ is clear. It suffices to show that $\Ext^1_A(\bS_{\mu}, A \otimes \bS_{\lambda})=0$. Let $0 \to A \otimes \bS_\lambda \to E \to \bS_\mu \to 0$ be an extension. We may suppose that some subrepresentation of $\bS_1 \otimes \bS_\mu$ appears in $A \otimes \bS_\lambda$ (if not, then no nontrivial extension is possible because $\bS_\mu \subset E$ would be closed under multiplication by $A$). 

First suppose that $\mu / \lambda \notin \HS$. Then we can add a box to some (unique) column of $\mu$ to get a partition $\nu$ with $\bS_\nu \subset A \otimes \bS_\lambda$; by our assumption $\mu / \lambda \notin \HS$, we may add another box in the same column (and more boxes to the left of it if necessary to make the result a partition) to get a partition $\eta$ so that $\eta / \lambda \in \HS$ and $\eta / \mu \notin \HS$. If the extension is non-split, then the smallest $A$-submodule of $E$ containing $\bS_\mu$ must also contain $\bS_\nu$. By Proposition~\ref{prop:pierisubmod}, the $A$-submodule generated by $\bS_\nu$ 
contains $\bS_\eta$. In particular, $\bS_\eta \subset A \otimes \bS_\mu$, but this contradicts that $\eta / \mu \notin \HS$.

Now we suppose that $\mu / \lambda \in \HS$ so that $E$ contains $\bS_\mu$ with multiplicity $2$. Let $\nu$ be a partition such that $\nu / \lambda \in \HS$ and $\bS_\nu$ appears in $\bS_1 \otimes \bS_\mu$. Since $\bS_\nu$ appears with multiplicity $1$ in $A \otimes \bS_\lambda$, there must be a copy of $\bS_\mu$ in $E$, call it $N$, that does not generate $\bS_\nu$ under $A$. If there is only $1$ such $\nu$ with $\nu / \lambda \in \HS$ and $\bS_\nu \subset \bS_1 \otimes \bS_\mu$, then we conclude that $N$ is an $A$-submodule and this shows that $E$ is a trivial extension. Otherwise, let $\eta$ be another partition with the same properties of $\nu$. Define $\theta_i = \max(\eta_i, \nu_i)$. Then $\theta$ is a partition and $\theta / \lambda \in \HS$. The $A$-submodule generated by $\bS_\mu$ contains $\bS_\theta$, and also the $A$-submodule generated by $\bS_\eta$ contains $\bS_\theta$ (here we use Proposition~\ref{prop:pierisubmod}). We have a map $\pi \colon A \otimes \bS_\mu \to E$ by sending $\bS_\mu$ to $N$. By construction, $\ker \pi$ contains $\bS_\nu$ and hence $\bS_\theta$, so it must necessarily also contain $\bS_\eta$. So the image of $\pi$ is just $N$, and again we see that $N$ is an $A$-submodule, which shows that $E$ is a trivial extension.
\end{proof}

\subsection{The section functor} \label{ss:sectiondefn}

Recall that $T$ denotes the natural functor $\Mod_A \to \Mod_K$.

\begin{theorem} \label{thm:Tadjoint}
The functor $T \colon \Mod_A \to \Mod_K$ has a right adjoint.
\end{theorem}

We denote this right adjoint by $S$ and call it the {\bf section functor}.  We need a lemma before proving the theorem.  Let $\tilde{\Mod}_A$ denote the category of all $A$-modules (recall that $\Mod_A$ only contains the finitely generated ones), and let $\tilde{T} \colon \tilde{\Mod}_A \to \tilde{\Mod}_K$ be the quotient by the Serre subcategory $\ttors$ of torsion $A$-modules. We say that an object of $\wt{\Mod}_A$ is saturated if $\Ext^i(N, M)=0$ for $i=0,1$ whenever $N$ is a torsion module of $\wt{\Mod}_A$. We note that if $M$ is a torsion-free object of $\wt{\Mod}_A$ and $T=\varinjlim T_i$ is a torsion object of $\wt{\Mod}_A$, then the natural map $\Ext^1(T, M) \to \varprojlim \Ext^1(T_i, M)$ is injective. It thus suffices to check that $\Ext^1(N, M)=0$ for $N$ of finite length to show that $M$ is saturated, as every torsion object is a direct limit of finite length objects. In particular, if $M$ is an object of $\Mod_A$ then the notion of saturated defined here agrees with the one defined in the previous section.

\begin{lemma} \label{lem:tildeS}
The functor $\tilde{T}$ has a right adjoint $\tilde{S}$. If $M$ is saturated then the natural map $M \to \wt{S}(\wt{T}(M))$ is an isomorphism.
\end{lemma}

\begin{proof}
Let $M$ be a torsion-free object of $\wt{\Mod}_A$. Let $\cC(M)$ denote the category of extensions
$$ 0 \to M \to N \to N/M \to 0 $$
in $\wt{\Mod}_A$, where $N/M$ is torsion. We say that an extension as above is {\bf essential} if every non-zero subobject of $N$ has non-zero intersection with $M$. It is easy to see that the essential extensions are cofinal in $\cC(M)$, i.e., every extension maps to an essential one.

We claim that the cardinality of the multiplicity spaces of simple objects of an essential extension is bounded. To see this, first observe the following:
\begin{itemize}
\item If $L$ is a torsion module and $n$ is a cardinal then an essential extension of $L^{\oplus n}$ by $M$ can exist only if $n \le \dim(\Ext^1(L, M))$.
\item Let $N$ be an extension of $M$, and let $N_n$ be the inverse image in $N$ of the $\fm^n$ torsion in $N/M$. Then $N$ is essential if and only if each $N_n$ is. 
\item Let $\cC_n$ be the category of torsion $A$-modules annihilated by $\fm^n$. Then the collection $\cI_n$ of isomorphism classes of indecomposable objects in $\cC_n$ forms a set. This is easy to see since there are only finitely many simples in $\cC_n$ (up to isomorphism) and extensions only go in one direction.
\end{itemize}
Suppose now that $N$ is an essential extension. Then $N_n/M$ decomposes as $\bigoplus_{I \in \cI_n} I^{\oplus m(I)}$, where $\cI_n$ is a set and the cardinality $m(I)$ is bounded in terms of $M$. It follows that the cardinality of $N_n$ is bounded. Since $N$ is the union of the $N_n$, its cardinality is bounded as well.

It follows that there is a cofinal \emph{set} of objects in $\cC(M)$. Therefore, the colimit $\wt{M}$ of the natural functor $\cC(M) \to \wt{\Mod}_A$ exists, as $\wt{\Mod}_A$ has all small colimits. It is easy to see from the definition of $\wt{M}$ that it is saturated, and the map $M \to \wt{M}$ is injective. Thus every torsion-free object injects into a saturated object. Clearly, every object has a maximal torsion subobject. The two statements of the lemma now follow from \cite[Prop.~III.2.4]{gabriel} and the corollary of \cite[Prop.~III.3.3]{gabriel}.
\end{proof}

\begin{proof}[Proof of Theorem~\ref{thm:Tadjoint}]
Lemma~\ref{lem:tildeS} and Proposition~\ref{prop:Llamsat} show that the natural map $L_{\lambda}^0 \to \tilde{S}(\tilde{T}(L_{\lambda}^0))$ is an isomorphism.  In particular, $\tilde{S}(L_{\lambda})$ is a finitely generated $A$-module.  The functor $\tilde{S}$ is left exact, since it is a right adjoint.  It follows that $\tilde{S}$ carries all finite length objects of $\tilde{\Mod}_K$ to finitely generated objects of $\tilde{\Mod}_A$, i.e., $\tilde{S}$ restricts to a functor $S \colon \Mod_K \to \Mod_A$. For $M \in \Mod_A$ and $N \in \Mod_K$ we have identifications
\begin{displaymath}
\Hom_K(T(M), N) = \Hom_K(\wt{T}(M), N) \cong \Hom_A(M, \wt{S}(N)) = \Hom_A(M, S(N)),
\end{displaymath}
and so $S$ is the right adjoint of $T$.
\end{proof}

We now give some of the basic properties of the section functor $S$.

\begin{proposition}
\label{prop:Sprop}
In the following, $M$ is an object of $\Mod_A$ and $M'$ is an object of $\Mod_K$.
\begin{enumerate}[\rm (a)]
\item The functor $S$ is left exact.
\item The functor $S$ takes injective objects of $\Mod_K$ to injective objects of $\Mod_A$.
\item The adjunction $T(S(M')) \to M'$ is an isomorphism.
\item The module $S(M')$ is saturated.
\item The adjunction $M \to S(T(M))$ is an injection (resp.\ an isomorphism) if and only if $M$ is torsion-free (resp.\ saturated).
\end{enumerate}
\end{proposition}

\begin{proof}
\begin{enumerate}[\rm (a)] 
\item This is true for every right adjoint. 

\item This is true for any functor that has  an exact left adjoint. 

\item This is \cite[Prop.~III.2.3]{gabriel}.  

\item This is \cite[Lem.~III.2.2]{gabriel}.

\item This follows from \cite[Prop.~III.2.3]{gabriel} and its corollary. \qedhere
\end{enumerate}
\end{proof}

\begin{corollary} \label{cor:S-L-free}
We have $S(L_{\lambda})=L_{\lambda}^0$ and $S(Q_{\lambda})=A \otimes \bS_{\lambda}$.
\end{corollary}

\begin{proof}
These follow from Proposition~\ref{prop:Sprop}(e) and the results of \S\ref{ss:saturated}.
\end{proof}

\begin{corollary} \label{cor:projinj}
Every projective object of $\Mod_A$ is also injective.
\end{corollary}

\begin{proof}
We noted preceding Proposition~\ref{prop:pierisubmod} that every projective object of $\Mod_A$ are finite direct sums of objects of the form $A \otimes \bS_\lambda$. Injectivity of $A \otimes \bS_{\lambda}$ follows from Corollary~\ref{cor:S-L-free} and Proposition~\ref{prop:Sprop}(b).
\end{proof}

\begin{corollary}
The functor $S$ takes injective objects of $\Mod_K$ to projective objects of $\Mod_A$.
\end{corollary}

\begin{proof}
This follows since every injective object of $\Mod_K$ is a direct sum of $Q_{\lambda}$'s.
\end{proof}

\begin{remark} \label{rmk:differentS}
It is possible to give a different definition of $S$.  First, the natural map
\begin{displaymath}
\Hom_A(A \otimes \bS_{\lambda}, A \otimes \bS_{\mu}) \to \Hom_K(Q_{\lambda}, Q_{\mu})
\end{displaymath}
is an isomorphism (Corollary~\ref{cor:A-K-isom}).  Since every projective object of $\Mod_A$ is a direct sum of $A \otimes \bS_{\lambda}$'s, and every injective object of $\Mod_K$ is a direct sum of $Q_{\lambda}$'s, we see that
\begin{displaymath}
T \colon \ul{\cP}(\Mod_A) \to \ul{\cI}(\Mod_K)
\end{displaymath}
is an equivalence of categories, where $\ul{\cP}$ (resp.\ $\ul{\cI}$) denotes the full subcategory on projectives (resp.\ injectives).  We can therefore define
\begin{displaymath}
S' \colon \ul{\cI}(\Mod_K) \to \ul{\cP}(\Mod_A)
\end{displaymath}
as the inverse of $T$.  This functor induces an equivalence of categories
\begin{displaymath}
\rR S' \colon \rD^b(\Mod_K) \to \Perf,
\end{displaymath}
where $\Perf$ is the full subcategory of $\rD^b(\Mod_A)$ on complexes quasi-isomorphic to a bounded complex of projectives, and we can define $S'$ on all of $\Mod_K$ as the zeroth cohomology of $\rR S'$.  One can show that $S$ and $S'$ are isomorphic.  The advantage of this construction is that it is immediate and explicit:  there is no question of existence, and it is obvious that $S'(Q_{\lambda})=A \otimes \bS_{\lambda}$.  The disadvantage is that it is not clear that $S'$ is the adjoint of $T$.
\end{remark}

\begin{remark}
To give a feel for the above discussion in a more familiar setting, let us examine the case of $\bC[t]$-modules.  Let $\cA$ denote the category of finitely generated nonnegatively graded $\bC[t]$-modules and let $\cB$ be the quotient of $\cA$ by the subcategory of torsion modules.  Let $T \colon \cA \to \cB$ be the natural functor.  Let $F$ be the module $\bC[t]$, let $F[n]$ denote the graded shifted module and put $\ol{F}=T(F)$.  Then $\cB$ is equivalent to $\Vec$, every object being a direct sum of copies of $\ol{F}$.  (Note that there is an injection $F[n] \to F$ with torsion cokernel, which induces an isomorphism $T(F[n]) \to T(F)=\ol{F}$.)  The right adjoint of $T$ exists, and we call it $S$.  We have $S(\ol{F})=F$.  In particular, this shows that $F$ is an injective object of $\cA$, which can also be verified directly.  As below, we say that an object $M$ of $\cA$ is ``saturated'' if the map $M \to S(T(M))$ is an isomorphism.  Clearly, $M$ is saturated if and only if it is isomorphic to a direct sum of $F$'s. In particular, $F[n]$ is not saturated for $n>0$.
\end{remark}

\subsection{Injective resolutions} \label{ss:modAinjectives}

The main result is the following.

\begin{theorem} \label{thm:injA}
Every object of $\Mod_A$ has finite injective dimension.
\end{theorem}

We first require a lemma. Let $\fm \subset A$ denote the unique homogeneous maximal ideal, i.e., the kernel of the surjection $A \to \bC$. As an object of $\cV$, we have $\fm = \bigoplus_{d \ge 1} \Sym^d$. The $k$th power is described similarly as $\fm^k = \bigoplus_{d \ge k} \Sym^d$. For an $A$-module $M$, we can think of $\fm M$ as the submodule of $M$ obtained by multiplying elements of $M$ by variables $x_1, x_2, \dots$, and we have $\fm^k M = \fm(\fm^{k-1} M)$ in general.

\begin{lemma} \label{lem:torsinj}
An injective object of $\Mod_A^{\tors}$ remains injective in $\Mod_A$.
\end{lemma}

\begin{proof}
First, let $N$ and $I$ be objects of $\Mod_A$ such that the maximal size of a partition in $I$ is less than the minimal size of a partition in $N$.  Then $\Ext^1_A(N, I)=0$.  Indeed, let
\begin{displaymath}
0 \to I \to M \to N \to 0
\end{displaymath}
be an extension.  The hypotheses imply that this extension splits uniquely in $\cV$, and it is easy to see that this splitting is a map of $A$-modules.

We now prove the lemma.  Let $I$ be an injective of $\Mod_A^{\tors}$, let $n$ be the maximal size of a partition in $I$ and let $N$ be an arbitrary object of $\Mod_A$.  We have an exact sequence
\begin{displaymath}
0 \to \fm^{n+1} N \to N \to N/\fm^{n+1} N \to 0
\end{displaymath}
and therefore an exact sequence
\begin{displaymath}
\Ext^1_A(N/\fm^{n+1} N, I) \to \Ext^1_A(N, I) \to \Ext^1_A(\fm^{n+1} N, I).
\end{displaymath}
The leftmost group vanishes since $I$ is injective in $\Mod_A^{\tors}$ and any extension of torsion modules is again torsion.  The rightmost group vanishes by the previous paragraph.  Thus the middle group vanishes, and so $I$ is injective.
\end{proof}

\begin{proof}[Proof of Theorem~\ref{thm:injA}]
From the equivalence $\Mod_A^{\tors}=\Mod_K$ and our results on $\Mod_K$, we see that $\Mod_A^{\tors}$ has enough injectives and every object has finite injective dimension.  (This can be seen much more directly, in fact.)  The lemma shows that an injective resolution in $\Mod_A^{\tors}$ remains an injective resolution in $\Mod_A$, and so all torsion $A$-modules have finite injective dimension in $\Mod_A$.

Now let $M$ be an arbitrary $A$-module.  Since every object of $\Mod_K$ has finite injective dimension, we can choose a finite injective resolution $T(M) \to \bI^{\bullet}$ in $\Mod_K$.  Applying $S$, we obtain $S(T(M)) \to S(\bI^{\bullet})$ with $S(\bI^{\bullet})$ a complex of injectives.  Since $TS=\id$, the cohomology of this complex is torsion, and thus has finite injective dimension.  It follows that $S(T(M))$ has finite injective dimension.  Since the kernel and cokernel of $M \to S(T(M))$ are torsion, and thus of finite injective dimension, it follows that $M$ has finite injective dimension.
\end{proof}

Let $I_{\lambda}$ be the injective envelope of $\bS_{\lambda}$ in $\Mod_A^{\tors}$. 

\begin{remark} \label{rmk:torsion-inj}
The modules $I_\lambda$ have a simple description. First, consider the embedding $\bC^\infty \subset \bC \oplus \bC^\infty$. The subgroup of $G=\GL(\bC \oplus \bC^\infty)$ that fixes $\bC$ is a semidirect product $G' = \GL_\infty \ltimes \bC^\infty$ and it is clear that $A$-modules are the same as polynomial $G'$-modules: picking a basis $x_1, x_2, \dots$ for $\bC^\infty$, a module over the additive group $\bC^\infty$ is the choice of commuting operators $X_i$, one for each $x_i$, and hence is a module over $\bC[x_1, x_2, \dots]$; the semidirect product translates to this module being a representation of $\GL_\infty$ which is compatible with the action of $\bC[x_1, x_2, \dots]$.

It follows from \cite[Proposition 5.2.5]{infrank} that the restriction of $\bS_\lambda(\bC \oplus \bC^\infty)$ from $G$ to $G'$ is injective and indecomposable in the category of torsion objects and its socle is $\bS_\lambda(\bC^\infty)$, so $I_\lambda = {\rm Res}^G_{G'}(\bS_\lambda(\bC \oplus \bC^\infty))$. Furthermore, if we set $\mu = (\lambda_2, \lambda_3, \lambda_4, \dots)$ and $\nu = (\lambda_1 + 1, \lambda_3, \lambda_4, \dots)$, then we have a presentation
\[
A \otimes \bS_\nu \to A \otimes \bS_\mu \to I_\lambda \to 0,
\]
and the $I_\lambda$ are exactly the $A$-modules that are resolved by the EFW complexes discussed in \S\ref{ss:exampleEFW} if we take $\alpha = \mu$ and $e = \lambda_1 + 1 - \lambda_2$. 
\end{remark}

\begin{theorem} \label{thm:Ainj}
Every injective object of $\Mod_A$ is a direct sum of modules of the form $I_{\lambda}$ and $A \otimes \bS_{\lambda}$. 
\end{theorem}

\begin{proof}
Let $I$ be an injective object of $\Mod_A$.  It is clear that its maximal torsion subobject $I_{\tors}$ is an injective object of $\Mod_A^{\tors}$, and therefore a sum of $I_{\lambda}$'s.  Since $I_{\tors}$ remains injective in $\Mod_A$ by Lemma~\ref{lem:torsinj}, we have $I=I_{\tors} \oplus I'$ with $I'$ a torsion-free injective.

It thus suffices to show that a torsion-free injective $I$ is a sum of modules of the form $A \otimes \bS_{\lambda}$. Note that $I$ is saturated, by definition. For any saturated module $M$ we have $\Hom_A(N, M)=\Hom_K(T(N), T(M))$, which shows that $T$ takes saturated injective objects of $\Mod_A$ to injective objects of $\Mod_K$.  Thus $T(I)$ is injective in $\Mod_K$, and therefore a sum of $Q_{\lambda}$'s.  Finally, we see that $I=S(T(I))$ is a direct sum of modules of the form $S(Q_{\lambda})=A \otimes \bS_{\lambda}$.
\end{proof}

\subsection{Local cohomology}
\label{ss:loccoh}

For an $A$-module $M$, define $\rH^0_{\fm}(M)$ to be the maximal torsion submodule of $M$.  The functor $\rH^0_{\fm} \colon \Mod_A \to \Mod_A^{\tors}$ is the right adjoint to the inclusion $\Mod_A^{\tors} \to \Mod_A$.  As $\rH^0_{\fm}$ is left exact, its right derived functors $\rH^i_{\fm}$ exist.  We call $\rH_{\fm}^{\bullet}$ the {\bf local cohomology} functors. Set $\rD^b_{\tors}(A) = \rD^b(\Mod_A^{\tors})$. We write $\rR\Gamma_{\fm} \colon \rD^b(A) \to \rD^b_{\tors}(A)$ for the derived functor of $\rH^0_{\fm}$ at the level of derived categories.

\begin{proposition}
Let $M$ be in $\Mod_A$.  Then $\rH^i_{\fm}(M)$ is a finitely generated torsion $A$-module for all $i$, and is zero for $i \gg 0$.
\end{proposition}

\begin{proof}
Since $M$ has a finite length resolution by finitely generated injectives we see that $\rH^i_{\fm}(M)$ is finitely generated for all $i$ and zero for $i \gg 0$.  Since $\rH^0_{\fm}$ takes values in $\Mod_A^{\tors}$ so do its derived functors.
\end{proof}

\begin{proposition}
\label{prop:locS}
Let $M$ be an object of $\rD^b(A)$.  Then we have a distinguished triangle
\begin{displaymath}
\rR\Gamma_{\fm}(M) \to M \to \rR S(T(M)) \to
\end{displaymath}
in $\rD^b(A)$, where the first two maps are the (co)units of the adjunctions.
\end{proposition}

\begin{proof}
It suffices to treat the case where $M$ is a finite length complex of injectives.  We have an exact sequence of complexes
\begin{displaymath}
0 \to \rH^0_{\fm}(M) \to M \to M' \to 0.
\end{displaymath}
The leftmost complex is nothing other than $\rR\Gamma_{\fm}(M)$. The kernel of the natural map $M \to S(T(M))$ contains $\rH^0_\fm(M)$ since it is torsion and the induced map $M' \to S(T(M))$ is an isomorphism by the structure of injective objects of $\Mod_A$ (Theorem~\ref{thm:Ainj}). Thus $M'$ is identified with $\rR S(T(M))$.
\end{proof}

\begin{corollary} \label{prop:localsheaf}
Let $M$ be in $\Mod_A$ and let $M'=T(M)$ be its image in $\Mod_K$.  We have an exact sequence
\begin{displaymath}
0 \to \rH^0_{\fm}(M) \to M \to S(M') \to \rH^1_{\fm}(M) \to 0,
\end{displaymath}
and for each $i>1$, an isomorphism
\begin{displaymath}
\rH^{i+1}_{\fm}(M)=\rR^iS(M').
\end{displaymath}
\end{corollary}

\begin{proof}
This is the long exact sequence on cohomology that comes from the distinguished triangle in Proposition~\ref{prop:locS}.
\end{proof}

This statement should be compared with \cite[Proposition A1.11]{syzygies} which provides a comparison between local cohomology for local rings and sheaf cohomology on the punctured spectrum.

\subsection{Torsion theories and decompositions of derived categories}
\label{ss:deriveddecomp}

In this section, we give a general result on decomposing derived categories which we apply in the next section to $\rD^b(A)$.  Let $\cA$ be an abelian category.  A {\bf hereditary torsion theory} is a pair $(\cT, \cF)$ of strictly full subcategories, whose objects are called torsion and torsion-free respectively, satisfying the following axioms:
\begin{itemize}
\item We have $\Hom(T, F)=0$ for $T \in \cT$ and $F \in \cF$.
\item For any $M \in \cA$ there is a short exact sequence
\begin{displaymath}
0 \to M_t \to M \to M_f \to 0
\end{displaymath}
with $M_t \in \cT$ and $M_f \in \cF$.
\item Any subobject of an object in $\cT$ belongs to $\cT$.
\end{itemize}
In this situation, $\cT$ is a Serre subcategory of $\cA$.  Furthermore, the assignment $M \mapsto M_t$ is the adjoint to the right inclusion $\cT \to \cA$, and therefore functorial. See \cite[\S 1.12]{borceux} for a detailed treatment of torsion theories.

Fix a torsion theory $(\cT, \cF)$ on $\cA$.  We assume the following two conditions hold:
\begin{itemize}
\item Every object of $\cA$ has finite injective dimension.
\item The inclusion $\cT \to \cA$ takes injective objects to injective objects.
\end{itemize}
If $I$ is an injective object of $\cA$ then $I_t$ is automatically injective in $\cT$, and thus injective in $\cA$ by the above axiom, and so we have a (non-canonical) splitting $I=I_t \oplus I_f$; this shows that $I_f$ is injective as well.

Let $\cD$ be the bounded derived category of $\cA$, which we think of as the homotopy category of finite length complexes of injectives.  Let $\cD_t$ (resp.\ $\cD_f$) be the full subcategory of $\cD$ on complexes of torsion injective objects (resp.\ torsion-free injective objects).  Define $\cD'$ to be the category of triples $(X, Y, f)$ with $X \in \cD_f$, $Y \in \cD_t$ and $f \in \Hom_{\cD}(X, Y)$.  A morphism $(X, Y, f) \to (X', Y', f')$ in $\cD'$ consists of morphisms $X \to X'$ and $Y \to Y'$ in $\cD$ such that the obvious square commutes. The following is the main result of this section.

\begin{proposition}
\label{prop:derived0}
There is a natural equivalence of categories $\Phi \colon \cD \to \cD'$.
\end{proposition}

\begin{proof}
We begin by defining a functor $\Phi \colon \cD \to \cD'$.  Let $I$ be an object of $\cD$.  We have a short exact sequence of complexes
\begin{displaymath}
0 \to I_t \to I \to I_f \to 0.
\end{displaymath}
Our hypotheses imply that this sequence splits termwise, i.e., we can find a splitting $s \colon I_f \to I$, which may not be compatible with differentials.  Put $f=sd-ds$.  Then one readily verifies that $f$ defines a map of complexes $f \colon I_f \to I_t[-1]$ (here we are using cohomological indexing and we remind that the shift is defined by $I[n]^i = I^{i-n}$).  A different choice of $s$ results in a homotopic $f$, so $f$ as an element of $\Hom_{\cD}(I_f, I_t[-1])$ depends only on $I$.  We put $X=I_f$, $Y=I_t[-1]$ and $\Phi(I)=(X, Y, f)$.

Suppose now that $\phi \colon I \to I'$ is a map in $\cD$.  Then $\varphi$ induces maps $\psi_X \colon X \to X'$ and $\psi_Y \colon Y \to Y'$.  Choose termwise sections $s \colon I_f \to I$ and $s' \colon I'_f \to I'$ and let $f$ and $f'$ be the resulting maps.  Then
\begin{displaymath}
\xymatrix{
X \ar[r]^f \ar[d]_{\psi_X} & Y \ar[d]^{\psi_Y} \\
X' \ar[r]^{f'} & Y' }
\end{displaymath}
commutes up to the homotopy $\phi s - s'\phi \colon X \to Y'[1]$.  Thus $\phi$ induces a map $\psi \colon (X, Y, f) \to (X', Y', f')$ in $\cD'$. If $\phi$ is replaced by a homotopic map, then so is $\psi$.  This completes the definition of $\Phi$.

We now define a functor $\Psi \colon \cD' \to \cD$.  Thus suppose that $(X, Y, f)$ is an object of $\cD'$.  Let $I$ be the cone of $f$, shifted by $-1$.  Precisely, $I^n=X^n \oplus Y^{n-1}$, with the differential being defined by the usual formula $(x,y) \mapsto (d(x), f(x)-d(y))$.  Then $I$ is an object of $\cD$, and we put $\Psi(X, Y, f)=I$.

Now suppose that $\psi \colon (X, Y, f) \to (X', Y', f')$ is a map in $\cD'$.  Choose actual maps of complexes $\tilde{\psi}_X$, $\tilde{\psi}_Y$, $\tilde{f}$ and $\tilde{f}'$ representing the given homotopy classes of maps. Choose a map $h \colon X \to Y'[1]$ such that $\tilde{\psi}_Y \tilde{f}-\tilde{f}' \tilde{\psi}_X=dh+hd$.  Define a map $\varphi \colon I \to I'$ by letting $\varphi^n$ be the map
\begin{displaymath}
X^n \oplus Y^{n-1} \to (X')^n \oplus (Y')^{n-1}, \qquad
(x, y) \mapsto (\tilde{\psi}_X(x), \tilde{\psi}_Y(y)+h(x)).
\end{displaymath}
Then $\varphi$ is a map of complexes.  Furthermore, up to homotopy it is independent of all choices.  This completes the definition of $\Psi$.

Now we show that $\Phi$ and $\Psi$ are quasi-inverse to one another. Start with an object $I \in \cD$. We choose the splitting $I = I_t \oplus I_f$. Then the differential can be written as a sum of three maps $d_t \colon I^\bullet_t \to I^{\bullet+1}_t$, $d_f \colon I^\bullet_f \to I^{\bullet+1}_f$, and $u \colon I^\bullet_f \to I^{\bullet+1}_t$. (Here we use that $\Hom(I^n_t, I^m_f)=0$.) The map $sd-ds$ can be taken to be $u$ in the construction of $\Phi$. When we apply $\Psi$ to $(I_f, I_t[-1], u)$, we get back $I$, so $\Psi \Phi \simeq \id_\cD$. Running this argument backwards shows that $\Phi \Psi \simeq \id_{\cD'}$.
\end{proof}

\begin{corollary}
\label{cor-decomp}
We have the following:
\begin{enumerate}[\rm (a)]
\item Let $M$ be an object of $\cD$. Then there exists a distinguished triangle
\begin{displaymath}
M_t \to M \to M_f \to
\end{displaymath}
with $M_t \in \cD_t$ and $M_f \in \cD_f$. This decomposition is unique up to unique isomorphism.
\item Let $g \colon M \to N$ be a map in $\cD$. Then there exist maps $f \colon M_t \to N_t$ and $h \colon M_f \to N_f$ such that
\begin{displaymath}
\xymatrix{
M_t \ar[r] \ar[d]^f & M \ar[r] \ar[d]^g & M_f \ar[r] \ar[d]^h & \\
N_t \ar[r] & N \ar[r] & N_f \ar[r] & }
\end{displaymath}
is a map of triangles. Suppose $(f', g', h')$ is a second map of triangles. If either $g=g'$, or $f=f'$ and $h=h'$, then $(f,g,h)=(f',g',h')$.
\end{enumerate}
\end{corollary}

\begin{proof}
This is clear from working with $\cD'$. 
\end{proof}

\begin{remark}
\label{rmk:tstruc}
The category $\cD$ has a standard t-structure $\cD^{\le 0}$ consisting of objects $M$ for which $\rH^i(M)=0$ for $i>0$.  Under the above equivalence, an object $(X, Y, f)$ of $\cD'$ belongs to $\cD^{\le 0}$ if $\rH^0(f)$ is surjective and $\rH^i(f)$ is an isomorphism for $i>0$.    
\end{remark}

\begin{example}
Let $\cA$ be the category of finitely generated nonnegatively graded $\bC[t]$-modules, let $\cT$ be the full subcategory on torsion (finite length) objects and let $\cF$ be the full subcategory on torsion-free (projective) objects.  Then $(\cT, \cF)$ is a torsion theory in the above sense, and satisfies the additional axioms we gave.  The indecomposable torsion injective objects are of the form $\bC[t]/t^n$.  There is a single (up to isomorphism) indecomposable torsion-free injective, which is $\bC[t]$.  

The category $\cD_t$ is equivalent to the derived category of $\cT$ (as is always the case), while $\cD_f$ is equivalent to $\rD^b(\Vec)$, the bounded derived category of vector spaces; the equivalence $\rD^b(\Vec) \to \cD_f$ is given by $- \otimes \bC[t]$.  If $V \in \rD^b(\Vec)$ and $M \in \cD_t$ then giving a map $V \otimes \bC[t] \to M$ in $\cD$ is the same as giving a map $V \to M_0$ in $\rD^b(\Vec)$, where $M_0$ is the degree 0 piece of $M$.  Thus the category $\cD'$ can be described equivalently as follows:  its objects are triples $(V, M, f)$, with $V \in \rD^b(\Vec)$, $M \in \cD_t$ and $f \in \Hom_{\rD^b(\Vec)}(V, M_0)$.

The equivalence $\cD=\cD'$ provided by Proposition~\ref{prop:derived0} gives a description of $\cD$ which involves only the derived category of torsion objects and the derived category of vector spaces.  This is the sort of equivalence we seek for $\rD^b(A)$ in the following section.
\end{example}

\subsection{The decomposition into perfect and torsion pieces}
\label{ss:pt-decomp}

Let $\cT$ be the subcategory of $\Mod_A$ consisting of torsion objects, and let $\cF$ be the subcategory consisting of torsion-free objects. Then $(\cT, \cF)$ forms a hereditary torsion theory, in the sense of the previous section. We set $\Tors = \cD_t$, which is the subcategory of $\rD^b(A)$ consisting of objects which are quasi-isomorphic to finite length complexes of torsion modules, and we set $\Perf = \cD_f$, which is the subcategory of $\rD^b(A)$ consisting of perfect objects, i.e., those which are quasi-isomorphic to a finite length complex of projective (= torsion-free injective) modules.

Let $\cD'$ be the category whose objects are triples $(X, Y, f)$ with $X \in \Perf$, $Y \in \Tors$, and $f \in \Hom_{\rD^b(A)}(X, Y)$; morphisms are defined in the obvious manner (see \S\ref{ss:deriveddecomp}). Proposition~\ref{prop:derived0} applies in our situation to give:

\begin{proposition}
\label{prop:derived1}
There is a natural equivalence $\rD^b(A) \cong \cD'$.
\end{proposition}

We will refine this equivalence in the next section. For now, we note that Corollary~\ref{cor-decomp} gives the following important result:

\begin{theorem} \label{thm:pt-decomp}
We have the following:
\begin{enumerate}[\rm (a)]
\item Let $M$ be an object of $\rD^b(A)$. Then there exists a distinguished triangle
\begin{displaymath}
M_t \to M \to M_f \to
\end{displaymath}
with $M_t \in \Tors$ and $M_f \in \Perf$. This decomposition is unique up to unique isomorphism.
\item Let $g \colon M \to N$ be a map in $\rD^b(A)$. Then there exist maps $f \colon M_t \to N_t$ and $h \colon M_f \to N_f$ such that
\begin{displaymath}
\xymatrix{
M_t \ar[r] \ar[d]^f & M \ar[r] \ar[d]^g & M_f \ar[r] \ar[d]^h & \\
N_t \ar[r] & N \ar[r] & N_f \ar[r] & }
\end{displaymath}
is a map of triangles. Suppose $(f', g', h')$ is a second map of triangles. If either $g=g'$, or $f=f'$ and $h=h'$, then $(f,g,h)=(f',g',h')$.
\end{enumerate}
\end{theorem}

\begin{remark}
By Proposition~\ref{prop:locS}, one can take $M_t=\rR\Gamma_{\fm}(M)$ and $M_f=\rR S(T(M))$.
\end{remark}

\begin{remark} \label{rmk:newIT}
A consequence of Theorem~\ref{thm:pt-decomp} is that a non-exact projective complex of $A$-modules whose homology consists of torsion modules has infinite length. This statement can be seen as an instance of the ``new intersection theorem'' \cite[Theorem 13.4.1]{roberts}.
\end{remark}

\subsection{Description of $\rD^b(A)$}
\label{ss:DbA}

In the previous section, we gave an equivalence $\rD^b(A)=\cD'$ (Proposition~\ref{prop:derived1}). This is not an entirely satisfactory description of $\rD^b(A)$, since the definition of $\cD'$ refers to $\rD^b(A)$ (in the choice of $f$). We now remove this flaw. Let $\Phi \colon \Mod_K \to \Mod_A^{\tors}$ be the equivalence defined in \S \ref{ss:modKmodAequiv}.

\begin{proposition}
\label{prop:derived2}
For $M, N \in \rD^b(K)$ we have a natural isomorphism
\begin{displaymath}
\Hom_{\rD^b(A)}(\rR S(M), \Phi(N))=\Hom_{\rD^b(K)}(N, M)^*.
\end{displaymath}
\end{proposition}

\begin{proof}
Let $N$ be an injective object of $\Mod_K$.  We begin by defining a map
\begin{displaymath}
\trace_N \colon \Hom_A(S(N), \Phi(N)) \to \bC.
\end{displaymath}
Suppose first that $N=Q_{\lambda}$. Then $S(Q_{\lambda})=A \otimes \bS_{\lambda}$ (Corollary~\ref{cor:S-L-free}), while $\Phi(Q_{\lambda})$ is the injective envelope $I$ of the simple $A$-module $\bS_{\lambda}$, and therefore contains $\bS_{\lambda}$ with multiplicity 1. We therefore have natural isomorphisms
\begin{displaymath}
\Hom_A(S(Q_{\lambda}), \Phi(Q_{\lambda})) \to \Hom_{\cV}(\bS_{\lambda}, I) \to \bC
\end{displaymath}
We let $\trace_N$ be the composite map.

Now suppose that $N=\bigoplus_{\lambda} U_{\lambda} \otimes Q_{\lambda}$, where the $U_{\lambda}$ are multiplicity spaces, almost all of which are zero. We define $\trace_N$ to be the composition
\begin{displaymath}
\Hom_A(S(N), \Phi(N)) \to \bigoplus_{\lambda} \End(U_{\lambda}) \otimes \Hom_A(S(Q_{\lambda}), \Phi(Q_{\lambda})) \to \bC,
\end{displaymath}
where the second map is the trace map on $\End(U_{\lambda})$ combined with $\trace_{Q_{\lambda}}$.  Since $\Aut(N) \subset \GL(\bigoplus_\lambda U_\lambda)$ is block upper-triangular (by Proposition~\ref{prop:injhomsets1}), it is clear that $\trace_N$ is invariant under $\Aut(N)$.  Thus given a general injective object $N$ of $\Mod_K$, we can choose an isomorphism $N=\bigoplus_\lambda U_{\lambda} \otimes Q_{\lambda}$ and apply the above definition, and the result is independent of the choice of isomorphism.

Suppose now that $N$ and $M$ are injectives of $\Mod_K$.  We have a pairing
\begin{displaymath}
\begin{split}
\Hom_K(N, M) \otimes \Hom_A(S(M), \Phi(N))
&\to \Hom_A(S(N), S(M)) \otimes \Hom_A(S(M), \Phi(N)) \\
&\to \Hom_A(S(N), \Phi(N)) \\
&\to \bC.
\end{split}
\end{displaymath}
The first map is $S$ on the first factor, the second map is composition, the third map is the trace map. This yields a map
\begin{displaymath}
F_{M,N} \colon \Hom_A(S(M), \Phi(N)) \to \Hom_K(N, M)^*
\end{displaymath}
We claim $F_{M,N}$ is an isomorphism. One readily verifies that $(M,N) \mapsto F_{M,N}$ respects direct sums in the obvious sense, so it suffices to verify the claim when $M$ and $N$ are indecomposable injectives. This follows easily from the definition of the trace map.

The definition of $F_{M,N}$ extends easily to complexes of injective objects, as does the fact that it is an isomorphism. This establishes the result.
\end{proof}

We now give our main result on $\rD^b(A)$.  Define a category $\cD''$ as follows.  The objects of $\cD''$ are triples $(X, Y, \gamma)$ where $X$ and $Y$ are objects of $\rD^b(\Mod_K)$ and $\gamma$ is an element of $\Hom_{\rD^b(K)}(Y, X)^*$.  A morphism $\phi \colon (X, Y, \gamma) \to (X', Y', \gamma')$ consists of morphisms $\phi_X \colon X \to X'$ and $\phi_Y \colon Y \to Y'$ in $\rD^b(K)$ such that the diagram
\begin{displaymath}
\xymatrix{
\Hom(Y', X) \ar[r]^{\phi_Y^*} \ar[d]_{(\phi_X)_*} & \Hom(Y, X) \ar[d]^{\gamma} \\
\Hom(Y', X') \ar[r]^-{\gamma'} & \bC }
\end{displaymath}
commutes.  Note that $\cD''$ is defined completely in terms of $\rD^b(\Mod_K)$. We then have:

\begin{theorem} \label{thm:Dequiv}
There is a natural equivalence of categories $\rD^b(A)=\cD''$.
\end{theorem}

\begin{proof}
We have an equivalence of categories
\begin{displaymath}
\cD'' \to \cD', \qquad (X, Y, \gamma) \mapsto (\rR S(X), \Phi(Y), f),
\end{displaymath}
where $f$ corresponds to $\gamma$ under the isomorphism given in Proposition~\ref{prop:derived2}.  The result now follows from Proposition~\ref{prop:derived1}.
\end{proof}

\begin{question}
\label{ques2}
Remark~\ref{rmk:tstruc} gives a description of the standard t-structure on $\cD'$.  Is there a nice description of this t-structure on $\cD''$?
\end{question}

\subsection{\texorpdfstring{Rigidity of $\Mod_A$}{Rigidity of ModA}}
\label{ss:modArigid}

We now show that the category $\Mod_A$ has as few symmetries as possible.

\begin{proposition} \label{prop:autoAtriv}
The auto-equivalence group of $\Mod_A$ is trivial.
\end{proposition}

\begin{proof}
Let $\cI_0$ be the category of indecomposable injectives in $\Mod_A$.  By Theorem~\ref{thm:Ainj}, there are two sorts of objects in this category:  the injective envelopes $I_{\lambda}$ of the simple objects $\bS_{\lambda}$ and the projective objects $P_{\lambda}=A \otimes \bS_{\lambda}$.  Define $\lambda \le \mu$ if $\mu/\lambda \in \HS$.  For $\lambda \le \mu$ we can choose non-zero maps 
\begin{displaymath}
f_{\mu,\lambda} \colon I_{\mu} \to I_{\lambda}, \qquad
g_{\mu,\lambda} \colon P_{\mu} \to P_{\lambda}, \qquad
h_{\lambda,\mu} \colon P_{\lambda} \to I_{\mu}
\end{displaymath}
such that when $\lambda \le \mu \le \nu$ we have
\begin{displaymath}
f_{\nu,\lambda}=f_{\nu,\mu} f_{\mu,\lambda}, \qquad
g_{\nu,\lambda}=g_{\nu,\mu} g_{\mu,\lambda}, \qquad
h_{\mu,\nu}=h_{\lambda,\nu} g_{\mu,\lambda}, \qquad
h_{\lambda,\mu}=f_{\nu,\mu} h_{\lambda,\nu}.
\end{displaymath}
This can be seen as follows. First, we can choose $f$ and $g$ so that the first two equations hold since $N(S)$ is contractible (see \S \ref{ss:quivermodK}).  Next, note that there is a canonical choice for $h_{\lambda,\lambda}$ and then the definition of $h_{\lambda,\mu}$ is forced by either of the third or fourth equation above. The third and fourth equations then follow in general.

Let $F$ be an auto-equivalence of $\Mod_A$.  Then $F$ induces an auto-equivalence of $\cI_0$ that preserves the set $\{I_\lambda\}$ (since they are the finite length injectives) and the set $\{P_\lambda\}$.  Since $\Mod_A^{\tors}$ and $\Mod_K$ have no auto-equivalences, we can assume that $F$ is the identity on the $I$'s, $P$'s, $f$'s and $g$'s.  We have $F(h_{\lambda,\lambda})=\alpha_{\lambda} h_{\lambda,\lambda}$ for some scalar $\alpha_{\lambda}$.  Suppose $\lambda \le \mu$.  Applying $F$ to the commutative square
\begin{displaymath}
\xymatrix{
P_{\lambda} \ar[r]^{h_{\lambda,\lambda}} & I_{\lambda} \\
P_{\mu} \ar[u]^{g_{\mu,\lambda}} \ar[r]^{h_{\mu,\mu}} & I_{\mu} \ar[u]_{f_{\mu,\lambda}} }
\end{displaymath}
shows that $\alpha_{\lambda}=\alpha_{\mu}$.  We conclude that $\alpha_{\lambda}=\alpha$ is independent of $\lambda$.  We now obtain an isomorphism $\phi \colon \id \to F$ by defining $\phi(P_{\lambda})$ to be the identity but $\phi(I_{\lambda})$ to be scaling by $\alpha$.  This isomorphism extends to one on all of $\Mod_A$ by Lemma~\ref{prop:equiv}.
\end{proof}

\begin{remark}
It is also true that there are no equivalences $\Mod_A \to \Mod_A^{\op}$:  indeed, every object of $\Mod_A$ has finite injective dimension, but every non-projective object has infinite projective dimension.
\end{remark}

\begin{proposition}
The automorphism group of the identity functor of $\Mod_A$ is $\bC^{\times}$.
\end{proposition}

\begin{proof}
The proof is the same as the proof of Proposition~\ref{prop:autoAtriv}.
\end{proof}

\subsection{\texorpdfstring{$\rK$-theory of $\Mod_A$}{K-theory of ModA}}
\label{ss:modAKtheory}

We now compute the Grothendieck group of $\Mod_A$.

\begin{proposition}
\label{prop:KA}
The group $\rK(\Mod_A)$ is a free module of rank $2$ over $\rK(\cV_\rf)$ with basis $\{[\bC], [A]\}$.
\end{proposition}

\begin{proof}
This follows immediately from finiteness of injective dimension (Theorem~\ref{thm:injA}) and the classification of injectives (Theorem~\ref{thm:Ainj}).
\end{proof}

The following proposition is similar to the above, but is more convenient in applications.

\begin{proposition}
\label{prop:ind}
Let $\cP$ be a property of finitely generated $A$-modules.  Suppose that:
\begin{enumerate}[\rm (a)]
\item Given a short exact sequence
\begin{displaymath}
0 \to M_1 \to M_2 \to M_3 \to 0
\end{displaymath}
in which $\cP$ holds on two of the modules, it holds on the third.
\item The property $\cP$ holds for all simple modules $\bS_{\lambda}$.
\item The property $\cP$ holds for all projective modules $A \otimes \bS_\lambda$.
\end{enumerate}
Then $\cP$ holds for all finitely generated $A$-modules.
\end{proposition}

\begin{proof}
Again, this follows immediately from finiteness of injective dimension (Theorem~\ref{thm:injA}) and the classification of injectives (Theorem~\ref{thm:Ainj}).
\end{proof}

\part{Invariants of $A$-modules}


\section{Hilbert series}

Let $M$ be an object of $\cV^{\gfin}$, regarded as a sequence $M=(M_n)_{n \ge 0}$ of representations of symmetric groups $S_n$.  We define the Hilbert series of $M$ as
\begin{displaymath}
H_M(t)=\sum_{n \ge 0} \dim(M_n) \frac{t^n}{n!}.
\end{displaymath}
We then have the following result \cite[Thm.~3.1]{snowden}:  if $M$ is a finitely generated module over a twisted commutative algebra finitely generated in degree 1 then $H_M(t)$ is a polynomial in $t$ and $e^t$.  For finitely generated $A$-modules, the proof shows that only the first power of $e^t$ can appear, i.e., the Hilbert series is of the form $p(t)e^t+q(t)$ where $p$ and $q$ are polynomials.  Our goal in this section is to generalize this result in two very different ways.

First, we define ``enhanced'' Hilbert series, which not only record the dimensions of the $M_n$, but their structure as $S_n$-representations. The definition and rationality result is given in \S\ref{ss:enhancedhilbert}.  In \S \ref{ss:charpoly}, we show that the ``$p$-part'' of the enhanced Hilbert series is equivalent to the so-called ``character polynomial'' from the representation theory of symmetric groups.  Using our theory, we give a conceptual derivation of a formula for character polynomials.  In \S \ref{ss:hilbertq}, we show that the ``$q$-part'' of the enhanced Hilbert series is determined by local cohomology.  This gives a conceptual description of the failure of the character polynomial to compute the actual character in small degrees.

Second, we introduce certain differential operators in \S\ref{ss:diffop} and show every module satisfies a differential equation.  This is a sort of categorification of the rationality result from \cite{snowden}, and provides an analogue of homogeneous systems of parameters in a certain sense.

\subsection{Enhanced Hilbert series} \label{ss:enhancedhilbert}

Let $\lambda$ be a partition of $n$. We define the following notations. 
\begin{itemize}
\item Let $c_{\lambda}$ denote the conjugacy class of $S_n$ corresponding to $\lambda$, i.e., the parts of $\lambda$ specify the cycle type of the permutation. In particular, $(1^n)$ corresponds to the identity conjugacy class. We write $\trace(c_{\lambda} \vert M)$ for the trace of $c_{\lambda}$ on $M_n$.  
\item We write $t^{\lambda}$ for $t_1^{m_1(\lambda)} t_2^{m_2(\lambda)} \cdots$, where $m_i(\lambda)$ denotes the number of times $i$ appears in $\lambda$.  
\item We write $\lambda!$ for $m_1(\lambda)! m_2(\lambda)! \cdots$.  
\end{itemize}

\begin{example}
If $\lambda=(2,1,1,1)$ then $c_{\lambda}$ is the conjugacy class of transpositions in $S_5$, $t^{\lambda} = t_1^3 t_2$, and $\lambda!=3! \cdot 1!=6$.
\end{example}

Let $M$ be an object of $\cV^{\gfin}$.  The {\bf enhanced Hilbert series} of $M$ is the formal series in variables $t_1, t_2, \ldots$ given by
\begin{displaymath}
\tilde{H}_M(t)=\sum_{\lambda} \trace(c_{\lambda} \vert M) \frac{t^{\lambda}}{\lambda!}.
\end{displaymath}
The isomorphism class of $M$ as an object of $\cV$ is completely determined by $\tilde{H}_M$. The enhanced Hilbert series therefore has much more information in it than the usual Hilbert series; in fact, one can recover the usual Hilbert series directly from the enhanced Hilbert series by setting $t_i$ to 0 for $i \ge 2$.  The enhanced Hilbert series is multiplicative (we emphasize that $\otimes$ is {\it not} the pointwise tensor product on sequences of $S_n$-representations, but rather the induction product, see \S\ref{sec:catV}):
\begin{displaymath}
\tilde{H}_{M \otimes N}=\tilde{H}_M \tilde{H}_N
\end{displaymath}
(see, for example, \cite[Proposition 7.18.2]{stanley}). Thus $M \mapsto \tilde{H}_M$ provides an isomorphism of rings $\rK(\cV^{\gfin}) \otimes \bQ \to \bQ \lbb t_i \rbb$.  Define 
\begin{displaymath}
T_0=\sum_{i \ge 1} t_i.
\end{displaymath}

\begin{lemma} \label{lem:hilbert-A}
$\tilde{H}_A(t)=\exp(T_0)$.
\end{lemma}

\begin{proof}
As an object of $\cV^{\gfin}$, $A$ corresponds to the sequence of $S_n$-representations which is the trivial representation for each $n$. In particular, $\trace(c_\lambda \vert A) = 1$ for all $\lambda$, so $\tilde{H}_A(t) = \sum_\lambda t^\lambda / \lambda!$. Since the coefficient of $t_1^{p_1} \cdots t_m^{p_m}$ in $\exp(T_0)$ is $(p_1! \cdots p_m!)^{-1}$, this finishes the proof.
\end{proof}

The basic fact on enhanced Hilbert series is the following theorem.

\begin{theorem} \label{thm:enhanced}
Let $M$ be a finitely generated $A$-module. There is an integer $n$ and two (unique) polynomials $p_M(t),q_M(t) \in \bQ[t_1, \dots, t_n]$ so that
$\tilde{H}_M(t) = p_M(t)\exp(T_0) + q_M(t)$.
\end{theorem}

\begin{proof}
It is clear that $\tilde{H}_M$ is a polynomial if $M=\bS_{\lambda}$. Multiplicativity of $\tilde{H}$ and Lemma~\ref{lem:hilbert-A} shows that $M$ satisfies the theorem if $M= A \otimes \bS_{\lambda}$.  Since $\tilde{H}$ factors through K-theory, the theorem follows from Proposition~\ref{prop:ind}.
\end{proof}

\begin{remark}
If we define $\tilde{H}^*_M(t) = \sum_\lambda \trace(c_\lambda \vert M) t^\lambda$ and $T^*_0 = \prod_{i \ge 1} (1-t_i)^{-1}$, then the first proof of Theorem~\ref{thm:enhanced} shows that there are $p(t), q(t) \in \bZ[t_1, \dots, t_n]$ so that $\tilde{H}^*_M(t) = p(t) T^*_0 + q(t)$.
\end{remark}

\begin{remark}
If we perform the substitution $t_i \mapsto p_i / i$, where $p_i$ denotes the $i$th power sum symmetric function $p_i = \sum_{n \ge 1} x_n^i$, then $\tilde{H}_M(t)$ becomes the Frobenius characteristic of the family of symmetric group representations $M = (M_n)_{n \ge 0}$ (see \cite[\S 7.18]{stanley}). 
\end{remark}

\subsection{Character polynomials and $p$} \label{ss:charpoly}

Let $M$ be a finitely generated $A$-module.  Using the equality $\tilde{H}_M(t)=p_M(t) \exp(T_0)+q_M(t)$ afforded by Theorem~\ref{thm:enhanced}, a simple computation shows that there is a polynomial $X_M(a_1, \ldots, a_n)$ called the {\bf character polynomial} of $M$, so that the quantity $\trace(c_{\mu} \vert M)$ is obtained from $X_M$ by putting $a_i=m_i(\mu)$, at least for $\vert \mu \vert$ large enough (this fact is also stated in \cite[Theorem 1.6]{fimodules}). We will see shortly that $q_M(t)$ is exactly the failure for this to happen when $|\mu|$ is small, and that we have $\trace(c_\mu \vert M) = X_M(m_1(\mu), \dots, m_n(\mu))$ whenever $|\mu| > \deg q_M(t)$.

In fact, we have the following formula for $X_M$ in terms of $p_M$.  Define a linear map (``umbral substitution'')
\begin{displaymath}
\bQ[t_i] \to \bQ[a_i], \qquad \prod_i t_i^{d_i} \mapsto \prod_i (a_i)_{d_i},
\end{displaymath}
where $(x)_d=x(x-1) \cdots (x-d+1)$ is the falling factorial.  This map is not a ring homomorphism.  We denote the image of $p$ under this map by $\downarrow p$.  We then have:

\begin{proposition}
\label{prop:hilbertp}
In the above notation, $X_M=\ \downarrow p_M$.
\end{proposition}

\begin{proof}
Pick $\mu$ so that $|\mu| > \deg q_M(t)$. Then $\trace(c_\mu \vert M)/\mu!$ is the coefficient of $t^\mu$ in $p_M(t) \exp(T_0)$, and we claim that this is the same as $(\downarrow p_M)(m_1(\mu), m_2(\mu), \dots)/\mu!$. 

We will prove the slightly more general statement that $\mu!$ times the coefficient of $t^\mu$ in $p(t) \exp(T_0)$ is given by $(\downarrow p)(m_1(\mu), m_2(\mu), \dots)$ for any polynomial $p(t)$. By linearity, it is enough to check this claim when $p(t) = t^\nu$ is a monomial. The coefficient of $t^\mu$ in $t^\nu \exp(T_0)$ is $\prod_i ((m_i(\mu) - m_i(\nu))!)^{-1}$ with the convention that $(n!)^{-1} = 0$ if $n<0$. Multiplying this by $\mu!$ gives $\prod_i(m_i(\mu))_{m_i(\nu)} = (\downarrow t^\nu)(m_1(\mu), m_2(\mu), \dots)$, as desired.
\end{proof}

Observe that both $X_M$ and $p_M$ are unaffected if $M$ is changed by a finite length object, and so both make sense for $M \in \Mod_K$.  In fact, $p$ and $X$ define linear maps from $\rK(\Mod_K) \otimes \bQ$ to $\bQ[t_i]$ and $\bQ[a_i]$, respectively.  The map $p$ is a ring homomorphism, while $X$ is not.  (To see that $p$ is a ring homomorphism, note that $p_{A \otimes V}=\tilde{H}_V$ for $V \in \cV$, and so $p$ is multiplicative on the injective objects of $\Mod_K$.  Since these span the $\rK$-group, the result follows.)

Write $X^{\lambda}$ and $p^{\lambda}$ in place of $X_M$ and $p_M$ when $M=L_{\lambda}$.  Note that the polynomial $X^{\lambda}$ describes the character of the irreducible representation $\bM_{(N, \lambda)}$ for $N$ sufficiently large; in fact, this is how character polynomials are usually thought of. Our resolution for $L_{\lambda}$ in $\Mod_K$ can be used to give an explicit formula for $X^{\lambda}$.  By Proposition~\ref{prop:cartanmatrix}, we have
\begin{displaymath}
p^{\lambda} = \sum_{\mu,\, \lambda/\mu \in \VS} (-1)^{|\lambda|-|\mu|} p_{A \otimes \bS_\mu}.
\end{displaymath}
Since $\tilde{H}_{A \otimes \bS_\mu}=\tilde{H}_{\bS_{\mu}} \tilde{H}_A=\tilde{H}_{\bS_{\mu}} \exp(T_0)$, we have
\begin{displaymath}
p_{A \otimes \bS_\mu}(t) = \tilde{H}_{\bS_{\mu}}(t) = \sum_{|\lambda| = n} \trace(c_\lambda \vert \bM_\mu) \frac{t^\lambda}{\lambda!}.
\end{displaymath}
For formulas for these traces, see \cite[\S I.7]{macdonald} or \cite[\S 7.18]{stanley}.  Since $X^{\lambda}=\downarrow p^{\lambda}$, the above two formulas give an explicit formula for $X^{\lambda}$.

\begin{remark}
In fact, sharper results on character polynomials are known; see \cite[Examples I.7.13--14]{macdonald} and \cite{garsia} for further details.  In particular, a more efficient version of the formula for $X^{\lambda}$ given above can be found in \cite[Proposition I.1]{garsia}. We believe that our derivation of these facts from the structure of $\Mod_K$ is more conceptual than the usual approaches, which is why we have included this section despite the known results in the literature.
\end{remark}

\begin{example}
For a partition $\lambda$, define $Y^\lambda(t) = p_{A \otimes \bS_\lambda}(t)$.
\begin{enumerate}[(a)]
\item We have $[L_1] = [Q_1] - [Q_\emptyset]$. Since $Y^1(t) = t_1$ and $Y^\emptyset(t) = 1$, we get $X^1(a_1) = a_1 - 1$. 

\item We have $[L_{2,1}] = [Q_{2,1}] - [Q_2] - [Q_{1,1}] + [Q_1]$. We calculate the polynomials $Y^\mu$: 
\[
Y^{2,1} = \frac{t_1^3}{3} - t_3, \qquad Y^2 = \frac{t_1^2}{2} + t_2, \qquad Y^{1,1} = \frac{t_1^2}{2} - t_2, \qquad Y^1 = t_1.
\]
Hence $X^{2,1}(a_1, a_2, a_3) = (a_1)_3/3 - a_3 - (a_1)_2 + a_1$. \qedhere
\end{enumerate}
\end{example}

\subsection{Local cohomology and $q$} \label{ss:hilbertq}

Recall from \S \ref{ss:loccoh} that $\rH^0_{\fm}(M)$ denotes the torsion submodule of $M$ and $\rR^i \Gamma_{\fm}=\rH^i_{\fm}$ denotes the $i$th derived functor of $\rH^0_{\fm}$.

\begin{proposition} \label{prop:hilbertlocal}
Let $M$ be an object of $\Mod_A$.  Then
\begin{displaymath}
q_M(t)=\tilde{H}_{\rR \Gamma_{\fm}(M)}(t)=\sum_{i \ge 0} (-1)^i \tilde{H}_{\rH^i_{\fm}(M)}(t),
\end{displaymath}
with $q_M$ as in Theorem~\ref{thm:enhanced}.
\end{proposition}

This is an analogue of \cite[Corollary A1.15]{syzygies} which expresses the difference between the Hilbert polynomial and Hilbert function of a graded module over a polynomial ring in terms of the Euler characteristic of its local cohomology modules.

\begin{proof}
Each side factors through K-theory, so it suffices to check the proposition when $M$ is either a simple module $\bS_{\lambda}$ or a projective module $A \otimes \bS_{\lambda}$.  In the first case, $\rH^0_{\fm}(M)=M$ while $\rH^i_{\fm}(M)=0$ for $i>0$, which proves the statement.  In the second case, $\rH^i_{\fm}(M)=0$ for all $i$, while $\tilde{H}_M(t)=\tilde{H}_{\bS_{\lambda}}(t) \exp(T_0)$, which shows that $q_M(t) = 0$, and the statement follows.
\end{proof}

\subsection{Differential operators} \label{ss:diffop}

The category $\cV$ admits a derivation $\bD \colon \cV \to \cV$ called the ``Schur derivative.'' We describe this in two of the models described in \S\ref{sec:catV}. First suppose that $F$ is a polynomial functor.  Then $\bD F$ is the functor assigning to a vector space $V$ the subspace of $F(V \oplus \bC)$ on which $\bG_m$ acts through its standard character (this copy of $\bG_m$ acts by multiplication on $\bC$ only).  For example,
\begin{displaymath}
\bigwedge^{n}(V \oplus \bC)=\bigwedge^{n}(V) \oplus \bigwedge^{n-1}(V) \otimes \bC,
\end{displaymath}
and so $\bD(\bigwedge^{n})=\bigwedge^{n-1}$.  Similarly, $\bD(\Sym^n)=\Sym^{n-1}$. So in particular, $\bD(A) = A$.

Now suppose that $W=(W_n)$ is a sequence of representations of symmetric groups.  Then $W'=\bD W$ is the sequence with $W'_n=W_{n+1} \vert_{S_n}$, that is, $W'_n$ is obtained by restricting the representation $W_{n+1}$ of $S_{n+1}$ to $S_n$.  Notice that
\begin{displaymath}
H_{W'}(t)=\sum_{n \ge 0} \dim(W'_n) \frac{t^n}{n!}=\sum_{n \ge 0} \dim(W_{n+1}) \frac{t^n}{n!}=\frac{d}{dt} H_W(t).
\end{displaymath}
In other words,
\begin{displaymath}
H_{\bD W}(t)=\frac{d}{dt} H_W(t),
\end{displaymath}
that is, the Schur derivative lifts the usual derivative on Hilbert series. There is a simple formula for the enhanced Hilbert series, but we will not discuss it in this paper:
\[
\tilde{H}_{\bD W}(t) = \frac{\partial}{\partial t_1} \tilde{H}_W(t).
\]

Here are some basic properties of the Schur derivative:
\begin{itemize}
\item Leibniz rule:  $\bD(F \otimes G)$ is naturally isomorphic to $(\bD(F) \otimes G) \oplus (F \otimes \bD(G))$. 
\item Chain rule:  $\bD(F \circ G)=(\bD(F) \circ G) \otimes \bD(G)$, where $\circ$ is composition of Schur functors.
\end{itemize}

Let $M$ be an $A$-module which we think of a sequence of symmetric group representations $(M_n)$. Since $A$ corresponds to the sequence of trivial representations, we have a $S_n$-equivariant multiplication map
\begin{displaymath}
M_n \to M_{n+1}.
\end{displaymath}
Since $(\bD M)_n = M_{n+1}|_{S_n}$, this induces a map of $A$-modules
\begin{displaymath}
M \to \bD M.
\end{displaymath}
We define $\partial(M)$ to be the two-term complex $[M \to \bD M]$.  More generally, if $M$ is a complex of $A$-modules, then the above procedure defines a map $M \to \bD M$, and we define $\partial(M)$ to be the cone of this complex. This induces a triangulated functor
\begin{displaymath}
\partial \colon \rD^b(A) \to \rD^b(A),
\end{displaymath}
where $\rD^b(A)$ is the bounded derived category of $\Mod_A$.

\begin{theorem} \label{thm:diffeq}
Let $M \in \rD^b(A)$ be a complex. Then there are nonnegative integers $n_1, n_2$ so that $\partial^{n_1} \bD^{n_2} M = 0$ in $\rD^b(A)$.
\end{theorem}

\begin{proof}
Represent $M$ as a complex. The statement is true for $M$ if we can prove it is true for each term of $M$, so we can reduce to the case of a module. All finite length objects are killed by some power of $\bD$. So using Proposition~\ref{prop:ind}, we just need to show the statement for $A \otimes \bS_\lambda$. We will do this by induction on $|\lambda|$. When $|\lambda|=0$, we note that $\partial(A) = 0$ since the map $A \to \bD A$ is an isomorphism. Using the Leibniz rule, we have 
\[
\bD(A \otimes \bS_\lambda) = (\bD(A) \otimes \bS_\lambda) \oplus (A \otimes \bD(\bS_\lambda)) = A \otimes (\bS_\lambda \oplus \bD(\bS_\lambda)),
\]
and the map
\[
\partial \colon A \otimes \bS_\lambda \to \bD(A \otimes \bS_\lambda)
\]
is an injection with cokernel $A \otimes \bD(\bS_\lambda)$. Now $\bD(\bS_\lambda)$ is a sum of Schur functors $\bS_\mu$ where $|\mu| < |\lambda|$. It follows that $\rH^i(\partial(A \otimes \bS_{\lambda}))$ is zero, except when $i=1$, where it is a direct sum of smaller free modules, so we are done by induction.
\end{proof}

\begin{remark}
Let $M$ be a graded $\bC[t]$-module in nonnegative degrees.  Define $\bD(M)$ to be the module supported in nonnegative degrees whose degree $n$ piece is the degree $n+1$ piece of $M$.  Define $\partial(M)$ to be the two-term complex $[M \to \bD{M}]$, where the map is multiplication by $t$.  Then any finitely generated $\bC[t]$-module is annihilated by an operator of the form $\bD^n \partial^m$ with $m=0$ (if $M$ is torsion) or $m=1$ (otherwise).

Suppose $M$ as above is not a torsion module.  Then the homology of $\partial(M)$ is torsion.  This is a reflection of the fact that $t$ forms a system of parameters for $M$ (recall that if the support variety of a module $M$ has dimension $d$, then a system of parameters is a sequence $x_1, \dots, x_d$ so that $M/(x_1,\dots,x_d)M$ is a torsion module). Thus Theorem~\ref{thm:diffeq} can be viewed as something analogous to the existence of a finite system of parameters for $A$-modules.
\end{remark}

\begin{remark}
One might hope for a common generalization of Theorems~\ref{thm:enhanced} and~\ref{thm:diffeq}.  We will address this in \cite{hilbert}. \qedhere
\end{remark}


\section{Koszul duality and the Fourier transform} \label{sec:fourier}

In this section we introduce the Fourier transform, which is a certain modification of Koszul duality that is special to modules with $\GL_\infty$-actions. The first few subsections review well-known material on dg-algebras and Koszul duality. The fundamental finiteness properties of resolutions of finitely generated $A$-modules are discussed in \S\ref{ss:hilbsyz}. The $\natural$-duality functor, which is intermediate for the definition of the Fourier transform, is introduced in \S\ref{ss:natural} and its relation to Koszul duality is discussed in \S\ref{ss:natural-koszul}. The definition of the Fourier transform and its basic properties are given in \S\ref{ss:fourier}. In \S\ref{ss:poincare}, we study the Poincar\'e series of modules in $\Mod_A$. In \S\ref{ss:hilbertreciprocity}, we explain how the Fourier transform gives a certain reciprocity between perfect and torsion complexes, and how this interchanges the $p$- and $q$-parts of enhanced Hilbert series.  In \S\ref{ss:modKfourier}, we explain how the Fourier transform can be transferred to $\rD^b(K)$.  Finally in \S\ref{ss:exampleEFW} we calculate some examples of this theory.

\subsection{Review of chain complexes}
\label{ss:chcx}

Let $\cA$ be an abelian category. We write $\Ch(\cA)$ for the category of chain complexes in $\cA$. We use homological indexing, so if $U$ is a complex then $U_n$ is its $n$th term and the differential is a map $U_n \to U_{n-1}$. We say that a complex $U$ is {\bf bounded} (resp.\ {\bf bounded below}, resp.\ {\bf bounded above}) if $U_n=0$ for all but finitely many $n$  (resp.\ for $n \ll 0$, resp.\ for $n \gg 0$). We write $\Ch^b(\cA)$, $\Ch^+(\cA)$, and $\Ch^-(\cA)$ for the category of bounded, bounded below, and bounded above complexes.

Suppose $\otimes$ is a tensor structure on $\cA$. Then there is an induced tensor structure on both $\Ch^+(\cA)$ and $\Ch^-(\cA)$, which we still denote by $\otimes$, defined as follows. Suppose $U$ and $V$ are both bounded-below (or both bounded-above) chain complexes. Then $U \otimes V$ is the chain complex given by
\begin{displaymath}
(U \otimes V)_n=\bigoplus_{i+j=n} U_i \otimes V_j, \qquad
d(u \otimes v)=du \otimes v + (-1)^j u \otimes dv,
\end{displaymath}
where here $u \in U_i$ and $v \in V_j$. Note that the direct sum is finite.

Now suppose that $\tau$ is a symmetric structure on $\otimes$, i.e., $\tau_{U,V}$ is an isomorphism $U \otimes V \to V \otimes U$ for all $U$ and $V$ such that $\tau_{V,U} \tau_{U,V} = 1_{U \otimes V}$. Then there is an induced symmetric tensor structure on both $\Ch^+(\cA)$ and $\Ch^-(\cA)$, which we still denote by $\tau$, defined as follows. Let $U$ and $V$ be two bounded-below chain complexes. Then
\begin{displaymath}
\tau_{U,V} \colon U \otimes V \to V \otimes U, \qquad
\tau_{U,V}=\bigoplus_{i+j=n} (-1)^{ij} \tau_{U_i,V_j}.
\end{displaymath}
In this way, $\Ch^+(\cA)$ and $\Ch^-(\cA)$ are symmetric tensor categories. If $\cA$ has all direct sums, the above discussion applies equally well to $\Ch(\cA)$.

For us, a {\bf dg-algebra} is a commutative, associative, unital algebra in $\Ch^+(\cA)$. If $R$ is a dg-algebra, then a dg-$R$-module is an object $M$ of $\Ch^+(\cA)$ equipped with a map $R \otimes M \to M$ satisfying the usual axioms. The derived category of $R$, denoted $\rD^+(R)$, is the localization of the category of dg-modules at the class of quasi-isomorphisms. This category admits a symmetric tensor product. We regard algebras in $\cA$ as dg-algebras concentrated in degree 0, and similarly for modules. Similar remarks hold for coalgebras.

Given a chain complex $U$, we let $U[n]$ be the chain complex with $U[n]_i=U_{i+n}$. We note that if $U$ is an object of $\cA$, regarded as a chain complex concentrated in degree 0, then $\Sym^n(U[-1])=(\lw{n}{U})[-n]$.

\subsection{Review of Koszul duality}

The category of $\bZ$-graded complex vector spaces (not to be confused with the subcategory of $\Ch(\Vec)$ with trivial differentials) is denoted $\Vec^\bZ$. Pick $V \in \Vec^\bZ$ and let $R=\Sym(V)$ and $S=\Sym(V[-1])$. Note that $S_n=\lw{n}(V)$. We regard $R$ as a dg-algebra and $S$ as a dg-coalgebra in $\Ch(\Vec^{\bZ})$. Let $\bK \in \Ch(\Vec^{\bZ})$ be the Koszul complex on $V$. This is the complex with $\bK_n = R \otimes \lw{n}(V)$ for $n \ge 0$, and $\bK_n=0$ for $n<0$, and where the differential is the Koszul differential. The complex $\bK$ is naturally a dg-$R$-module and a dg-$S$-comodule. 

We define the {\bf Koszul dual} of a dg-$R$-module $M$ to be the dg-$S$-comodule $\cK(M) = \bK \otimes_R M$; similarly, we define the {\bf Koszul dual} of a dg-$S$-comodule $N$ to be the dg-$R$-module $\cK'(N)=N \otimes^S \bK$. There are natural quasi-isomorphisms $\bK \otimes_R \bK \to S$ and $R \to \bK \otimes^S \bK$ given by exterior algebra multiplication and the dual of symmetric algebra multiplication, respectively, which yield maps $\cK \cK' \to \id$ and $\id \to \cK' \cK$. These induce isomorphisms at the derived level:

\begin{proposition}
\label{prop:kosz}
The functors $\cK \colon \rD^+(R) \to \rD^+(S)$ and $\cK' \colon \rD^+(S) \to \rD^+(R)$ are mutually quasi-inverse equivalences.
\end{proposition}

\begin{proof}
This is essentially \cite[Corollary 2.7]{efs}, except that there the result is phrased in terms of modules over the graded dual of $S$ (which is a dg-algebra) rather than comodules over $S$.  
\end{proof}

\subsection{The Hilbert syzygy theorem}
\label{ss:hilbsyz}

Keep the notation of the previous section: $V$ is a graded complex vector space, $R = \Sym(V)$, and $S = \Sym(V[-1])$.

Let $M$ be a dg-$S$-comodule. Then the homology $\rH_i(M)$ is a graded vector space; write $\rH_i(M)_j$ for its $j$th graded piece. We define the $k$th {\bf linear strand} of $M$ to be the graded vector space $L_k(M)=\bigoplus_{i \in \bZ} \rH_{i-k}(M)_i$, where the $i$th summand has degree $i$. This is naturally a graded comodule over the coalgebra $S'=L_0(S)=\bigwedge V$. If $M_1 \to M_2 \to M_3 \to$ is a triangle in $\rD^+(S)$ then there is a long exact sequence of linear strands:
\begin{displaymath}
\cdots \to L_{-1}(M_3) \to L_0(M_1) \to L_0(M_2) \to L_0(M_3) \to L_1(M_1) \to \cdots
\end{displaymath}
We now recall the classical \emph{Hilbert syzygy theorem}:

\begin{theorem} \label{thm:hilb-syz-orig}
Suppose $V$ is finite dimensional and $M$ is a finitely generated graded $R$-module. Then $L_i(\cK(M))$ is a finitely cogenerated $S'$-module for all $i$, and zero for all but finitely many $i$.
\end{theorem}

In the context of the theorem, $S'$ is finite dimensional, and so the theorem is equivalent to the statement that the total homology $\rH_{\bullet}(\cK(M))$ is finite dimensional, which is how it is usually stated (see \cite[Corollary 19.7]{eisenbud}). However, from our point of view, the fact that $\rH_{\bullet}(\cK(M))$ is finite dimensional is a red herring, and the above form of the theorem is preferable.

When $V$ is infinite dimensional, the above theorem is easily seen to be false in general. We now show that it holds in the presence of $\GL$-equivariance.

From now on we take $V=\bC^{\infty}$, and write $A$ in place of $R$ and $B$ in place of $S$; by an $A$-module or $B$-comodule we will always mean one equipped with a polynomial $\GL_\infty$-action, and we use the grading coming from the equivariance. We let $B'=L_0(B)$; this is the same as $B$, except it is concentrated in homological degree 0. There is an obvious notion of finite cogeneration for $B'$-modules, taking into account the $\GL$-action (similar to what we have previously discussed for $A$-modules).

\begin{theorem}
\label{thm:hilbsyz}
Let $M$ be a finitely generated $A$-module. Then $L_i(\cK(M))$ is a finitely cogenerated $B'$-module for all $i$, and zero for all but finitely many $i$.
\end{theorem}

\begin{proof}
Suppose
\begin{displaymath}
0 \to M_1 \to M_2 \to M_3 \to 0
\end{displaymath}
is a short exact sequence of $A$-modules. Applying $L_{\bullet}(\cK(-))$, we obtain a long exact sequence of $B'$-modules. We thus see that if the theorem holds for two of the three modules above then it holds for the third. By Proposition~\ref{prop:ind}, it therefore suffices to verify the theorem for the $A$-modules $\bS_{\lambda}$ and $A \otimes \bS_{\lambda}$. Now, $\cK(\bC)=B$ and $\cK(A)=\bC$ as dg-$B$-modules; this is a simple and standard calculation with the Koszul complex. The $\GL$-equivariance does not enter into the definition of $\cK$, and so tensoring with $\bS_{\lambda}$ pulls out of $\cK$. We thus find $\cK(\bS_{\lambda})=B \otimes \bS_{\lambda}$ and $\cK(A \otimes \bS_{\lambda})=\bS_{\lambda}$. The result now easily follows.
\end{proof}

\begin{remark}
The space $L_i(\cK(M))$ is the $i$th linear strand of the minimal free resolution of $M$, in the usual sense. Thus, even though the resolution of $M$ is typically infinite, there are only finitely many linear strands, and each linear strand can be given an algebraic structure (namely, that of a $B'$-comodule) under which it is finitely cogenerated. We can therefore say that resolutions are determined by a finite amount of data, in a strong sense.
\end{remark}

\begin{remark}
There is a similar result for $\cK'$ that can be proved in essentially the same way. We will circumvent the need for such a result using the $\natural$-dual functor defined below.
\end{remark}

Define the {\bf (Castelnuovo--Mumford) regularity} of an $A$-module $M$ to be the supremum of the set of integers $n$ for which the $n$th linear strand of the minimal free resolution of $M$ is non-zero.

\begin{corollary} \label{cor:regularity}
Every object of $\Mod_A$ has finite regularity.
\end{corollary}

\subsection{\texorpdfstring{The $\natural$-dual functor}{The \#-dual functor}} \label{ss:natural}

Recall from \S\ref{sec:catV} that $\cV=\cV_1$ is the category of polynomial representations of $\GL_\infty$ and that $\cV^{\gf}$ is the subcategory of graded-finite ones. We write $\tau_{V,W} \colon V \otimes W \to W \otimes V$ for the standard isomorphism $\tau_{V,W}(v \otimes w)=w \otimes v$, which is a symmetric structure on the tensor product on $\cV$. We let $\sigma$ be the symmetric structure given by $\sigma_{V,W}(v \otimes w)=(-1)^{\deg(v) \deg(w)} w \otimes v$, when $v$ and $w$ are homogeneous. Transpose is then a \emph{symmetric} tensor functor $(\cV, \tau) \to (\cV, \sigma)$; see \cite[\S 7.4]{expos} for more details. Both $\sigma$ and $\tau$ induce symmetric structures on the tensor product on $\Ch(\cV)$ and its subcategory $\Ch^+(\cV^{\gf})$, as discussed in \S \ref{ss:chcx}.

We define four operations on $\Ch(\cV)$. Let $M$ be an object of $\Ch(\cV)$. The notation $M_{i,j}$ denotes the $j$th graded piece of $M_i$, where $i$ indicates homological grading.

\begin{itemize}
\item The {\bf transpose} of $M$, denoted $M^{\dag}$, is the complex with $(M^{\dag})_n=(M_n)^{\dag}$, that is, we just apply the transpose to each term of $M$.
\item The {\bf dual} of $M$, denoted $M^{\vee}$, is the complex with $(M^{\vee})_n=(M_{-n})^{\vee}$, that is, we apply duals termwise and flip the complex.
\item The {\bf right sheer} of $M$, denoted $M^R$, is the complex with $(M^R)_n=\bigoplus_{i \in \bZ} M_{n+i,i}$. 
\item The {\bf left sheer} of $M$, denoted $M^L$, is the complex with $(M^L)_n=\bigoplus_{i \in \bZ} M_{n-i,i}$.
\end{itemize}
We have the following easy facts:
\begin{itemize}
\item Transpose commutes with dual and right and left sheer.
\item Right and left sheer are inverse to each other.
\item Dual intertwines right and left sheer, that is $(M^R)^{\vee}=(M^{\vee})^L$.
\item Right and left sheer are symmetric tensor functors $(\Ch(\cV), \tau) \to (\Ch(\cV), \sigma)$. (Sheering does not preserve the standard symmetric structure $\tau$ on $\Ch(\cV^{\gf})$ since both $\tau$ and sheering involve homological degree.)
\end{itemize}

We now define a new contravariant operation on $\Ch(\cV)$, denoted $(-)^{\natural}$, as follows:
\begin{displaymath}
M^{\natural}=((M^\vee)^{L})^{\dag}.
\end{displaymath}
We claim that $(-)^{\natural} \colon \Ch^+(\cV^{\gf}) \to \Ch(\cV)$ is a \emph{symmetric} tensor functor. Indeed, $(-)^\vee$ sends $\Ch^+(\cV^{\gf})$ to $\Ch^-(\cV^{\gf})$ and is a symmetric tensor functor, $(-)^L$ gives a tensor functor $\Ch^-(\cV^{\gf}) \to \Ch(\cV)$ which interchanges $\sigma$ and $\tau$, and finally $(-)^\dagger$ is a tensor functor on $\Ch(\cV)$ which interchanges $\sigma$ and $\tau$. For $M \in \Ch^+(\cV^{\gf})$, the natural map $M \to (M^{\natural})^{\natural}$ is an isomorphism; this can easily be seen using the properties listed above.

We say that a complex $M \in \Ch(\cV)$ is {\bf $\natural$-finite} if it satisfies the following conditions: 
\begin{enumerate}
\item $M_{i,j}$ is a finite length object of $\cV$, for all $i$ and $j$,
\item there exists $A \in \bZ$ such that $M_n=0$ for $n<A$,
\item there exists $B \in \bZ$ such that $M_{i,j}=0$ for $j-i<B$. 
\end{enumerate}
We let $\Ch^{\natural}(\cV)$ be the full subcategory of $\Ch(\cV)$ on $\natural$-finite objects. Obviously, $\Ch^{\natural}(\cV)$ is a subcategory of $\Ch^+(\cV^{\gf})$; in fact, $M \in \Ch^{\natural}(\cV)$ if and only if both $M$ and $M^{\natural}$ belong to $\Ch^+(\cV^{\gf})$. This shows that $\Ch^{\natural}(\cV)$ is stable under $\natural$. One easily sees that it is also stable under $\otimes$. It follows that $(-)^{\natural}$ induces a contravariant auto-equivalence of symmetric tensor categories on $\Ch^{\natural}(\cV)$. We note that $\bK$, $A$, and $B$ all belong to $\Ch^{\natural}(\cV)$, as does any finitely generated dg-$A$-module or any finitely cogenerated dg-$B$-comodule.

\subsection{\texorpdfstring{Interplay between $\natural$-dual and Koszul duality}{Interplay between \#-dual and Koszul duality}} \label{ss:natural-koszul}

A simple computation shows that $(\bC^{\infty})^{\natural}=\bC^{\infty}[-1]$. Since $(-)^{\natural}$ is a symmetric tensor functor, it follows that $A=B^{\natural}$ as dg-algebras (maintaining the notation from the second half of \S \ref{ss:hilbsyz}). Thus $(-)^{\natural}$ gives an equivalence between the categories of $\natural$-finite dg-$A$-modules and $\natural$-finite dg-$B$-comodules. The following result shows how this identification interacts with Koszul duality:

\begin{proposition}
\label{prop:natkosz}
Suppose $M$ is a $\natural$-finite dg-$A$-module. Then there is a natural isomorphism $\cK(M)^{\natural}=\cK'(M^{\natural})$ of dg-$A$-modules.
\end{proposition}

We begin with a lemma.

\begin{lemma}
\label{lem:kosz}
There is an isomorphism $\bK = \bK^{\natural}$ respecting all structure.
\end{lemma}

\begin{proof}
For clarity in this proof, write $V=\bC^{\infty}$ and $U=\bC^{\infty}[-1]$. Note that there are isomorphisms $f \colon V^{\natural} \to U$ and $g \colon U^{\natural} \to V$; we take $g=f^{\natural}$, so that $f=g^{\natural}$ as well. Since $(-)^{\natural}$ is a symmetric tensor functor and $\bK = \Sym(V \oplus U)$ (ignoring the differential), it follows that there is an isomorphism $\phi \colon \bK \to \bK^{\natural}$ respecting all structure except perhaps the differential. (In fact, up to scalars, $\phi$ is the unique isomorphism which is $\GL$-equivariant, $A$-linear, $B$-colinear, and satisfies $\phi^{\natural}=\phi$.) But notice that the differential on $\bK$ is the unique one which is $A$-linear, $B$-colinear, and induces the natural isomorphism $\Sym^1(U) \to \Sym^1(V)$. This final property is clearly respected by $\phi$, and so $\phi$ is compatible with the differential.
\end{proof}

\begin{proof}[Proof of Proposition~\ref{prop:natkosz}]
If $M$ and $N$ are $\natural$-finite dg-$A$-modules then there is a natural isomorphism $(M \otimes_A N)^{\natural}=M^{\natural} \otimes^B N^{\natural}$ (use that the $A$-module structure on $M \otimes N$ turns into the $B$-comodule structure on $M^\natural \otimes N^\natural$ under the natural identification $(M \otimes N)^\natural = M^\natural \otimes N^\natural$). We therefore have natural isomorphisms
\begin{displaymath}
\cK(M)^{\natural}=(K \otimes_A M)^{\natural} = \bK \otimes^B M^{\natural}=\cK'(M^{\natural}),
\end{displaymath}
where we used the isomorphism of Lemma~\ref{lem:kosz}.
\end{proof}

\subsection{The Fourier transform} \label{ss:fourier}

Let $M$ be a $\natural$-finite dg-$A$-module. We define the {\bf Fourier transform} of $M$, denoted $\cF(M)$, to be the $\natural$-finite dg-$A$-module $\cF(M)=\cK(M)^{\natural}$.

\begin{theorem}
The Fourier transform is a contravariant auto-equivalence of the bounded derived category of finitely generated $A$-modules:
\[
\cF \colon \rD^b(\Mod_A)^{\op} \xrightarrow{\simeq} \rD^b(\Mod_A).
\]
Furthermore, there is a natural isomorphism $\cF^2=\id$.
\end{theorem}

\begin{proof}
Unraveling definitions, we find that $\rH_i(\cF(M))=L_i(\cK(M))^{\vee,\dag}$ as $A$-modules. (Note: $L_i(\cK(M))$ is a $B'$-comodule, and $(B')^{\vee,\dag}=A$ as algebras). Thus if $M$ belongs to the bounded derived category of finitely generated $A$-modules then so does $\cF(M)$, by Theorem~\ref{thm:hilbsyz}. The isomorphism $\cF^2=\id$ comes from Proposition~\ref{prop:natkosz} and Proposition~\ref{prop:kosz}.
\end{proof}

The following basic computations are important enough that we state them formally here:

\begin{proposition} \label{prop:fourier-basic}
Let $\lambda$ be a partition with $\vert \lambda \vert=r$. We then have $\cF(\bS_{\lambda}[k])=A \otimes \bS_{\lambda^{\dag}}[r+k]$ and $\cF(A \otimes \bS_{\lambda}[k])=\bS_{\lambda^{\dag}}[r+k]$.
\end{proposition}

\begin{proof}
We have already seen in the proof of Theorem~\ref{thm:hilbsyz} that as dg-$B$-comodules, we have $\cK(\bS_\lambda) = B \otimes \bS_\lambda$ and $\cK(A \otimes \bS_\lambda) = \bS_\lambda$. The functor $\cK$ commutes with the translation functor $(-)[k]$. A short computation gives $\bS_{\lambda}[k]^{\natural}=\bS_{\lambda^{\dag}}[r+k]$, from which the result follows.
\end{proof}

Define $\cF_k$ to be the $k$th homology of $\cF$. If $M$ is a module (as opposed to a complex), then
\begin{displaymath}
\cF_k(M)=\bigoplus_{i \ge 0} \Tor^i_A(M, \bC)_{i+k}^{\dagger,\vee}
\end{displaymath}
as an object of $\cV$. In particular, $\cF_k(M)=0$ for $k<0$. If
\begin{displaymath}
0 \to M_1 \to M_2 \to M_3 \to 0
\end{displaymath}
is a short exact sequence of modules, then there is a long exact sequence
\begin{displaymath}
\cdots \to \cF_1(M_3) \to \cF_1(M_2) \to \cF_1(M_1) \to \cF_0(M_3) \to \cF_0(M_2) \to \cF_0(M_1) \to 0
\end{displaymath}

\begin{remark} \label{rmk:Gexact}
The above sequence shows that $\cF_0 \colon \Mod_A^\op \to \Mod_A$ is a right-exact functor. The $\cF_k$ are \emph{not} the derived functors of $\cF_0$. Indeed, it is not difficult to see that $\cF_k$ can fail to vanish on torsion injective $A$-modules for $k>0$ (which are projective in $\Mod_A^\op$).
\end{remark}

\subsection{Poincar\'e series} \label{ss:poincare}

Given an $A$-module $M$, define its {\bf Poincar\'e series} by
\begin{displaymath}
P_M(t, q)=\sum_{n \ge 0} (-q)^n H_{\Tor^A_n(M, \bC)}(t).
\end{displaymath}
By setting $q=1$ in the Poincar\'e series and multiplying by $H_A(t)=e^t$, one recovers the Hilbert series.  Note that the Poincar\'e series of $M$ has nontrivial information about the $A$-module structure on $M$, whereas the Hilbert series of $M$ is only defined in terms of the underlying object of $\cV$.  The Poincar\'e series does \emph{not} factor through K-theory, so it cannot be studied directly via Proposition~\ref{prop:ind}.  Nonetheless, as a corollary of the results in the preceding section, we have:

\begin{theorem}
The Poincar\'e series is of the form $f(t, q)+g(t, q)e^{-tq}$ where $f$ and $g$ are polynomials in $t$ and Laurent polynomials in $q$.
\end{theorem}

\begin{proof}
A simple manipulation gives $P_M(t,q)=\sum_{k \ge 0} (-q)^{-k} H_{\cF_k(M)}(-qt)$. Since each $\cF_k(M)$ is a finitely generated $A$-module, its Hilbert series is of the form $a(t)+b(t)e^t$ for polynomials $a$ and $b$.  As the sum is finite, the result follows.
\end{proof}

\begin{remark}
The above proof applies equally well to the ``enhanced Poincar\'e series.''
\end{remark}

\begin{remark}  
Although $P_M$ only has nonnegative powers of $q$ in its power series expansion, one may need negative powers to express $P_M$ in terms of elementary functions. See \eqref{eqn:exampleP02} for a specific example. 
\end{remark}

\subsection{The duality theorem and reciprocity} \label{ss:hilbertreciprocity}

Let $M$ be an object of $\rD^b(A)$, the bounded derived category of finitely generated $A$-modules. We have a canonical distinguished triangle
\begin{equation}
\label{eq:decomp}
M_t \to M \to M_f \to
\end{equation}
with $M_t \in \Tors$ and $M_f \in \Perf$ (see Theorem~\ref{thm:pt-decomp}). We now show that the Fourier transform interchanges the perfect and torsion pieces in \eqref{eq:decomp}.

\begin{theorem} \label{thm:perf-tors}
We have natural isomorphisms $\cF(M_t)=\cF(M)_f$ and $\cF(M_f)=\cF(M)_t$.
\end{theorem}

Before proving this theorem, we require a lemma:

\begin{lemma} \label{lem:fourier-pt}
The Fourier transform of a perfect complex is a torsion complex, and conversely, the Fourier transform of a torsion complex is a perfect complex.
\end{lemma}

\begin{proof}
For the purposes of this proof, define the {\bf size} of a bounded torsion complex to be the sum of the lengths of its terms. Suppose $M \in \Tors$. We can find an exact triangle $M_1 \to M \to M_2 \to$ where $M_1$ is a torsion complex of smaller size and $M_2$ is a torsion complex of size 1, i.e., $M_2=\bS_{\lambda}[n]$ for some $\lambda$ and $n$. Applying $\cF$, we obtain a triangle $\cF(M_2) \to \cF(M) \to \cF(M_1) \to$. By Proposition~\ref{prop:fourier-basic}, $\cF(M_2)$ is perfect. We can assume $\cF(M_1)$ is perfect by induction on size. It follows that $\cF(M)$ is perfect as well. Thus $\cF$ takes torsion complexes to perfect complexes. The reverse property follows from a similar argument.
\end{proof}

\begin{proof}[Proof of Theorem~\ref{thm:perf-tors}]
The decomposition \eqref{eq:decomp} for $\cF(M)$ is
\begin{displaymath}
\cF(M)_t \to \cF(M) \to \cF(M)_f \to
\end{displaymath}
Applying $\cF$ to the decomposition \eqref{eq:decomp} for $M$, we obtain a triangle
\begin{displaymath}
\cF(M_f) \to \cF(M) \to \cF(M_t) \to
\end{displaymath}
Furthermore, $\cF(M_f)$ is torsion and $\cF(M_t)$ is perfect, by Lemma~\ref{lem:fourier-pt}. By Theorem~\ref{thm:pt-decomp}, we obtain the stated isomorphism.
\end{proof}

The following proposition describes how the Fourier transform interacts with enhanced Hilbert series. It is a direct corollary of Proposition~\ref{prop:hilbertlocal} and Theorem~\ref{thm:perf-tors}, but we prefer to give a direct, elementary proof.

\begin{proposition}
\label{prop:fourier-hilbert}
Let $M$ be an $A$-module.  Then $\tilde{H}_M(t)=\exp(T_0) \tilde{H}_{\cF(M)}(-t)$.  In particular, writing $\tilde{H}_M=p_M \exp(T_0)+q_M$, we have $p_{\cF(M)}(t)=q_M(-t)$ and $q_{\cF(M)}(t)=p_M(-t)$.
\end{proposition}

\begin{proof}
For notational convenience, put $[M]=\tilde{H}_M$.  The terms of the minimal free resolution of $M$ are given by $A \otimes \Tor^A_p(M, \bC)$, and so we have
\begin{displaymath}
[M]=\sum_{p \ge 0} (-1)^p [A] [\Tor^A_p(M, \bC)].
\end{displaymath}
We thus have
\begin{displaymath}
[M]=[A] \sum_{n \ge 0} (-1)^n [\cT_n(M)]^-=[A] \sum_{n \ge 0} (-1)^n [\cF_n(M)^\dag]^{-},
\end{displaymath}
where $[-]^-$ is the map $t^{\lambda} \mapsto (-1)^{\vert \lambda \vert} t^{\lambda}$.  Note that this is the same as $t_i \mapsto (-1)^i t_i$.  If $V=(V_n)_{n \ge 0}$ is an object of $\cV$, thought of as a sequence of representations of symmetric groups, then $V^{\dag}=(V_n \otimes \sgn)_{n \ge 0}$.  We have $\trace(c_{\lambda} | V_n \otimes \sgn)=(-1)^k \trace(c_{\lambda} | V_n)$, where $k$ is the number of even cycles in the cycle decomposition of $\lambda$.  It follows that $[V^{\dag}]$ is gotten from $[V]$ by changing $t_i$ to $(-1)^{i+1} t_i$. Thus $[-^\dag]^{-}$ is the operation $t_i \mapsto -t_i$, which completes the proof.
\end{proof}

\subsection{The Fourier transform on $\rD^b(K)$}
\label{ss:modKfourier}

It follows from Lemma~\ref{lem:fourier-pt} that $\cF$ induces an equivalence of categories $\Perf^{\op} \to \Tors$. On the other hand, we know that $\Perf$ is equivalent to $\rD^b(K)$ while $\Tors$ is equivalent to $\rD^b_{\tors}(A)$, and that $\Mod_K$ and $\Mod_A^{\tors}$ are equivalent.  It follows that we obtain an equivalence
\begin{displaymath}
\cF_K \colon \rD^b(K)^{\op} \to \rD^b(K), \qquad
\cF_K(M)=\Psi(\cF(\rR S(M))),
\end{displaymath}
where $\Psi$ is as in \S \ref{ss:modKmodAequiv}.  One can show that if $\lambda$ is a partition of size $n$ then
\begin{displaymath}
\cF_K(Q_{\lambda})=L_{\lambda^{\dag}}[-n], \qquad \cF_K(L_{\lambda})=Q_{\lambda^{\dag}}[-n]
\end{displaymath}
and $\cF_K^2=\id$.  These formulas give an explanation for the symmetries in $\rK$-theory observed in Remark~\ref{rmk:modKpairing}.

\subsection{Example: EFW complexes} \label{ss:exampleEFW}

The resolutions discussed in this section are taken from \cite[\S 3]{efw}. They also appeared earlier in \cite[Theorem 8.11]{olver}. Pick a partition $\alpha$ and a positive integer $e$. From Pieri's formula, there is a unique, up to scalar multiple, non-zero map
\[
A \otimes \bS_{(\alpha_1 + e, \alpha_2, \dots)} \to A \otimes \bS_\alpha
\]
(in $(\alpha_1 + e, \dots)$, we are only adding $e$ to the first part of $\alpha$). Let $M(\alpha, e)$ be the cokernel of this map. By Proposition~\ref{prop:pierisubmod}, this is a finite length $A$-module and its free resolution was constructed in \cite[\S 3]{efw}. In this section, we will show that the regularity of $M(\alpha,e)$ is $|\alpha| + e-1 + \alpha_1$ and we will calculate its Poincar\'e series (modulo finitely many error terms).

First, we review the free resolution of $M(\alpha,e)$. For $i \ge 1$, define partitions $\alpha(i)$ by
\[
\alpha(i)_j = \begin{cases} \alpha_1 + e & j=1\\ \alpha_{j-1} + 1 & 1 < j \le i\\ \alpha_j & j > i\end{cases}
\]
and set $\alpha(0) = \alpha$. By Pieri's formula, we have non-zero maps
\[
d_i \colon A \otimes \bS_{\alpha(i)} \to A \otimes \bS_{\alpha(i-1)}
\]
for all $i \ge 1$, which are unique up to scalar multiple, and furthermore, $d_{i-1}d_i = 0$. Set $\bF_i = A \otimes \bS_{\alpha(i)}$. As a consequence of \cite[\S 3]{efw} (alternatively, using Proposition~\ref{prop:pierisubmod}), the complex $\bF_\bullet$ is exact in positive degrees, and is a resolution of $M(\alpha,e) = \rH_0(\bF_\bullet)$. In particular, the module $M(\alpha,e)$ is generated in degree $|\alpha|$ and $M(\alpha,e)$ has regularity $|\alpha| + e-1 + \alpha_1$.

Before we calculate the Poincar\'e series of $M(\alpha,e)$, we do a preliminary calculation. Let $\mu$ be a partition and let $(1^N)$ be a sequence of $N$ 1's. The hook length formula \cite[Corollary 7.21.6]{stanley} shows that the dimension of $\bM_{\mu + (1^N)}$ is
\[
d_\mu(N) = \frac{(N+|\mu|)!}{\prod_{i=1}^{\ell(\mu)} (N+\mu_i-i+1) \cdot (N- \ell(\mu))! \cdot \prod_{b \in \mu} {\rm hook}(b)},
\]
which is a polynomial in $N$ of degree $|\mu|$.

Let $\beta = \alpha(\ell(\alpha)+1) - (1^{\ell(\alpha)+1})$ (this is the partition obtained by subtracting $1$ from each part of $\alpha(\ell(\alpha)+1)$). For $N > \ell(\alpha)$, we have $\alpha(N) = \beta + (1^N)$. Hence modulo the first $\ell(\alpha)$ terms, the Poincar\'e series of $M(\alpha, e)$ is
\[
P_{M(\alpha,e)} \sim (-q)^{-|\beta|} \sum_{N > \ell(\alpha)} (-qt)^{N+|\beta|} \frac{d_\beta(N)} {(N+|\beta|)!}.
\]

In particular, consider $\alpha = \emptyset$ and $e=2$, so that $M(\emptyset, 2)$ is the quotient of $A$ by the square of its maximal ideal. Then 
\begin{align} \label{eqn:exampleP02}
P_{M(\emptyset, 2)} = 1 -q^{-1}\sum_{n \ge 2} \frac{(-qt)^{n}}{(n)!}(n-1) = 
(1-q^{-1}) + (t + q^{-1})e^{-qt}.
\end{align}
For general $e\ge 2$, $M(\emptyset, e)$ is the quotient of $A$ by the $e$th power of its maximal ideal, and
\[
P_{M(\emptyset, e)} = 1 + \frac{(-q^{-1})^{e-1}}{(e-1)!} 
\sum_{n \ge e} \frac{(-qt)^n}{n!} (n-1)(n-2) \cdots (n-e+1)
\]
Recall that the generating function of $x^d e^x$ is $\sum_{n \ge d} \frac{x^n}{n!} (n)_d$ where $(n)_d = n(n-1) \cdots (n-d+1)$ is the falling factorial. Hence if we want to simplify an expression of the form $\sum_n \frac{x^n}{n!} p(n)$ for some polynomial $n$, we have to convert $p$ into the falling factorial basis.

In our situation, we have $p(n) = (n-1)_{e-1}$, which satisfies the identity 
\[
(n-1)_{e-1} = (n)_{e-1} - (e-1)(n-1)_{e-2}.
\]
Expanding this, we can express it as
\[
(n-1)_{e-1} = \sum_{i=0}^{e-1} (-1)^i (e-1)_i (n)_{e-1-i}.
\]
So 
\begin{align*}
P_{M(\emptyset, e)} &= f(t,q) + \frac{(-q^{-1})^{e-1}}{(e-1)!} \left(\sum_{i=0}^{e-1} (-1)^i (e-1)_i (-tq)^{e-1-i}\right) e^{-qt}\\
&= f(t,q) + \left(\sum_{i=0}^{e-1} \frac{t^{e-1-i} q^{-i}}{(e-1-i)!} \right) e^{-qt}
\end{align*}
for some Laurent polynomial $f(t,q)$ with $f(t,1) = 0$. 


\section{Depth and local cohomology} \label{sec:depth}

In this section, we study an important homological invariant of modules: depth. We start in \S\ref{ss:depth} by giving a definition of depth. In \S\ref{ss:localcohomology}, we prove an analogue of Grothendieck's vanishing theorem for local cohomology.  In \S\ref{ss:comparison}, we make a comparison with a possible definition of local cohomology for modules over $\Sym(\bC^\infty)$ which might not have a $\GL_\infty$-equivariant structure and show that this definition agrees with the one we give for $\GL_\infty$-equivariant modules. Finally, in \S\ref{ss:localchar}, we calculate the local cohomology of the modules $L_\lambda^{\ge D}$ from \S\ref{ss:simp-inj}, and show that they lift well-known formulas in $\rK$-theory for character polynomials.

\subsection{Depth} \label{ss:depth}

Given an $A$-module $M$, let $d_M(n)$ be the depth of $M(\bC^n)$ with respect to the homogeneous maximal ideal $A(\bC^n)_{>0}$. By the Auslander--Buchsbaum formula \cite[Theorem 19.9]{eisenbud}, we have 
\begin{align} \label{eqn:AB}
\pdim_{A(\bC^n)} M(\bC^n) + d_M(n) = n.
\end{align}

\begin{theorem} \label{thm:depth}
Let $M \in \Mod_A$ be non-projective.  Then $d_M(n)$ is a weakly decreasing function. In particular, the limit $\lim_{n \to \infty} d_M(n)$ exists, i.e., $d_M(n)$ is independent of $n$ for $n \gg 0$.
\end{theorem}

We call this limit the {\bf depth} of $M$. We say that non-zero projective modules have infinite depth. 

\begin{proof} 
Let $\bK_\bullet$ be the Koszul complex, i.e., $\bK_i = A \otimes \bigwedge^i(\bC^{\infty})$. Then $\Tor_i^A(\bC, M) = \rH_i(\bK_\bullet \otimes_A M)$ is a Schur functor and its specialization to $n$ variables gives $\Tor^{A(\bC^n)}_i(\bC, M(\bC^n))$. This shows that $n \mapsto \pdim_{A(\bC^n)} M(\bC^n)$ is a weakly increasing function. By the Auslander--Buchsbaum formula \eqref{eqn:AB}, it suffices to show that $n \mapsto \pdim_{A(\bC^n)} M(\bC^n)$ is a strictly increasing function. 

So suppose that $\pdim_{A(\bC^n)} M(\bC^n) = \pdim_{A(\bC^{n+1})} M(\bC^{n+1})$ for some $n$. Let $L$ be this common value. Let $\bF_{\bullet}$ be the minimal projective resolution of $M$ as an $A$-module. Then $\bF_{\bullet}(\bC^k)$ is the minimal projective resolution of $M(\bC^k)$ as an $A(\bC^k)$-module, for any $k$. (This can be seen since minimal is equivalent to $d(\bF_i) \subset \fm \bF_{i-1}$, which is preserved when evaluation on $\bC^k$.) Since $\bF_{L+1}(\bC^{n+1}) = 0$, we see that $\bF_{L+1}$ is a sum of $A \otimes \bS_\lambda$ where $\ell(\lambda)>n+1$. Write $\bF_L = F \oplus F'$ where $F$ (resp.\ $F'$) is a free $A$-module generated by $\bS_{\mu}$'s with $\ell(\mu) \le n$ (resp.\ $\ell(\mu)>n$). By the assumption that $\pdim_{A(\bC^n)} M(\bC^n) = L$, we have $F \ne 0$. By Pieri's rule, every partition appearing in $F$ has at most $n+1$ rows. Since every partition appearing in $\bF_{L+1}$ has more than $n+1$ rows, it follows that there is no non-zero map $\bF_{L+1} \to F$, and hence the restriction of the differential $\bF_L \to \bF_{L-1}$ to $F$ is injective. Since $F$ is injective (Corollary~\ref{cor:projinj}), it follows that $\bF_{L-1}$ splits as $F \oplus F''$ for some free $A$-module $F''$. But this contradicts the minimality of $\bF_{\bullet}$ (recall that $F \ne 0$).
\end{proof}

\begin{lemma}
\label{lem:depth}
Let
\begin{displaymath}
0 \to M \to F \to N \to 0
\end{displaymath}
be an exact sequence of $A$-modules with $F$ projective.  Then $\depth(M)=\depth(N)+1$.
\end{lemma}

\begin{proof}
If $\depth(N)=\infty$ then $N$ is projective, and so $M$ is projective and $\depth(M)=\infty$ as well.  Assume now that $\depth(N)$ is finite.  We have an exact sequence
\begin{displaymath}
0 \to M(\bC^n) \to F(\bC^n) \to N(\bC^n) \to 0
\end{displaymath}
for each $n$.  We thus see that $d_M(n)=d_N(n)+1$ if $n>\depth(N)$ by the Auslander--Buchsbaum formula \eqref{eqn:AB} applied to a mapping cone construction for a minimal projective resolution of $N$ in terms of one for $M$.  The result follows.
\end{proof}

\subsection{Vanishing of local cohomology} \label{ss:localcohomology}

The following result determines where local cohomology vanishes in general.

\begin{proposition} \label{prop:lcextrema}
Let $M \in \Mod_A$ be non-zero. Then $\inf(\{d \mid \rH^d_\fm(M) \ne 0\})$ is the depth of $M$ (where we use the convention $\inf(\emptyset)=\infty$), and $\sup(\{d-1 \mid \rH^d_\fm(M) \ne 0\})$ is the injective dimension of the localization $T(M) \in \Mod_K$ (where we use the convention $\sup(\emptyset)=0$ and that the injective dimension of $0$ is $-1$.). 
\end{proposition}

This is an analogue of Grothendieck's vanishing theorem \cite[Proposition A1.16]{syzygies}, but we have swapped ``dimension'' with a different invariant. In between these two extrema, we show in Proposition~\ref{prop:simplelocal} that the pattern of which local cohomology groups is non-zero can be anything. This can be viewed as an analogue of the corresponding fact for local cohomology of local rings (see \cite[Theorem A]{ep}).

\begin{proof}
Let $M$ be a non-projective $A$-module.  Say that $M$ satisfies property $(A_n)$ if $\depth(M)=n$ and property $(B_n)$ if $\rH^n_{\fm}(M) \ne 0$ but $\rH^i_{\fm}(M)=0$ for $i<n$.  We show that $(A_n)$ and $(B_n)$ are equivalent, by induction on $n$.

We first consider the base case $n=0$.  If $\rH^0_{\fm}(M) \ne 0$ then $M$ has torsion, and so $M(\bC^n)$ has torsion for $n \ge \ell(M)$, which implies that $M(\bC^n)$ has depth $0$ for $n \gg 0$, and so $\depth(M)=0$.  Thus $(B_0)$ implies $(A_0)$.  Conversely, suppose that $\depth(M)=0$.  Then $M(\bC^n)$ has depth $0$ for $n \ge N$ for some $N$.  Let $n$ be greater than $N$ and $\ell(M)+1$.  As $M(\bC^n)$ has depth $0$, we can choose a non-zero $\GL_n$-stable subspace $V_0 \subset M(\bC^n)$ which is annihilated by $\fm(\bC^n)$.  We have $V_0=V(\bC^n)$ for some unique subobject $V \subset M$ in $\cV$.  By hypothesis, the map $\fm \otimes V \to M$ is zero when evaluated on $\bC^n$.  Since both spaces have at most $\ell(M)+1$ rows, the map is identically zero, and so $V$ is annihilated by $\fm$.  Thus $V$ defines a non-zero subspace of $\rH^0_{\fm}(M)$.  Therefore $(A_0)$ implies $(B_0)$.

Suppose now that $M$ satisfies either $(A_n)$ or $(B_n)$, with $n>0$.  Then $M$ is torsion-free (this is clear if $(B_n)$ holds; if $(A_n)$ holds then $(A_0)$ does not hold, and so $(B_0)$ does not hold, and so $\rH^0_{\fm}(M)$ is zero).  It follows that $M$ injects into its saturation $S(T(M))$.  We can choose an injection $T(M) \to I$ in $\Mod_K$, for some injective object $I$.  Let $F=S(I)$, a projective object of $\Mod_A$.  Since $S$ is left exact, $M$ injects into $F$.  We thus have an exact sequence
\begin{displaymath}
0 \to M \to F \to N \to 0.
\end{displaymath}
By Lemma~\ref{lem:depth}, $M$ satisfies $(A_n)$ if and only if $N$ satisfies $(A_{n-1})$. By considering the long exact sequence in local cohomology, we also see that $M$ satisfies $(B_n)$ if and only if $N$ satisfies $(B_{n-1})$. Hence we finish by induction.

For the second part, the trivial cases when $M$ is torsion or of the form $A \otimes \bS_\lambda$ are clear from the definitions and Proposition~\ref{prop:modKinj}. Otherwise, let $\bI^\bullet$ be a minimal injective resolution of $T(M)$ of length $n$. If $S(\bI^{n-1}) \to S(\bI^n)$ were surjective then it would be split, since $S$ carries injectives to projectives; applying $T$ we would find that $\bI^{n-1} \to \bI^n$ is split, contradicting minimality.  Thus $S(\bI^{n-1}) \to S(\bI^n)$ is not surjective, and using Corollary~\ref{prop:localsheaf}, we get $\rR^nS(T(M)) \cong \rH^{n+1}_\fm(M) \ne 0$. Finally, since $S(\bI^k) = 0$ for $k>n$, we conclude that $\rH^k_\fm(M) = 0$ for $k>n+1$.
\end{proof}

\begin{remark}[Cosyzygies] \label{rmk:cosyzygies}
Let $M$ be a finitely generated $A$-module which has positive finite depth $d$.  As we saw in the above proof, we can choose an injection $M \to F$ with $F$ projective, and we have $\depth(F/M)=d-1$.  Iterating this process, we obtain a long exact sequence
\[
0 \to M \to F_d \to F_{d-1} \to \cdots \to F_0 \to M' \to 0
\]
where each $F_i$ is projective and $\depth(M') = 0$ (i.e., $M'$ has torsion).
\end{remark}

\subsection{Comparison with non-equivariant modules} \label{ss:comparison}

Let $A_0$ denote the graded ring $A(\bC^{\infty})=\bC[x_1,x_2,\ldots]$, but regarded without any $\GL_{\infty}$-action.  We write $\fm_0$ for its maximal ideal.  Let $N$ be an $A_0$-module.  We define the local cohomology of $N$ by
\begin{displaymath}
\rH^i_{\fm_0}(N)= \varinjlim_d \ext^i_{A_0}(A_0/\fm_0^d, N).
\end{displaymath}
We define the depth of $N$ to be the supremum of lengths of regular sequences on $N$.  These are the usual definitions of these concepts, but they are not often applied in the setting of non-noetherian rings.  Given an $A$-module $M$, we obtain an $A_0$-module $M_0$ by simply forgetting the extra structure.

The purpose of this section is to compare the depth and local cohomology of an $A$-module with that of its underlying $A_0$-module.

\begin{proposition}
\label{prop:depthinf}
For an $A$-module $M$ we have $\depth(M)=\depth(M_0)$.
\end{proposition}

We first require a lemma:

\begin{lemma}
\label{lem:linseq}
Let $M$ be a $\GL_n$-equivariant module over $\Sym(\bC^n)$.  The following statements are equivalent:
\begin{enumerate}[\rm (a)]
\item There exists a regular sequence of length $d$ on $M$.
\item There exists $d$ linearly independent elements of $\bC^n$ which form a regular sequence on $M$.
\item Every sequence of $d$ linearly independent elements of $\bC^n$ forms a regular sequence on $M$.
\end{enumerate}
\end{lemma}

\begin{proof}
Assume $d>0$ otherwise there is nothing to prove. The set of zerodivisors of $M$ is a union of finitely many prime ideals, none of which is the maximal homogeneous ideal because we have a non-zerodivisor. Hence none of them contains the space of linear forms, and in particular, their union cannot contain the space of linear forms since $\bC$ is an infinite field. By induction on $d$, we see that (a) implies (b). Since every sequence of $d$ linearly independent elements of $\bC^n$ is conjugate under $\GL_n$, we see that (b) implies (c). Finally, (c) trivially implies (a).  
\end{proof}

\begin{proof}[Proof of Proposition~\ref{prop:depthinf}]
We first show that $\depth(M_0) \le \depth(M)$.  Thus let $(y_1, \ldots, y_d)$ be a regular sequence on $M_0$ with $d=\depth(M_0)$.  Then each $y_i$ belongs to $A(\bC^n)$ for $n$ sufficiently large, and $(y_1, \ldots, y_d)$ forms a regular sequence on $M(\bC^n)$.  This shows that $d \le d_M(n)$ for $n \gg 0$, and so $\depth(M_0) \le \depth(M)$.

We now show that $\depth(M) \le \depth(M_0)$.  Let $d=\depth(M)$.  Let $y = (y_1, \ldots, y_d)$ be a sequence of linearly independent elements of $\bC^{\infty}$.  By Lemma~\ref{lem:linseq}, $y$ forms a regular sequence on all $M(\bC^n)$ for $n \gg 0$.  We claim that $y$ forms a regular sequence on $M_0$.  Indeed, suppose that $y_im = 0$ in $M_0/(y_1, \ldots, y_{i-1}) M_0$. We have that $y_i$ belongs to $A(\bC^n)$ and $m$ belongs to $M(\bC^n)$ for $n \gg 0$, and since $y$ is a regular sequence on $M(\bC^n)$ for $n \gg 0$, we find $m=0$.  This proves the claim, and so $\depth(M) \le \depth(M_0)$.
\end{proof}

\begin{proposition}
\label{prop:infcoh}
Given an $A$-module $M$, there is a natural isomorphism
\begin{displaymath}
\rH^i_{\fm}(M)=\rH^i_{\fm_0}(M_0)
\end{displaymath}
compatible with the $\GL_{\infty}$-action on each side.
\end{proposition}

For an analogous statement about local cohomology for local rings, see \cite[Proposition A1.1]{syzygies}.  

\begin{proof}
Put $G^i(M)=\rH^i_{\fm_0}(M_0)$.  Then $G^{\bullet}$ defines a $\delta$-functor on $\Mod_A$ and $G^0$ coincides with $\rH^0_{\fm}$.  To prove the proposition, it suffices to show that $G^i(I)=0$ for $i>0$ whenever $I$ is an indecomposable injective object of $\Mod_A$.  There are two kinds of indecomposable injectives of $\Mod_A$: the injective envelopes of simple modules and modules of the form $A \otimes \bS_{\lambda}$.  Since $G^{\bullet}$ is just the usual local cohomology of the ring $A$, the higher functors vanish on all torsion modules.  This takes care of the first sort of indecomposable injective.  We now observe that we have $G^i(A \otimes \bS_{\lambda})=\bS_{\lambda} \otimes G^i(A)$, since the $\GL_{\infty}$-structure does not affect the computation of $G^i$.  Therefore, it is enough to show that $G^i(A)=0$ for $i>0$.

This vanishing follows from the fact that $A_0$ has infinite depth.  Since our situation is somewhat exotic, we give a proof.  Let $A_n$ be the quotient of $A_0$ by the ideal $(x_1, \ldots, x_n)$.  We have an exact sequence
\begin{displaymath}
0 \to A_n \xrightarrow{x_{n+1}} A_n \to A_{n+1} \to 0
\end{displaymath}
and an associated long exact sequence after applying $\rH_{\fm_0}^{\bullet}$.  This shows that if $\rH^i_{\fm_0}(A_n)=0$ for $i<k$ but $\rH^k_{\fm_0}(A_n) \ne 0$ then $\rH_{\fm_0}^{k-1}(A_{n+1}) \ne 0$, if $k>0$.  (Note:  local cohomology is always torsion, and so multiplication by an element of the maximal ideal on a non-zero local cohomology group is never injective.)  In particular, if $\rH_{\fm_0}^k(A_0) \ne 0$ then $\rH_{\fm_0}^0(A_k) \ne 0$.  However, none of the $A_k$ have any torsion, and so $\rH_{\fm_0}^0(A_k)=0$ for all $k$.  Thus $G^i(A)=\rH^i_{\fm_0}(A_0)=0$ for all $i$.
\end{proof}

\begin{remark}
Propositions \ref{prop:lcextrema}, \ref{prop:depthinf}, and~\ref{prop:infcoh} show that if $N$ is a module over $A_0$ then $\rH^i_{\fm_0}(N)=0$ for $i<\depth(N)$ and $\rH^i_{\fm_0}(N) \ne 0$ for $i=\depth(N)$, provided $N$ can be endowed with a $\GL_{\infty}$-equivariant structure.
\end{remark}

\begin{remark}
There are a few reasons we defined local cohomology as the derived functors of $\rH^0_{\fm}$ instead of first passing to $A_0$ and then using $\rH_{\fm_0}^i$. First, it does not seem to be clear that $\rH^i_{\fm_0}(M_0)$ belongs to $\cV$, i.e., that the $\GL_{\infty}$-action on it is polynomial: on the one hand, an equivariant injective resolution of $M_0$ as an $A_0$-module will involve non-polynomial modules; on the other, while one can equivariantly resolve $A_0/\fm_0^d$ using polynomial modules, applying $\Hom_{A_0}(-, M_0)$ destroys this property.  Second, it is not clear that $\rH^i_{\fm_0}(M_0)=0$ for $i \gg 0$.  And finally, it is not clear that $M \mapsto \rH^i_{\fm_0}(M_0)$ is the derived functor of $M \mapsto \rH^0_{\fm_0}(M_0)$.
\end{remark}

\begin{remark}  
The functors $\rH^i_{\fm}$ and $\varinjlim \Ext^i_A(A/\fm^d, -)$ are not isomorphic:  they even differ for $i=0$. For example, $\hom_A(A/\fm^d, \bC^\infty) = 0$ for all $d$ because of the $\GL$-equivariance, but $\rH^0_\fm(\bC^\infty) = \bC^\infty$.
\end{remark}

\begin{remark}
Consider the category $\cA$ of nonnegatively graded $\bC[t]$-modules.  One can define $\rH^0_{\fm}$ on this category to be the maximal torsion submodule, as we did above.  Its derived functors exist.  However, the analogue of Proposition~\ref{prop:infcoh} is false in this setting.  Indeed, $\bC[t]$ is an injective object of $\cA$, and so $\rH^1_{\fm}(\bC[t])=0$, while the usual theory of local cohomology shows that $\varinjlim \Ext^1_{\bC[t]}(\bC[t]/\fm^n, \bC[t])$ is non-zero.  This problem disappears if one considers all graded $\bC[t]$-modules. 
\end{remark}

\subsection{Local cohomology and character polynomials} \label{ss:localchar}

Consider the module $L_\lambda^0 = \bigoplus_{d \ge \lambda_1} \bS_{(d,\lambda_1)}$ from \S\ref{ss:simp-inj} and its enhanced Hilbert series
\[
\tilde{H}_{L_\lambda^0}(t) = p_{L_\lambda^0}(t) \exp(T_0) + q_{L_\lambda^0}(t).
\]
We discussed the character polynomial $X^\lambda$ in \S\ref{ss:charpoly}. In particular, we have that 
\[
X^\lambda(\mu) := X^\lambda(a_1, \dots, a_n) = \trace(c_\mu|\bM_{(N-|\lambda|,\lambda)})
\]
for any partition $\mu \vdash N$ with $m_i(\mu) = a_i$ and $N \gg 0$. By the determinantal expression for $\trace(c_\mu|\bM_{(N-|\lambda|,\lambda)})$ discussed in \cite[Example I.7.14]{macdonald}, we see that it in fact holds for $N \ge \lambda_1 + |\lambda|$.

Actually, this determinantal expression implies more. Namely, it gives an interpretation for $X^\lambda(\mu)$ when $|\mu| < \lambda_1 + |\lambda|$. Note that the coefficient of $t^\mu$ in $\tilde{H}_{L_\lambda^0}$ vanishes for such $\mu$, so this value of the character polynomial is related to the coefficient of $t^\mu$ in $q_{L_\lambda^0}$.

First, define the sequence $\rho = (-1, -2, -3, \dots)$. Given a permutation $w \in S_\infty$ and a sequence $\alpha = (\alpha_1, \alpha_2, \dots)$ with finitely many non-zero terms, define $w \bullet \alpha = w(\alpha + \rho) - \rho$ where the right-hand side is the usual permutation action on sequences. There are two mutually exclusive cases:
\begin{enumerate}[(a)]
\item \label{enum:singularwt} There is a non-identity permutation $w$ such that $w \bullet \alpha = \alpha$ (in which case we say that $\alpha$ is {\bf singular}), or 
\item \label{enum:regularwt} There is a unique permutation $w$ such that $w \bullet \alpha$ is weakly decreasing (in which case we write $w \bullet \alpha \ge 0$).
\end{enumerate}

Now we go back to character polynomials. Set $\alpha = (|\mu| - |\lambda|, \lambda)$. This is a sequence of integers. For instance, if $\lambda = (3,2,2)$ and $|\mu| = 5$, then $\alpha = (-2,3,2,2)$. In case \eqref{enum:singularwt}, $X^\lambda(\mu) = 0$. In case \eqref{enum:regularwt}, let $w$ be the unique permutation such that $w \bullet \alpha$ is weakly decreasing (and hence is a partition). Then 
\begin{align} \label{eqn:modification}
X^\lambda(\mu) = (-1)^{\ell(w)} \trace(c_\mu | \bM_{w \bullet \alpha})
\end{align}
where $\ell(w) = \#\{i<j \mid w(i) > w(j)\}$ is the number of inversions of $w$. In this case, we will write 
\[
\alpha \xrightarrow{\ell(w)} w \bullet \alpha.
\]
This gives us a formula for the Euler characteristic of $\rH^\bullet_\fm(L_\lambda^0)$ by Proposition~\ref{prop:hilbertlocal}. Since $L_\lambda^0$ is saturated by Proposition~\ref{prop:Llamsat}, we always have $\rH^0_\fm(L_\lambda^0) = \rH^1_\fm(L_\lambda^0) = 0$, so by Corollary~\ref{prop:localsheaf}, this is the Euler characteristic of $S$ applied to the minimal injective resolution of Theorem~\ref{thm:BGGresolution}. We will refine this statement in Proposition~\ref{prop:simplelocal}.

Since $X_M$ is in general a linear combination of the $X^\lambda$, similar remarks apply to any finitely generated module $M$ in place of $L_\lambda^0$.

Now we give a pictorial description of \eqref{eqn:modification}. First, given a partition $\lambda$, consider its Young diagram. (The following terminology is not standard, but we will only need it in the formulation of Proposition~\ref{prop:simplelocal}.) A {\bf border strip} $B$ of its Young diagram is a set of boxes which do not contain a $2 \times 2$ square, and such that no boxes of $\lambda$ in the complement of $B$ lie strictly below $B$, and such that $\lambda \setminus B$ is the Young diagram of a partition. Equivalently, this is a set of boxes which can be obtained by first removing a vertical strip from $\lambda$, and then removing a horizontal strip from the result. For example, for $\lambda = (7,5,3,3,2)$, we have marked a border strip with 2 connected components with $\times$:
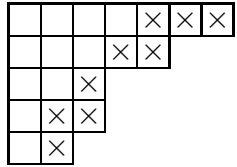
\[
\begin{ytableau}
\ & \ & \ & \ & \times & \times & \times \\
\ & \ & \ & \times & \times \\
\ & \ & \times \\
\ & \times & \times \\
\ & \times 
\end{ytableau}
\]
We say that a border strip is {\bf aligned} if it contains the last box in the first row of $\lambda$ and if it is connected. The rightmost 5 boxes in the above example form an aligned border strip. Aligned border strips are determined by their size. We denote the set of aligned border strips of $\lambda$ by $\BS(\lambda)$. The {\bf height} of a border strip, denoted ${\rm ht}(B)$, is the number of rows that it occupies. Then the $w \bullet \alpha$ in \eqref{eqn:modification} becomes the partition obtained by removing the aligned border strip of size $\lambda_1 - |\mu| + |\lambda|$ from $(\lambda_1, \lambda)$ (if possible) and $\ell(w)$ becomes the height of this border strip. If there is no aligned border strip of the specified size, then the right-hand side of \eqref{eqn:modification} is $0$.

To see the equivalence of the rules, start with the sequence $\alpha = (N - |\lambda|, \lambda)$ which we visualize as a modified Young diagram (if $N-|\lambda| < 0$, we can think about the first row going to the left). To sort the sequence $\alpha + \rho$, we only have to move the first entry to the $k$th entry for some $k$, so the permutation that we use is the cycle $(1,2,\dots,k)$. We can factor this cycle as $s_{k-1} \cdots s_2 s_1$ where $s_i$ is the transposition that swaps positions $i$ and $i+1$. In the process of calculating $s_{k-1} \bullet ( \cdots ( s_2 \bullet (s_1 \bullet \alpha) )\cdots )$, we will move the first entry past $k-1$ rows, add $k-1$ to it, and subtract $1$ from each of these rows that we moved past. All we have done is removed a border strip from $(\lambda_1, \lambda)$ and this gives us a bijection between modified sequences and removing border strips.

\begin{example}
We now work out the combinatorics above in a few cases of $\lambda$ and its relation to Hilbert series and local cohomology. The case $\lambda = (1)$ is treated in detail. The other cases are similar, so we omit the details for them.

\begin{enumerate}[(a)]
\item The simplest case above is $\lambda = (1)$. We consider the sequences $(N-1,1)$ for $N \ge 0$, and the discussion above interprets $X^1(\mu)$ (up to sign) as the trace of $c_\mu$ on a representation of the $N$th symmetric group.

When $N \ge 2$, $(N-1,1)$ is a partition, so $X^1(\mu) = \trace(c_\mu \vert \bM_{(N-1,1)})$ for all partitions $\mu$ of size $N$.
When $N=1$, we get the sequence $(0,1)$, which is singular, so $X^1(\mu) = 0$ for $\mu = (1)$.
When $N=0$, we get the sequence $(-1,1)$. If $s_1$ is the transposition that swaps the first and second indices, we get $s_1 \bullet (-1,1) = (0,0)$. So we write $(-1,1) \xrightarrow{1} (0,0)$, and we get $X^1(\mu) = -\trace(c_\mu \vert \bM_0)$ for $\mu = \emptyset$.

Applying $S$ to the minimal injective resolution of $L_1$, we get
\[
A \otimes \bS_1 \to A
\]
whose kernel is $L_1^0$ and whose cokernel is $\bC$. Hence $\rH^2_\fm(L_1^0) = \bC$. According to Proposition~\ref{prop:hilbertlocal}, this contributes to $q_{L_1^0}(t)$ and we see directly that it accounts for the difference between the enhanced Hilbert series $\tilde{H}_{L_1^0}(t)$ and the character polynomial $p_1^0(t) \exp(T_0)$.

\item Now consider the case $\lambda = (2)$. Then $(N-2,2)$ is a partition whenever $N \ge 4$ and for $N=0,1,2,3$, we get, respectively, that $(-2,2)$ is singular, $(-1,2) \xrightarrow{1} (1,0)$, $(0,2) \xrightarrow{1} (1,1)$, and $(1,2)$ is singular. Applying $S$ to the minimal injective resolution of $L_2$, we get
\[
A \otimes \bS_2 \to A \otimes \bS_1
\]
whose kernel is $L_2^0$ and whose cokernel is $\rH^2_\fm(L_2^0) = \bS_1 \oplus \bS_{1,1}$ (which could be deduced from Proposition~\ref{prop:pierisubmod}).  Note that we are only treating $\rH^2_{\fm}(L_2^0)$ as an object of $\cV$ here.

\item For a slightly more involved example, consider $\lambda = (2,1)$. Then $(N-3,2,1)$ is a partition whenever $N \ge 5$, and this sequence is singular for $N=0,2,4$. Otherwise, we have $(-2,2,1) \xrightarrow{2} (1)$ and $(0,2,1) \xrightarrow{1} (1,1,1)$. The corresponding border strips are
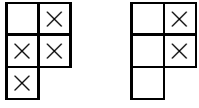
\begin{align*}
  \begin{ytableau}
    \ & \times \\
\times & \times \\
\times
  \end{ytableau}
\qquad 
  \begin{ytableau}
    \ & \times \\
\ & \times \\
\ 
  \end{ytableau}
\end{align*}

Applying $S$ to the minimal injective resolution of $L_{2,1}$, we get
\[
A \otimes \bS_{2,1} \to A \otimes (\bS_{1,1} \oplus \bS_2) \to A \otimes \bS_1.
\]
Again from Proposition~\ref{prop:pierisubmod}, we deduce that the cokernel of the last map is $\rH^3_\fm(L_{2,1}^0) = \bS_1$. Since the kernel of the first map is $L_{2,1}^0$, we can take the Euler characteristic (either by hand using Pieri's formula, or using the discussion in this section) to get that the middle homology is $\rH^2_\fm(L_{2,1}^0) = \bS_{1,1,1}$. \qedhere
\end{enumerate}
\end{example}

Here is the general situation for $L_\lambda^{\ge D}$, which upgrades the K-theory formula \eqref{eqn:modification}.

\begin{proposition} \label{prop:simplelocal}
For $i \ge 1$, the local cohomology of $L_\lambda^{\ge D}$ is given by
\[
\rH^i_{\fm}(L_\lambda^{\ge D}) = \bigoplus_{\substack{B \in \BS((D, \lambda)),\\
{\rm ht}(B) = i}} \bS_{(D, \lambda) \setminus B}
\]
as an object of $\cV$. As an $A$-module, $\rH^i_\fm(L_\lambda^{\ge D})$ is generated by the smallest partition in $\{(D,\lambda) \setminus B\}$ with ${\rm ht}(B)=i$.
\end{proposition}

\begin{proof}
Since the saturation of $L_\lambda^{\ge D}$ is $L_\lambda^0$, the formula is correct for $i=1$ by Corollary~\ref{prop:localsheaf}. Hence we may assume that $D = \lambda_1$ without loss of generality. To check it for $i>1$, we use the injective resolution of $L_\lambda$ in $\Mod_K$ given in Theorem~\ref{thm:BGGresolution}, apply the section functor $S$ to it, and calculate cohomology. By Proposition~\ref{prop:freesat}, we get a complex $\bI^\bullet$ where $\bI^j = \bigoplus_{\mu,\, \lambda/\mu \in \VS_j} A \otimes \bS_\mu$.

First, we consider how to count the occurrences of $L_\nu^{\ge D}$ in $\bI^j$. The only such $\nu$ that appear are obtained by first removing a vertical strip from $\lambda$ to get some partition $\mu$, and then removing a horizontal strip from $\mu$. In other words, the $\nu$ are those partitions we can get by removing a border strip $B$ from $\lambda$. In this case, we will have $D = \mu_1$, and $L_\nu^{\ge \mu_1}$ appears as a constituent of $\bI^{|\lambda|-|\mu|}$. Given a connected border strip, there are exactly two ways to get it from removing a vertical strip and then a horizontal strip: the box that lies in the top-right can be removed at either step, but the order in which the other boxes are removed is forced.

Now consider the border strip $B$ removed from $\lambda$ to get $\nu$. If $B$ does not contain the top-right box of $\lambda$, then the description above shows that all instances of $L_\nu^{\ge D}$ have $D = \lambda_1$, and the proof of Theorem~\ref{thm:BGGresolution} shows that the subcomplex (in the category $\cV$) consisting of them is exact. 

So we assume that $B$ contains the top-right box of $\lambda$ now. Let $r$ be the number of connected components of $B$, and order them from right to left. Then there are $2^r$ ways to obtain $B$ by first removing a vertical strip and then removing a horizontal strip. We encode these with a bitstring $\ul{a} = (a_1, \dots, a_r)$ with $a_i \in \{0,1\}$, where $a_i = 1$ if and only if we include the top-right box of the $i$th connected component in the vertical strip. Let $f$ be the number of boxes which are forced to be included in the vertical strip. Then the bitstring $\ul{a}$ contributes a copy of $L_\nu^{\ge \lambda_1 - a_1}$ to $\bI_{f + a_1 + \cdots + a_r}$. In particular, if $r \ge 2$, we see that the number of copies of $L_\nu^{\ge \lambda_1 - 1}$ and $L_\nu^{\ge \lambda_1}$ are evenly distributed, and in fact, that the subcomplex given by all of them is isomorphic to the $2$-term complex $L_\nu^{\ge \lambda_1} \to L_\nu^{\ge \lambda_1 - 1}$ tensored with $r-1$ copies of an exact complex $\bC \xrightarrow{\cong} \bC$, and hence this subcomplex is exact.

Finally, we deal with the case $r=1$, which corresponds in a natural way to removing an aligned border strip from $(\lambda_1, \lambda)$. In this case, the subcomplex given by the $L_\nu^{\ge D}$ looks like
\[
L_\nu^{\ge \lambda_1} \to L_\nu^{\ge \lambda_1 - 1}
\]
in cohomological degrees $f$ and $f+1$, respectively. Since these are the unique ways to get $L_\nu$, there are no other terms in cohomological degrees $f,f+1$ that interfere with this map, i.e., the cokernel of this map contributes to $\rR^{f+1}S(L_\lambda^0) \cong \rH^{f+2}_\fm(L_\lambda^0)$. The cokernel of this map is $\bS_{(\lambda_1-1, \nu)}$, which is exactly obtained from $(\lambda_1, \lambda)$ by removing the unique aligned border strip of size $f+2$.

Hence we have calculated the local cohomology as elements of $\cV$. The last statement on the $A$-module structure follows from the fact that $\rH^i_\fm(L_\lambda^0)$ is a subquotient of $A \otimes \bS_\eta$ where $\eta$ is the partition obtained from $\lambda$ by removing the $f$ forced vertical strip boxes mentioned above, plus the last box in the first row of $\lambda$.
\end{proof}

\begin{corollary} \label{cor:depthsimple}
The depth of $L_\lambda^{\ge D}$ is the multiplicity of $D$ in the partition $(D,\lambda)$. In particular, if $D>\lambda_1$, the depth is $1$, and otherwise, the depth of $L_\lambda^0$ is $1$ more than the multiplicity of $\lambda_1$ in $\lambda$.
\end{corollary}

\begin{proof}
This follows directly from Propositions~\ref{prop:lcextrema} and \ref{prop:simplelocal}.
\end{proof}

\begin{remark} \label{rmk:modKpairing2}
Now we have introduced enough language to simplify the calculation of $\langle Q_\lambda, L_\mu \rangle$ from Remark~\ref{rmk:modKpairing}. Recall that we defined $\langle M, N \rangle = \sum_i (-1)^i \dim \Ext^i_K(M,N)$ for $M,N \in \Mod_K$. We have
\begin{displaymath}
\langle Q_{\lambda}, L_{\mu} \rangle=\langle T(A \otimes \bS_{\lambda}), L_{\mu} \rangle
=\langle A \otimes \bS_{\lambda}, \rR S(L_{\mu}) \rangle
\end{displaymath}
by adjunction. Since $A \otimes \bS_\lambda$ is projective and $\rR S L_\mu$ is essentially the local cohomology of $L_\mu^0$ (Corollary~\ref{prop:localsheaf}), we see that this pairing is $m_0 - \sum_{i>0} (-1)^i m_i$ where $m_0$ denotes the multiplicity of $\lambda$ in $L_\mu^0$ (which is either $0$ or $1$) and $m_i$ for $i>0$ denotes the multiplicity of $\lambda$ in $\rH^i_\fm(L_\mu^0)$. By Proposition~\ref{prop:simplelocal}, we have $\sum_{i \ge 0} m_i \le 1$, so in particular, $\langle Q_\lambda, L_\mu \rangle \in \{-1,0,1\}$. The exact value can be calculated using the combinatorics of border strips (or using the shifted symmetric group action from above).
\end{remark}

\begin{remark} \label{rmk:stabdeg}
For an $A$-module $M$, let $d_i(M)$ be the maximum size of a partition appearing in $\rH^i_{\fm}(M)$.  Then by Proposition~\ref{prop:hilbertlocal}, we have
\begin{displaymath}
\deg(q_M) \le \max(d_i(M)).
\end{displaymath}
For $i \ge 2$, the quantity $\rH^i_{\fm}(M)$ depends only on the image $T(M)$ of $M$ in $\Mod_K$.  Let $\Lambda(M)$ be the set of partitions $\lambda$ for which $L_{\lambda}$ is a constituent of $T(M)$.  By filtering $T(M)$ and looking at various long exact sequences, we see that
\begin{displaymath}
d_i(M) \le \max_{\lambda \in \Lambda(M)} d_i(L_{\lambda}^0)
\end{displaymath}
for $i \ge 2$.  For a non-zero partition $\lambda$, let $d(\lambda)=\vert \lambda \vert+\lambda_1-n-1$, where $n$ is the multiplicity of $\lambda_1$ in $\lambda$; put $d(\emptyset)=0$.  Then Proposition~\ref{prop:simplelocal} shows that $d_i(L_{\lambda}^0) \le d(\lambda)$.  Combining all of this, we see that
\begin{displaymath}
\deg(q_M) \le \max \left( d_0(M), d_1(M), \max_{\lambda \in \Lambda(M)} d(\lambda) \right).
\end{displaymath}
Recall that the character polynomial of $M$ computes the character of the representation $M_n$ for $n>N$, for some integer $N$.  It is desirable to know the optimal value of $N$.  In fact, the optimal value of $N$ is exactly equal to the degree of $q$, and so the above inequality gives an explicit bound on the optimal value of $N$.  This bound is an improvement on the one given in \cite[Thm.~2.67]{fimodules}:  note that $d_0(M)$ and $d_1(M)$ are both $\le$ the stability degree (in the sense of \cite[Definition 2.34]{fimodules}), while $\max_{\lambda \in \Lambda(M)} d(\lambda)$ is $\le$ the weight \cite[Definition 2.50]{fimodules}, and this inequality is often strict.
\end{remark}

\begin{example} \label{eg:depthAB}
We will say a little bit about calculating the depth of the modules $L_\lambda^{\ge D}$ using the Auslander--Buchsbaum formula. This depth was obtained using local cohomology in Corollary~\ref{cor:depthsimple}. Given an $A$-module $M$, let $\Omega M$ be the syzygy module of $M$, i.e., the kernel of a minimal surjection $P \to M \to 0$ with $P$ projective. We extend this notation by setting $\Omega^k M = \Omega(\Omega^{k-1} M)$ for $k > 1$.

We will continue to use the notation of \S\ref{ss:exampleEFW}, so $M(\alpha, e)$ is the cokernel of the map $A \otimes \bS_\beta \to A \otimes \bS_\alpha$ where $\beta_1 = \alpha_1 + e$ and $\beta_i = \alpha_i$ for $i\ge 2$.

Let $\{\mu^1, \dots, \mu^r\}$ be the partitions obtained by adding a single box to $(D,\lambda)$ anywhere except the first row. Then the presentation of $L_\lambda^{\ge D}$ as an $A$-module is
\[
\bigoplus_{i=1}^r A \otimes \bS_{\mu^i} \to A \otimes \bS_{(D,\lambda)} \to L_\lambda^{\ge D} \to 0.
\]
If $\lambda = \emptyset$, then $L_\emptyset^{\ge D} = \Omega M(\emptyset,D)$ is the $D$th power of the maximal ideal of $A$, and was discussed in \S\ref{ss:exampleEFW}. The minimal free resolution in general was constructed in \cite[Corollary 2.10]{sw}. 

We will consider the case when $\lambda = (n)$ has a single part. If $D = n$, then 
\[
L_{(n)}^{\ge n} = \Omega^2 M((n-1,n-1), 1),
\]
so again, was discussed in \S\ref{ss:exampleEFW}. So we assume that $D > n$. The basic idea of \cite[\S 2.3]{sw} is to use mapping cones. In our situation, we have the presentation
\[
A \otimes (\bS_{(D,n+1)} \oplus \bS_{(D,n,1)}) \to A \otimes \bS_{(D,n)} \to L_{(n)}^{\ge D} \to 0.
\]
Define $A$-modules $N,N'$ using the presentations
\begin{align*}
A \otimes \bS_{(D,n,1)} \to A \otimes \bS_{(D,n)} \to N \to 0,\\
A \otimes \bS_{(D,n+1,1)} \to A \otimes \bS_{(D,n+1)} \to N' \to 0.
\end{align*}
Then $N = \Omega^2 M(n-1,D-n+1)$ and $N' = \Omega^2 M(n,D-n)$, so both have depth $2$. Also, we have a short exact sequence
\[
0 \to N' \to N \to L_{(n)}^{\ge D} \to 0.
\]
We can construct a free resolution for $L_{(n)}^{\ge D}$ by taking a mapping cone on the minimal free resolutions of $N'$ and $N$. All of the partitions that appear will be distinct, so in fact it will be minimal. In particular, the minimal free resolution $\bF_\bullet$ of $L_{(n)}^{\ge D}$ has the terms
\[
\bF_i = A \otimes (\bS_{(D,n,1^i)} \oplus \bS_{(D,n+1,1^{i-1})}) \qquad (i>0),
\]
so we see, by the Auslander--Buchsbaum formula, that $\depth L_{(n)}^{\ge D} = 1$ when $D > n$. This is in contrast with the fact that $\depth L_{(n)}^{\ge n} = 2$.

In general, \cite[Corollary 2.10]{sw} implies that $\depth L_\lambda^{\ge D}$ is the number of times that $D$ occurs in the partition $(D,\lambda)$, which agrees with Corollary~\ref{cor:depthsimple}. Again, the idea is to relate $L_\lambda^{\ge D}$ to simpler modules, constructed like $N$ and $N'$ above, and then to use a mapping cone construction. In general, $N$ and $N'$ are not of the form $\Omega^j M(\mu, e)$. 
\end{example}

\addtocontents{toc}{\bigskip}

\end{document}